\documentclass[leqno,11pt,letter]{amsart}

\usepackage{tikz-cd}
\usepackage{comment}
\usepackage{soul}
\usepackage{todonotes}
\usepackage{amsmath}
\usepackage{mathtools,amsfonts,stmaryrd}
\usepackage{enumitem}

\usepackage[full]{textcomp}
\usepackage{newpxtext}
\usepackage{cabin} 
\usepackage[varqu,varl]{inconsolata} 
\usepackage[bigdelims,vvarbb]{newpxmath} 
\usepackage[cal=cm,bb=esstix,bbscaled=1.05]{mathalfa} 
\usepackage[margin=2.4cm]{geometry}
\usepackage{savesym}

\let\bigtimes\relax
\usepackage{mathabx}
\usepackage{color}

\usepackage{hyperref}
\hypersetup{
 colorlinks,
 linkcolor={red!80!black},
 citecolor={blue!80!black},
 urlcolor={blue!80!black}
}

\usepackage{graphicx}
\usepackage[all,cmtip]{xy}
\usepackage{mleftright}
\usepackage{booktabs}
\usepackage{cite}
\usepackage{enumitem}
\usepackage{mathdots}
\usepackage{pagecolor}
\usepackage[export]{adjustbox}

\usetikzlibrary{arrows.meta}

\newcommand{\blackarrowright}{\mathrel{\tikz[baseline=-.6ex]{%
\draw[->,>=Latex,black,line width=0.5pt] (0,0) -- (3ex,0);}}}

\newcommand{\xblackrightarrow}[2][]{\mathrel{\tikz[baseline=-.6ex]{%
\draw[->,>=Latex,black,line width=0.5pt] (0,0) -- node[above] {#1} (4ex,0);}}#2}
\newcommand{\xblackleftarrow}[2][]{\mathrel{#2\tikz[baseline=-.6ex]{%
\draw[<-,>=Latex,black,line width=0.5pt] (0,0) -- node[above] {#1} (4ex,0);}}}

\usepackage{tikz}
\usetikzlibrary{backgrounds}
\usetikzlibrary{arrows}
\usetikzlibrary{shapes,shapes.geometric,shapes.misc}

\tikzstyle{tikzfig}=[baseline=-0.25em,scale=0.5]

\pgfkeys{/tikz/tikzit fill/.initial=0}
\pgfkeys{/tikz/tikzit draw/.initial=0}
\pgfkeys{/tikz/tikzit shape/.initial=0}
\pgfkeys{/tikz/tikzit category/.initial=0}

\pgfdeclarelayer{edgelayer}
\pgfdeclarelayer{nodelayer}
\pgfsetlayers{background,edgelayer,nodelayer,main}

\tikzstyle{none}=[inner sep=0mm]

\newcommand{\tikzfig}[1]{%
{\tikzstyle{every picture}=[tikzfig]
\IfFileExists{#1.tikz}
  {\input{#1.tikz}}
  {%
    \IfFileExists{./figures/#1.tikz}
      {\input{./figures/#1.tikz}}
      {\tikz[baseline=-0.5em]{\node[draw=red,font=\color{red},fill=red!10!white] {\textit{#1}};}}%
  }}%
}

\tikzstyle{every loop}=[]

\tikzstyle{basic_node}=[fill=none, draw=none, shape=circle, tikzit fill={rgb,255: red,160; green,160; blue,160}, tikzit draw=black]

\tikzstyle{arrow}=[fill=none, draw=black, tikzit draw=black, tikzit fill=none, ->]

\setlist[enumerate]{labelsep=*, leftmargin=1.5pc}
\setlist[enumerate]{label=\normalfont(\roman*), ref=\roman*}
\newtheorem{thm}{Theorem}[section]

\newtheorem{thm*}[intro]{Theorem}
\newtheorem{lemma}[thm]{Lemma}
\newtheorem{lemma*}[intro]{Lemma}
\newtheorem{cor}[thm]{Corollary}
\newtheorem{cor*}[intro]{Corollary}
\newtheorem{prop}[thm]{Proposition}
\newtheorem{prop*}[intro]{Proposition}

\newtheorem{claim*}[]{Claim}

\newtheorem{conjecture*}[intro]{Conjecture}
\theoremstyle{definition}
\newtheorem{example}[thm]{Example}
\newtheorem{example*}[intro]{Example}
\newtheorem{remark}[thm]{Remark}
\newtheorem{remark*}[intro]{Remark}
\newtheorem{definition}[thm]{Definition}
\newtheorem{definition*}[intro]{Definition}

\newtheorem{notation}[thm]{Notation}

\newtheorem{problem*}[intro]{Problem}
\newtheorem{algorithm}[thm]{Algorithm}
\newtheorem{question*}[intro]{Question}
\newtheorem{setup}[thm]{Setup}
\numberwithin{equation}{section}

\newcommand{\Z}{\mathbb{Z}}
\newcommand{\Q}{\mathbb{Q}}
\newcommand{\R}{\mathbb{R}}
\newcommand{\C}{\mathbb{C}}
\newcommand{\N}{\mathbb{N}}

\newcommand{\op}{\operatorname}
\newcommand{\ol}{\overline}

\graphicspath{{images/}}

\begin{document}
\title[Combed Trisection Diagrams and Non-Semisimple 4-Manifold Invariants]{Combed Trisection Diagrams and \\ Non-Semisimple 4-Manifold Invariants}

\author{Julian Chaidez}
\address{Department of Mathematics\\University of Southern California\\Los Angeles, CA\\90007\\USA}
\email{julian.chaidez@usc.edu}

\author{Jordan Cotler}
\address{Society of Fellows\\Harvard University\\Cambridge, MA\\02138\\USA}
\email{jcotler@fas.harvard.edu}

\author{Shawn X. Cui}
\address{Department of Mathematics, Department of Physics and Astronomy\\Purdue University\\West Lafayette, IN\\47907\\USA}
\email{cui177@purdue.edu}

\begin{abstract} Given a triple $H$ of (possibly non-semisimple) Hopf algebras equipped with pairings satisfying a set of properties, we describe a construction of an associated smooth, scalar invariant $\tau_H(X,\pi)$
of a simply connected, compact, oriented $4$-manifold $X$ and an open book $\pi$ on its boundary. This invariant generalizes an earlier semisimple version and is calculated using a trisection diagram $T$ for $X$ and a certain type of combing of the trisection surface. We explain a general calculation of this invariant for a family of exotic 4-manifolds with boundary called Stein nuclei, introduced by Yasui. After investigating many low-dimensional Hopf algebras up to dimension 11, we have not been able to find non-semisimple Hopf triples that satisfy the criteria for our invariant.  Nonetheless, appropriate Hopf triples may exist outside the scope of our explorations.
\end{abstract}

\maketitle

\def \myHeader {}

\tableofcontents

\section{Introduction} 
\label{sec:introduction}

Topological quantum field theories (or TQFTs) have played a major role in low-dimensional topology since their inception \cite{witten1988topological,witten1989quantum}. Physically, a TQFT is a quantum field theory where the fields and action do not depend on the metric of space-time manifolds. Mathematically, a TQFT can be described by a monoidal functor from the category of cobordisms of a certain dimension to the category of vector spaces. In particular, TQFTs provide quantum invariants for smooth manifolds, and as such are important tools in topology.  

\vspace{3pt}

In dimension three, the Turaev-Viro-Barrett-Westbury (TVBW) \cite{turaev1992state,barrett1996invariants} and Witten-Reshetikhin-Turaev (WRT) \cite{reshetikhin1991invariants, turaev1994quantum} constructions provide two fundamental families of TQFTs, with the former serving as far-reaching generalizations of the well-known Jones polynomial of links. The quantum invariants  are sensitive to properties of 3-manifolds beyond classical homotopic/homological information. For example, the invariants can distinguish certain homotopy equivalent but non-diffeomorphic 3-manifolds \cite{freed1991computer}.

\vspace{3pt}

In dimension four, TQFTs have not been understood as well as their lower-dimensional counterparts. A broad family of state-sum quantum invariants/TQFTs are due to Douglas-Reutter \cite{douglas2018fusion} based on semisimple monoidal 2-categories, which generalize previous constructions by the third author \cite{cui2019four}, Crane-Yetter \cite{crane1993categorical, crane1997state}, Yetter \cite{yetter1993tqft}, and Dijkgraaf-Witten \cite{dijkgraaf1990topological}. Most currently known TQFTs are either proved or expected to be semisimple.   On the other hand, a no-go theorem \cite{reutter2022semisimple} points to the subtlety that semisimple TQFTs cannot distinguish smooth structures of 4-manifolds. See \cite{tata2023anomalies} for a construction of TQFTs defined for unoriented manifolds endowed with Pin structures which distinguish the real and fake $\mathbb{R}\mathbb{P}^4$. For the rest of the paper, we only consider TQFTs/invariants on orientable manifolds.  A central problem in dimension four is to define effective computable invariants  sensitive to smooth structures. Given the no-go theorem, one has to consider non-semisimple TQFTs or more generally 4-manifold invariants which do not extend to  semisimple TQFTs.

\vspace{3pt}

There have been numerous studies of non-semisimple WRT-type TQFTs in dimension three (based on certain non-semisimple tensor categories) where they have been shown to be more fruitful than their semisimple counterparts (e.g.~\cite{costantino2014quantum, de2020nonsemisimple, de20223}). From a dual perspective, there also exist two invariants of 3-manifolds -- the Kuperberg invariant \cite{kuperberg1996noninvolutory} and the Hennings invariant \cite{hennings1996invariants, kauffman1995invariants} -- constructed from finite-dimensional Hopf algebras. When the Hopf algebra is semisimple, the Kuperberg and Hennings invariants reduce, respectively, to the invariants of TVBW and WRT. Hence, the former two can be thought of as non-semisimple generalizations of the latter two.

\vspace{3pt}

Let us focus on the Kuperberg invariant. Given a Hopf algebra, the invariant is defined for a closed oriented 3-manifold endowed with a framing or a nowhere-vanishing vector field. Notably, for non-semisimple Hopf algebras, the Kuperberg invariant does not extend to a TQFT. Moreover, it allows for various generalizations using super-Hopf algebras \cite{neumann2021kuperberg} or Hopf objects in symmetric tensor categories  \cite{kashaev2019generalized}, and in some setups it provides an invariant of Spin$^c$ structures which recovers the Reidemeister torsion and the Seiberg-Witten invariant. These results show that the Kuperberg invariant, as a non-semisimple generalization of TVBW, contains rich information about 3-manifolds. Since the Douglas-Reutter invariant is a 4-dimensional analog of TVBW, to obtain non-semisimple 4-manifold invariants, it is natural to consider a 4-dimensional analog of the Kuperberg invariant.

\subsection{Trisection invariants} In \cite{chaidez20194manifold} we initiated the program of constructing Kuperberg-type invariants of 4-manifolds. These invariants are defined and computed using trisection diagrams  \cite{gk2016}, which are 4-dimensional analogues of Heegaard splittings. Trisections have been extensively studied \cite{abrams2018group, castro2018trisections, lambert2021symplectic, joseph2022bridge,agk2018grouptrisections,co2017lefschetztrisections} and have become a very useful tool for understanding 4-dimensional topology. For example, using trisections of $\mathbb{CP}^2$ \cite{lambert2020bridge}, Lambert-Cole re-proved the famous Thom Conjecture originally settled by Seiberg-Witten gauge theory \cite{kronheimer1994genus}. 

\begin{figure}[h!]
\centering
\includegraphics[width=.3\textwidth]{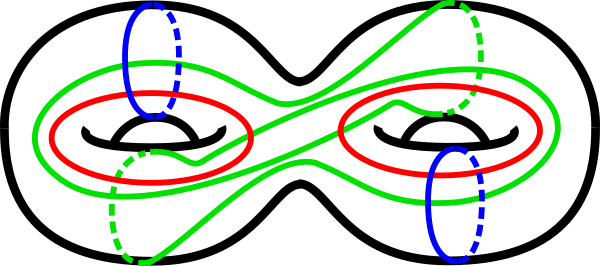} 
\caption{A genus $2$ trisection for the closed manifold $S^2 \times S^2$.}\label{fig:intro_trisection}
\end{figure} 

In our work~\cite{chaidez20194manifold}, the algebraic structure utilized for the construction of the 4-manifold invariant is a \emph{Hopf triplet} which roughly consists of three Hopf algebras and a bilinear form on each pair subject to compatibility conditions. Although Hopf triplets can be defined in the non-semisimple setting, the work ~\cite{chaidez20194manifold} restricts to the case of semisimple Hopf algebras. The resulting invariant, called the trisection invariant of 4-manifolds, includes a large class of dichromatic invariants \cite{barenz2016dichromatic}, Crane-Yetter invariants, and conjecturally the Kashaev invariant \cite{kashaev2014asimple}. Note that~\cite{chaidez20194manifold} only contains a construction of a scalar invariant, and does not construct a full semisimple TQFT.

\vspace{3pt}

The goal of this paper is to generalize the construction in \cite{chaidez20194manifold} to 4-manifolds with boundary and to certain non-semisimple Hopf triplets. We are inspired by the analogous non-semisimple generalization of Kuperberg's invariant \cite{kuperberg1996noninvolutory} to the case of balanced Hopf algebras. We now discuss the topological and algebraic input required to define and compute these invariants.

\subsection{Topological input} The topological input to our generalized trisection invariant is a pair
\[
(X,\pi)
\]
consisting of a smooth $4$-manifold $X$ with non-empty, connected boundary equipped with a marked open book $\pi$ on the boundary. Recall that an open book $(B,\pi)$ on a closed oriented connected $3$-manifold $Y$ consists of a closed embedded $1$-manifold $B \subset Y$ called the binding and a fibration $\pi$ of $Y \setminus B$ over $S^1$ that is trivial near $B$. Open books are a common tool in 3-manifold theory and contact topology (cf. \cite{giroux2006stable,winkelnkemper1973manifolds,geiges2008introduction}). A \emph{marked} open book is equipped with a distinguished, \emph{marked} component $C \subset B$.

\begin{figure}[h!]
\centering
\includegraphics[width=.6\textwidth]{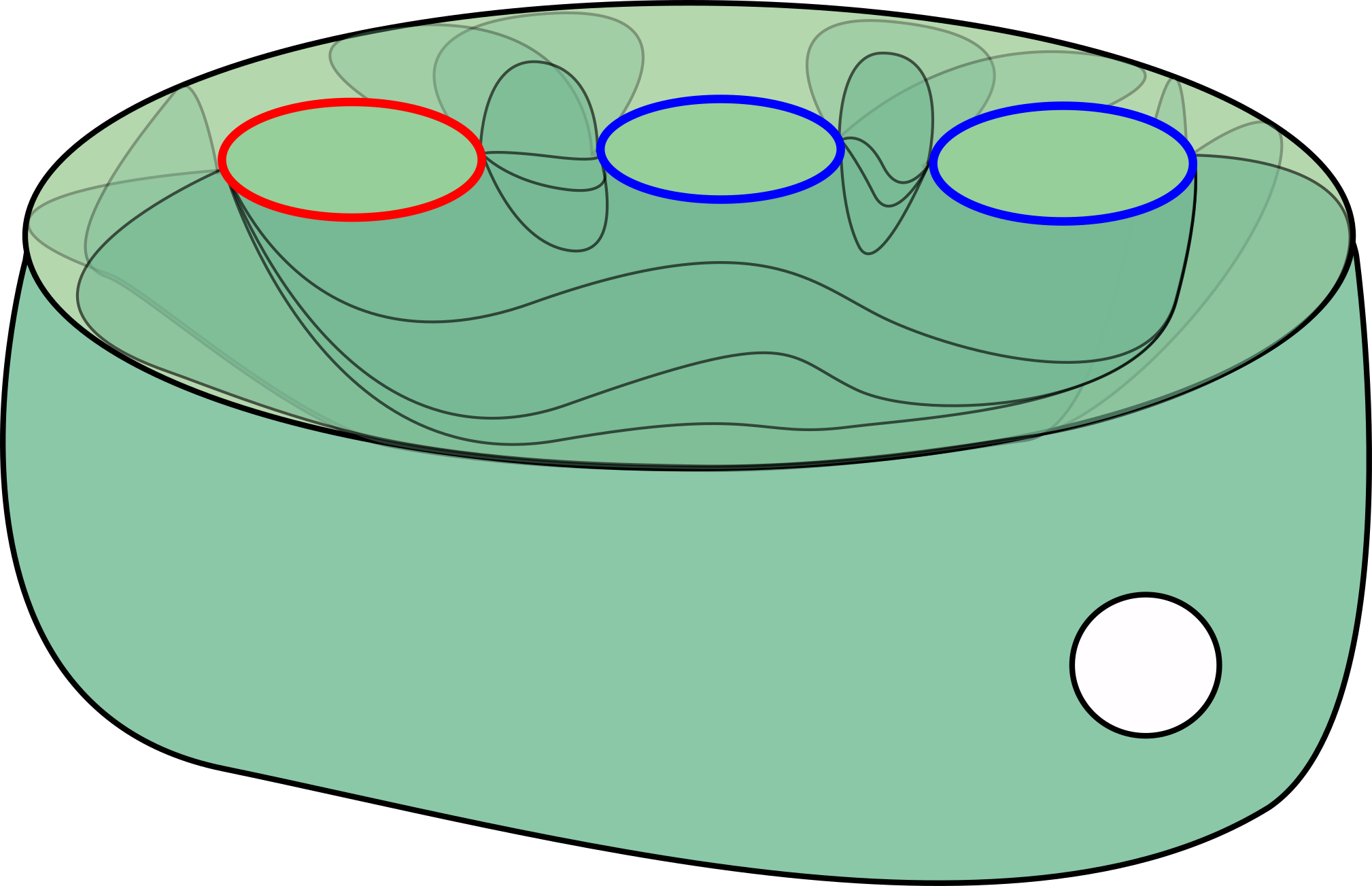} 
\caption{A cartoon of a pair $(X,\pi)$ of a $4$-manifold $X$ and a marked open book $\pi$ on $\partial X$. The boundary is foliated with surfaces (the fibers of $\pi$) except along the red and blue circles (the binding). The red circle is the marked component.}\label{fig:intro_4_manifold}
\end{figure} 

In our construction, we require a presentaton of $(X,\pi)$ as a \emph{marked trisection diagram}, which is a simple variation of the relative trisection diagrams introduced by Castro-Gay-Pinz\'{o}n-Caicedo \cite{cgpc2018relativetrisections}. Roughly, a marked trisection diagram
\[T = (\Sigma;\alpha,\beta,\kappa; C)\]
consists of a surface $\Sigma$, a marked boundary component $C \subset \partial \Sigma$, and three embedded closed 1-manifolds $\alpha, \beta, \kappa$ in $\Sigma$ satisfying certain conditions (see Definition \ref{def:trisection_diagram}).

\vspace{3pt}

%

There are a number of important moves relating different trisections: \emph{isotopy} of the manifolds $\alpha,
\beta$ and $\kappa$; \emph{handleslides} of the constituent curves in $\alpha,\beta$ and $\kappa$; and \emph{(de)stabilization} by boundary sum with one of the standard stabilizing diagrams (see Figure \ref{fig:stabilized_sphere_trisection_intro}) along the marked boundary component. A direct adaption of Theorem 3 in \cite{cgpc2018relativetrisections} states the following.

\begin{thm} There is a bijection between diffeomorphism classes of pairs $(X, \pi)$ and equivalence classes of marked trisection diagrams $T$ up to isotopy, handleslides, stabilization and destabilization. \end{thm}

\vspace{-10pt}

\begin{figure}[h!]
\centering
\includegraphics[width=\textwidth]{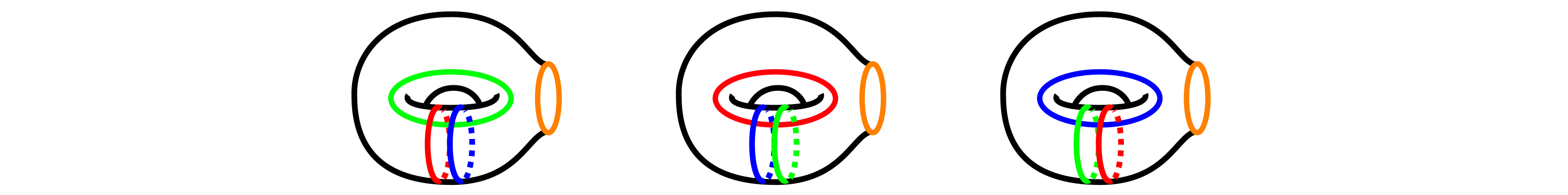}

\vspace{-5pt}

\caption{The standard stabilization diagrams $T^{\op{st}}_1,T^{\op{st}}_2$ and $T^{\op{st}}_3$.}\label{fig:stabilized_sphere_trisection_intro}
\end{figure}

In the non-semisimple setting, the marked trisection diagram $T$ must also be equipped with a \emph{singular combing}. This is a type of singular vector field $v$ on the surface $\Sigma$ satisfying a few criteria. For instance, $v$ has exactly one singularity of index $-1$ on each component of $\alpha$, $\beta$, and $\gamma$, and it has no additional singularities. Also, the degree of $v$ on each component of $\partial \Sigma$ is zero except on the marked component $C$. See Definition \ref{def:combing_on_diagram}. A \emph{combed} marked trisection diagram $(T, v)$ is a marked trisection diagram $T$ together with a combing $v$ on it. 

\vspace{3pt}

Our invariant will only depend on the choice of combing up to equivalence by certain combing moves. The first move, a \emph{basepoint isotopy}, pushes one singularity of a combing across an intersection between two trisection curves on $T$. The second move, a \emph{basepoint spiral}, twists a combing near one of its singularities. See Definition \ref{def:combing_move}. Two combings $T$ are defined to be equivalent if they are related by a sequence of diagram isotopies,  basepoint isotopies, and  basepoint spirals. Denote by $[v]$ the equivalence class of the combing $v$, and by \[\text{Comb}(T)\]
the set of equivalence classes of combings on $T$. Analogously to spin-c structures, the set of singular combings is a torsor over a homology group.

\begin{thm}
    Let $T$ be a marked trisection diagram for the pair $(X,\pi)$ of a connected compact smooth 4-manifold $X$ and a marked open book $\pi$ on $\partial X$. Then $\text{Comb}(T)$ admits a free transitive action of $H_1(X)$. 
\end{thm}

\begin{cor}
If $H_1(X) = 0$, then there is a unique equivalence class of combings on any marked trisection diagram of  $(X, \pi)$.
\end{cor}

\begin{remark} \label{rmk:intro_intrinsic_combings} In \cite{kuperberg1996noninvolutory}, Kuperberg is able to construct an essentially canonical bijection between equivalence classes of singular combings of a Heegaard diagram $(\Sigma,\alpha,\beta)$ and isotopy classes of non-vanishing vector-fields on the corresponding $3$-manifold $Y$. We are not aware of any such intrinsic description for $\text{Comb}(T)$. 

\vspace{3pt}

In particular, this prevents us from canonically associating an invariant to triples $(X,\pi,v)$ of a $4$-manifold $X$ with boundary open book $\pi$ and some additional, intrinsic data $v$. We hope to remedy this issue in future work.
\end{remark}

\subsection{Algebraic input} The algebraic input to our trisection invariant is a balanced Hopf triplet equipped with a set of cointegrals. As introduced in \cite{chaidez20194manifold}, a Hopf triplet $\mathcal{H}$ consists of three Hopf algebras over a field $k$, denoted by
\[H_\mu \quad\text{for}\quad \mu \in \{\alpha,\beta,\kappa\}\]
and three bilinear pairings satisfying several compatibility conditions, denoted by 
\[\langle-\rangle_{\mu\nu}\colon H_\mu \otimes H_\nu \to k \qquad\text{for each pair }\mu\nu \in \{\alpha\beta,\beta\kappa,\kappa\alpha\}\]
Similarly, the choice of cointegrals is simply a choice of elements
\[e_\mu:k \to H_\mu \quad\text{for}\quad \mu \in \{\alpha,\beta,\kappa\}\]
satisfying the standard cointegral identities (see Section \ref{subsec:hopf_algebras}). In \cite{chaidez20194manifold}, all Hopf algebras are assumed to be \emph{involutory}, i.e.~to have antipode squared equal to the identity map. In this paper, we weaken this by assuming that the Hopf algebras are \emph{balanced}, in the terminology of Kuperberg \cite{kuperberg1996noninvolutory}. A Hopf algebra is balanced if the antipode squared is equivalent to conjugation by the phase element, and a Hopf triplet $\mathcal{H}$ is balanced if all three of its Hopf algebras are balanced. See Definition \ref{def:hopf_triplet}.

\vspace{3pt}

Any balanced Hopf triplet $\mathcal{H}$ has an associated \emph{phase} $q_{\mathcal{H}} = \pm 1$ (see Definition \ref{def:triplet_phase_balanced}). Our invariant will take values in the set $k/\sim$ of equivalence classes, where two elements are equivalent if they are equal up to a multiplicative factor of $q_{\mathcal{H}}$. The usual multiplication on $k$ is well-defined on $k/\sim$. The scalar invariants that we define will take values in $k/\sim$.

\subsection{Main results} We can now outline the main results of this paper. Given a balanced Hopf triplet $\mathcal{H}$  equipped with cointegrals and a combed, marked trisection diagram $T$ with singular combing $v$, we construct an associated scalar, called the \emph{bracket}
\[\big\langle (T,v)\big\rangle_{\mathcal{H}} \in k\]
This scalar is obtained by contracting a tensor diagram assigned to $(T, v)$ consisting of various tensors from $\mathcal{H}$ (e.g. comultiplication, left/right cointegrals, pairings, etc.). The assignments are similar to the semisimple setting in~\cite{chaidez20194manifold}, but important adjustments are needed to deal with non-semisimplicity of the triplets and the appearance of combings.

\begin{thm}[Prop. \ref{prop:inv_combing_move} and Prop. \ref{prop:inv_trisection_move}] The bracket
$\langle (T,v) \rangle_{\mathcal{H}}$ only depends on the trisection $T$ and the equivalence class $[v]$ of $v$ as an element of $k/\sim$. Moreover, it is invariant under handle slides and isotopy. 
\end{thm}

\noindent If $T$ represents a 4-manifold $X$ with $H_1(X) = 0$, then every singular combing is isotopic. In this case $\langle T \rangle_{\mathcal{H}}$ will denote the bracket for any choice of combing. The invariant also satisfies a product rule with respect to boundary sum of diagrams and combings along the marked component.

\begin{prop} The bracket $\langle -\rangle_{\mathcal{H}}$ is multiplicative under boundary sum of combed, marked trisections:
\[
\langle(S \; \natural \; T, u \; \natural \; v)\rangle_{\mathcal{H}} = \langle (S, u)\rangle_{\mathcal{H}} \cdot \langle (T, v)\rangle_{\mathcal{H}}
\]
\end{prop}

In order to define the trisection invariant from the bracket, we must assume that the standard stabilization trisections (see Figure \ref{fig:stabilized_sphere_trisection_intro}) have non-zero invariant:
\begin{equation}\label{eq:intro_nondegenerate_H} \langle T^{st}_i \rangle_{\mathcal{H}} \neq 0 \quad\text{for}\quad i = 1,2,3\end{equation}
Under this assumption, for a combed marked trisection diagram $(T,v)$ of type $(g,k_1, k_2, k_3; p,b)$ (see Definition \ref{def:trisection_diagram}), we define
\begin{align*}
    \tau_{\mathcal{H}}(T,[v]):=& \langle T^{st}_1\rangle_{\mathcal{H}}^{-k_1} \cdot \langle T^{st}_2\rangle_{\mathcal{H}}^{-k_2} \cdot  \langle T^{st}_3\rangle_{\mathcal{H}}^{-k_3} \cdot \langle (T,v) \rangle_{\mathcal{H}}\,.
\end{align*}
Finally, let $(X, \pi)$ be a pair of a smooth 4-manifold $X$ with boundary and a marked open book $\pi$ on $\partial X$, and let $T$ be a trisection diagram for $(X,\pi)$. Then we define
\begin{align*}
        \tau_{\mathcal{H}}(X, \pi) := \Big\{\tau_{\mathcal{H}}(T, [v]) \ \Big|\ [v] \in \text{Comb}(T)\Big\} \subset k/\sim.
\end{align*}

\noindent The main result of the paper is the following.
\begin{thm}
    The set $\tau_{\mathcal{H}}(X, \pi)$ is a diffeomorphism invariant of the pair $(X, \pi)$ of a connected compact smooth 4-manifold $X$ with non-empty, connected boundary $\partial X$ and a marked open book $\pi$ on $\partial X$.
\end{thm}

\noindent The set $\op{Comb}(T)$ has one element if $H_1(X) = 0$. Thus, in this case we have the following result.

\begin{cor} Let $X$ be a connected compact smooth 4-manifold $X$ with $H_1(X) = 0$ with non-empty, connected boundary $\partial X$ and $\pi$ be a marked open book on $\partial X$. Then
\[\tau_{\mathcal{H}}(X, \pi) \in k/\sim\]
is a scalar invariant of the pair $(X, \pi)$.
\end{cor}

The condition (\ref{eq:intro_nondegenerate_H}) is critical to the normalization of the invariant. As we will discuss later, all of the examples of Hopf triples which we presently know to satisfy this condition all happen to be semisimple.  We cannot prove if semisimplicity is necessary. We thus leave it as future work to find non-semisimple Hopf triplets satisfying the condition. 

\begin{remark}[Set-Valued] The set-valued nature of our invariant $\tau_{\mathcal{H}}$ is a consequence of our lack of an intrinsic characterization of $\op{Comb}(T)$ (see Remark \ref{rmk:intro_intrinsic_combings}). We simply ``integrate out'' all combings to obtain an invariant. We leave it as a future direction to explore a topological interpretation of combings in terms of extra structures on 4-manifolds, to obtain a scalar invariant.
\end{remark}

\begin{remark}[Cointegrals] The invariant depends on the choice of cointegrals only up to a simple, overall scalar factor. We will discuss this dependence in detail later in the paper.
\end{remark}

\begin{remark}[Closed Case] Any closed $4$-manifold $M$ can be converted into a pair $(X,\pi)$ by deleting a ball $B \subset M$ and equipping the sphere boundary with the standard open book $\pi_{\op{std}}$ on $S^3$ with one binding component and genus $0$ pages. We can thus let
\[
\tau_{\mathcal{H}}(M) := \tau_{\mathcal{H}}(M \setminus B,\pi_{\op{std}})
\]
For semisimple Hopf triples, this invariant is essentially equivalent to the original constructed in \cite{chaidez20194manifold} up to an overall scalar multiplicative factor.\end{remark}

\begin{remark} There have been a number of very recent works on non-semisimple (3+1)-TQFTs. We highlight, in particular, the construction of Constantino-Geer-Haioun-Patureau-Mirand using ribbon categories \cite{costantino2023skein} and the announced work of Reutter-Walker \cite{walker2021universal}. It may be fruitful to compare these invariants with the trisection invariants discussed in this paper.
\end{remark}

\subsection{Calculations for exotic pairs} Several calculations of the closed version of our invariant were performed in \cite{chaidez20194manifold}. However, as with many quantum invariants, the runtime of these computations grows rapidly with the complexity of the diagrams. Thus, a major challenge to the further study of these invariants is the lack of simple examples of exotic smooth phenomena in trisection form.

\vspace{3pt}

In the course of this work, we computed trisections for a family of exotic pairs called \emph{Stein nuclei}. These spaces were originally introduced by Yasui \cite{yasui2014partial} as examples of small exotic pairs of Stein 4-manifolds. These trisections are among the smallest examples of exotic pairs presented via trisections. These examples provide a valuable testing ground for future non-semisimple invariants based on this work. 

\begin{figure}[h!]
\centering
\includegraphics[width=.3\textwidth]{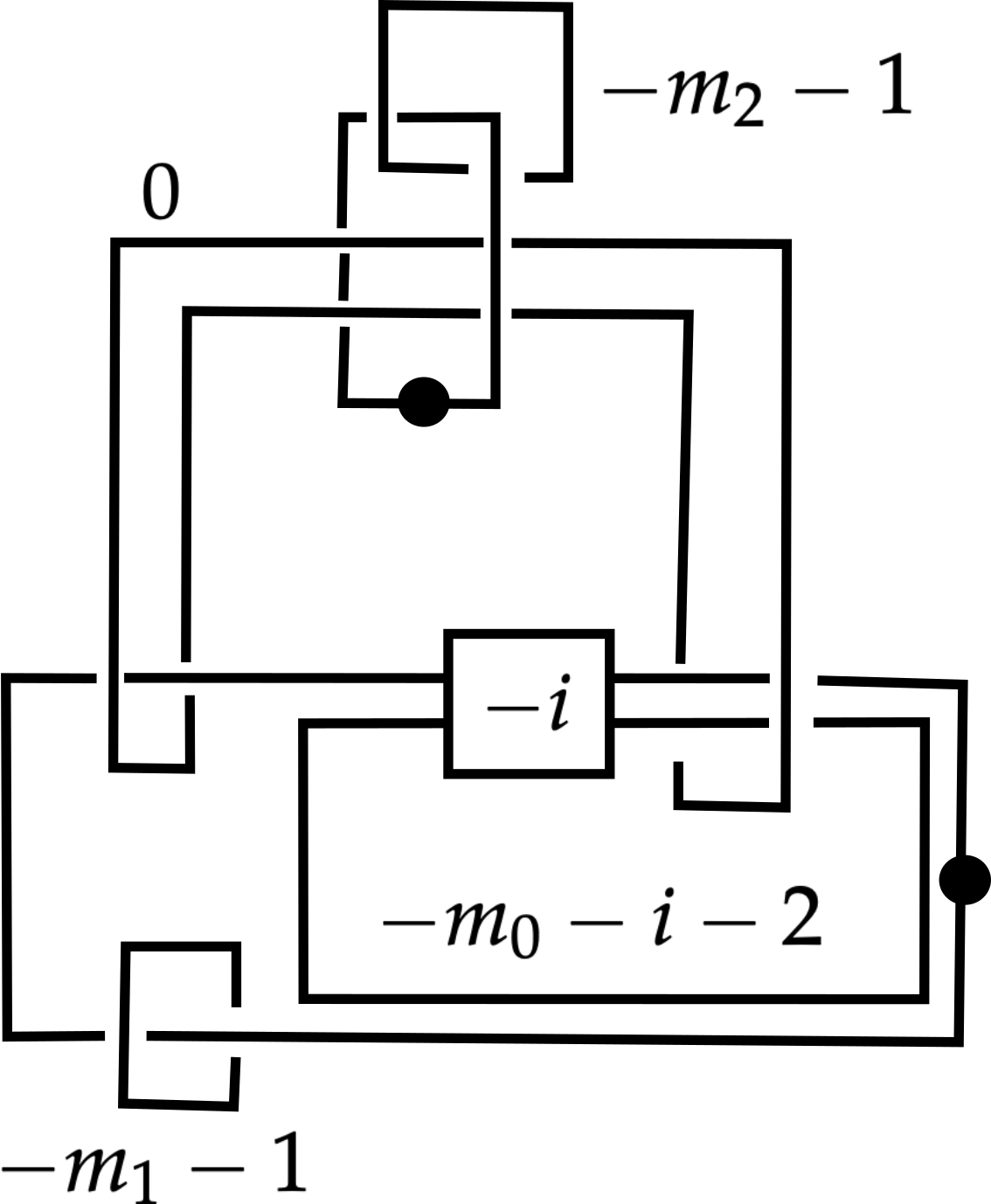}
\caption{The Kirby diagram for an arbitrary Stein nucleus. These manifolds are parametrized by certain integers $m_0,m_1,m_2$ and $i$.}\label{fig:intro_stein_nucleus}
\end{figure} 

\noindent We discuss the computations of the semisimple invariants of Stein nuclei, and a different pair of small exotic trisections introduced by Takahashi \cite{takahashi2023exotic}, in Section \ref{sec:stein_nuclei}.

\subsection*{Outline} This concludes the introduction {\bf \S \ref{sec:introduction}}. The rest of the paper is organized as follows.

\vspace{3pt}

In {\bf \S \ref{sec:preliminaries}} and {\bf \S \ref{sec:Hopf_triplets}} we discuss the algebraic preliminaries required for this paper, including Hopf algebras and Hopf triplets. In {\bf \S \ref{sec:trisections_and_combings}}, we discuss the topological preliminaries, including marked open books, relative trisections and combings. In {\bf \S \ref{sec:main_results}}, we prove the main results of the paper, defining the non-semisimple, relative trisection invariants and proving their basic properties. In {\bf \S \ref{sec:stein_nuclei}}, we trisect the Stein nuclei and discuss some aspects of the trisection invariants in this case.

\subsection*{Conventions} Throughout the paper, a 4-manifold $X$ is always assumed to be orientable, connected, compact, and smooth, with connected and non-empty boundary $\partial X$, unless otherwise specified. All vector spaces are finite-dimensional over a fixed algebraically closed field $k$ of characteristic zero.

\subsection*{Acknowledgements} We thank Kouichi Yasui for pointing out the examples in \cite{takahashi2023exotic} and other helpful comments on a previous draft. JCC was supported by National Science Foundation under Award No. 2103165.  JSC is supported by a Junior Fellowship from the Harvard Society of Fellows. SXC is partially supported by a startup fund from Purdue University and by National Science Foundation under Award No. 2006667 and No. 2304990.

\section{Hopf Algebra} \label{sec:preliminaries}

In this section, we present the background on Hopf algebras needed for the construction of the invariants described in this paper. 

\subsection{Tensor formalism} \label{subsec:tensor_formalism} We will use a standard, diagrammatic notation for tensor calculus (cf. \cite{kuperberg1996noninvolutory}). We begin by introducing this notation and establishing conventions.

\begin{notation}[Tensor] \label{not:tensor_symbol} The \emph{tensor symbol} for a linear map (or tensor) $T$ of the form
\[
T:U_1 \otimes \cdots \otimes U_a \to V_1 \otimes \cdots \otimes V_b \qquad \text{for vector-spaces $U_i$ and $V_j$}
\]
is a planar node $f$ with $a$ incoming edges and $b$ outgoing edges, written as follows:
\vspace{.1cm}
\begin{align*}
    \includegraphics[scale=.45, valign = c]{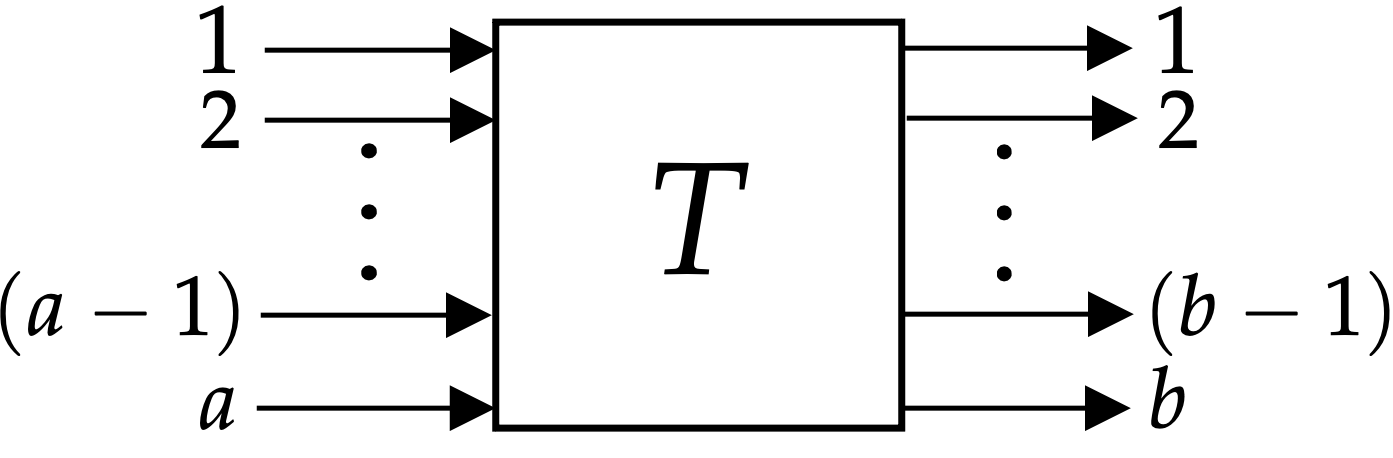}
\end{align*}
\vspace{.1cm}
The edges $i \to$ and $\to j$ correspond to the $U_i$ and $V_j$ factors, respectively. \end{notation}

We also fix the following notation for two basic operations on tensors: tensor products and traces. The simplicity of the notation for these operations is the main benefit of the diagrammatic formalism. 

\begin{notation}[Tensor Product] The symbol of the tensor product $S \otimes T$ of $S$ and $T$ is denoted
\vspace{.1cm}
\begin{align*}
    \includegraphics[scale=.4, valign = c]{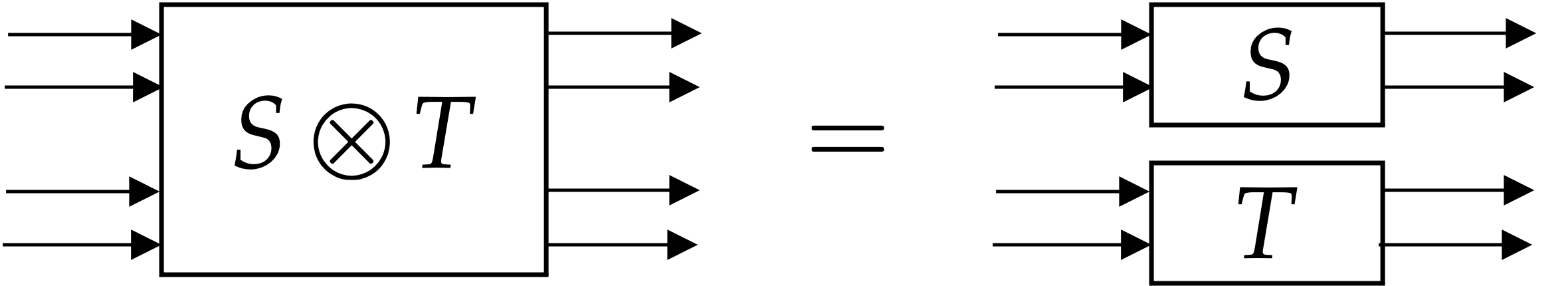}
\end{align*}
\end{notation}

\begin{notation}[Partial trace] The partial trace $\op{Tr}^c_d(T)$ of a tensor along an incoming index $c$ and outgoing index $d$, if $U_c = V_d$, is given by connecting the corresponding edges of the symbol for $T$:
\begin{align*}
    \includegraphics[scale=.4, valign = c]{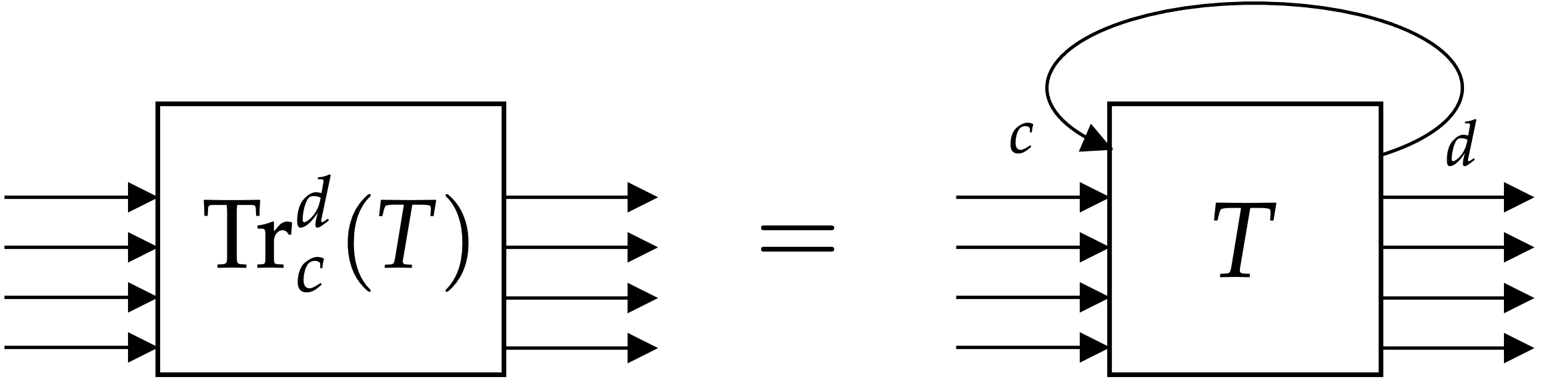}
\end{align*}
\vspace{.1cm}
\end{notation}

\begin{remark}[Ordering] We will usually omit the box of the tensor $T$. Also, most tensors we will be dealing with have the same vector space for all incoming and outgoing legs. In this case, unless the edges of $T$ are ordered explicitly, we take the following convention: \textbf{By default, the incoming edges  are ordered counterclockwise, and the outgoing edges are ordered clockwise.} When both types of edges are present, there is no ambiguity in determining the first incoming edge and the first outgoing edge. When there is only one type of edges present,  if the tensor is invariant under cyclic permutation of its factors, then there is no need to determine the first edge and the above convention of ordering still works. The exception occurs when the tensor is not invariant under cyclic permutation. In this case, we explicitly label the first edge and still refer to the above ordering convention.  

\end{remark}

\begin{notation}[Tensor contraction] Given two tensors $S$ and $T$, if an outgoing edge $d$ of $S$ and an incoming edge $c$ of $T$ correspond to the same vector space, then one can connect $d$ and $c$ to form a larger tensor. This is called a tensor contraction.  Note that any tensor contraction of $S$ and $T$ can be written as a partial trace of $S \otimes T$. 
\end{notation}

\begin{notation}[Tensor diagram]
    Usually we start with a collection of known tensors, called elementary tensors, and use them to form larger tensors by tensor contractions.  These larger tensors are called tensor diagrams. There is no essential difference between tensors and tensor diagrams, and we can use them interchangeably. In practice, tensors refer to building blocks, and tensor diagrams are those tensors constructed from building blocks. A tensor diagram without open edges can be evaluated to a scalar. We call the scalar the evaluation/contraction of the tensor diagram. See below for an example of such a tensor diagram:
    \vspace{.4cm}
    \begin{align*}
    \includegraphics[scale=.09, valign = c]{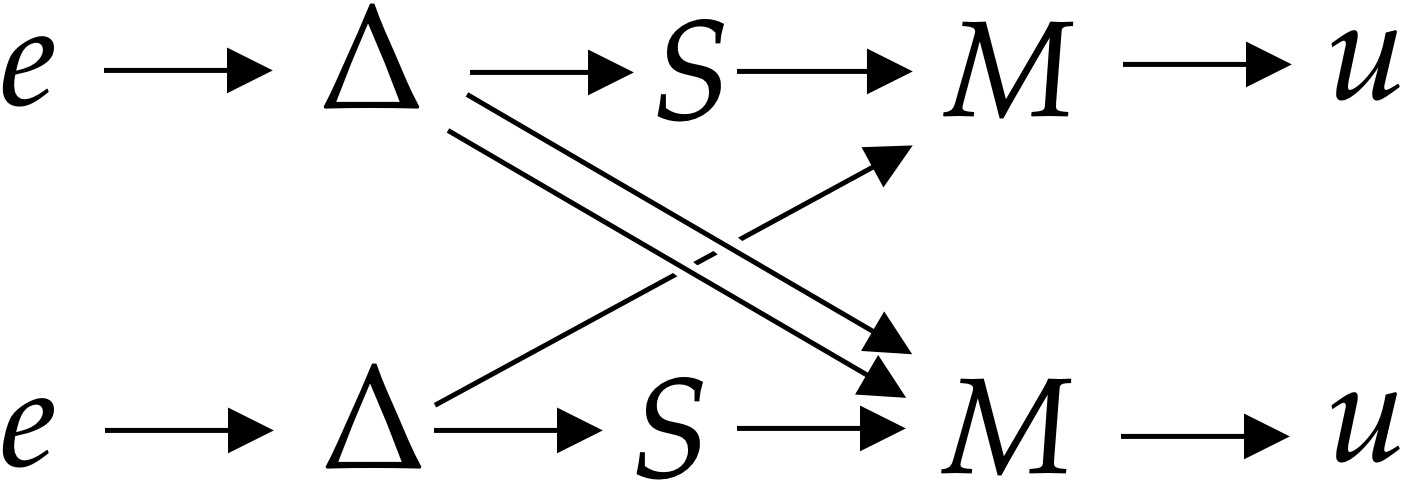}
    \end{align*}
    $$$$
\end{notation}

\begin{notation}[Edge orientation reversal]
    In a tensor, any edge corresponding to a vector space $V$ can have its orientation reversed while replacing $V$ by its linear dual $V^*$. The new tensor is naturally identified with the original tensor.
    \begin{align*}
    \includegraphics[scale=.4, valign = c]{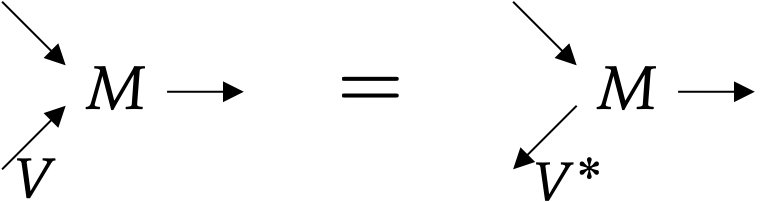}
    \end{align*}
\end{notation}

\subsection{Hopf algebras} \label{subsec:hopf_algebras} Here we review the theory of Hopf algebras. For a more detailed account of the theory, we refer the reader to Radford \cite{radford2011hopfalgebras}.

\begin{definition}\label{def:hopf_algebra} A \emph{Hopf algebra} $H = H(M, \eta, \Delta, \epsilon, S)$ is a vector space $H$ over $k$ equipped with the following structure tensors:
\begin{align*}
    \includegraphics[scale=.4, valign = c]{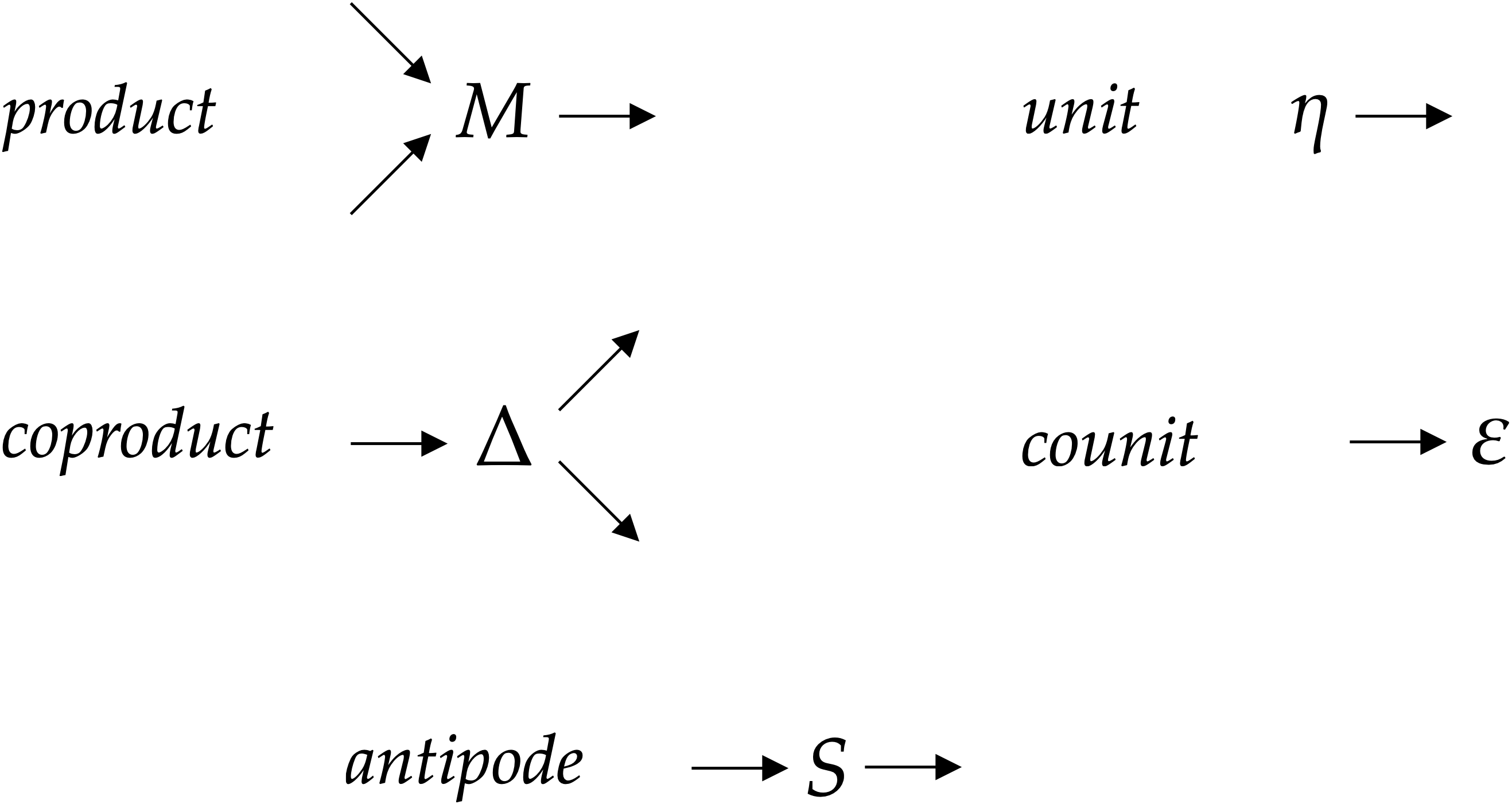}
\end{align*}
$$$$
These structure tensors must satisfy the following compatibility properties:
\begin{itemize} 

\item (Algebra) $H$ is an associative, unital algebra with product $M$ and unit $\eta$. That is, 
\vspace{.4cm}
\begin{align*}
    \includegraphics[scale=.4, valign = c]{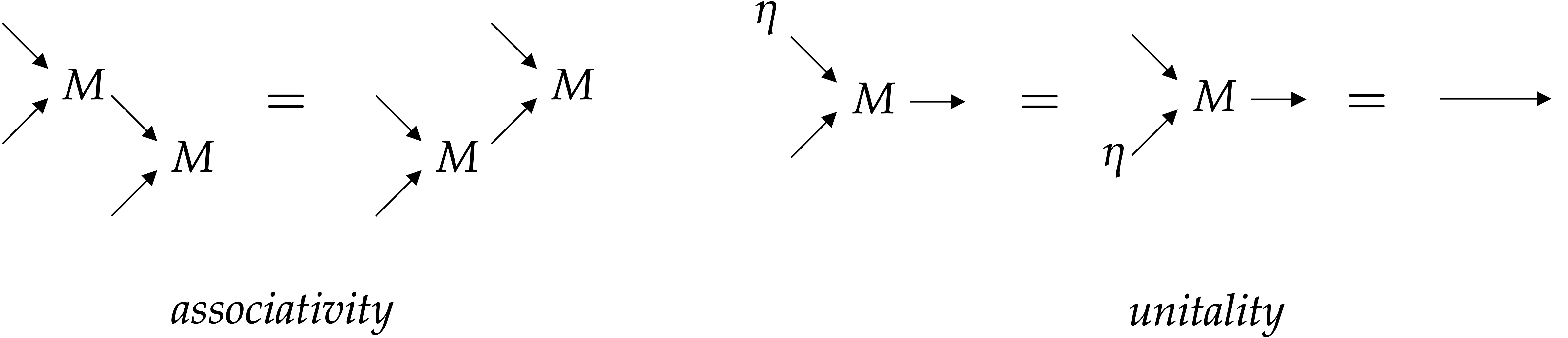}
\end{align*}
$$$$
\item (Coalgebra) $H$ is a coassociative, counital algebra with coproduct $\Delta$ and counit $\epsilon$.  In diagrams,
\vspace{.4cm}
\begin{align*}
    \includegraphics[scale=.4, valign = c]{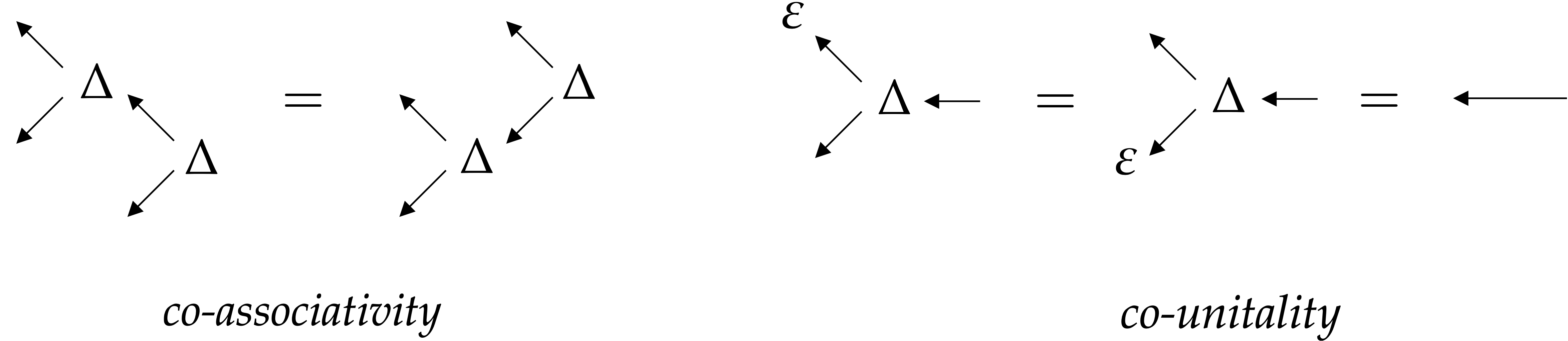}
\end{align*}
$$$$
\item (Bialgebra) The coalgebra and algebra structures define a bialgebra structure on $H$.  Diagrammatically we have
\vspace{.4cm}
\begin{align*}
    \includegraphics[scale=.4, valign = c]{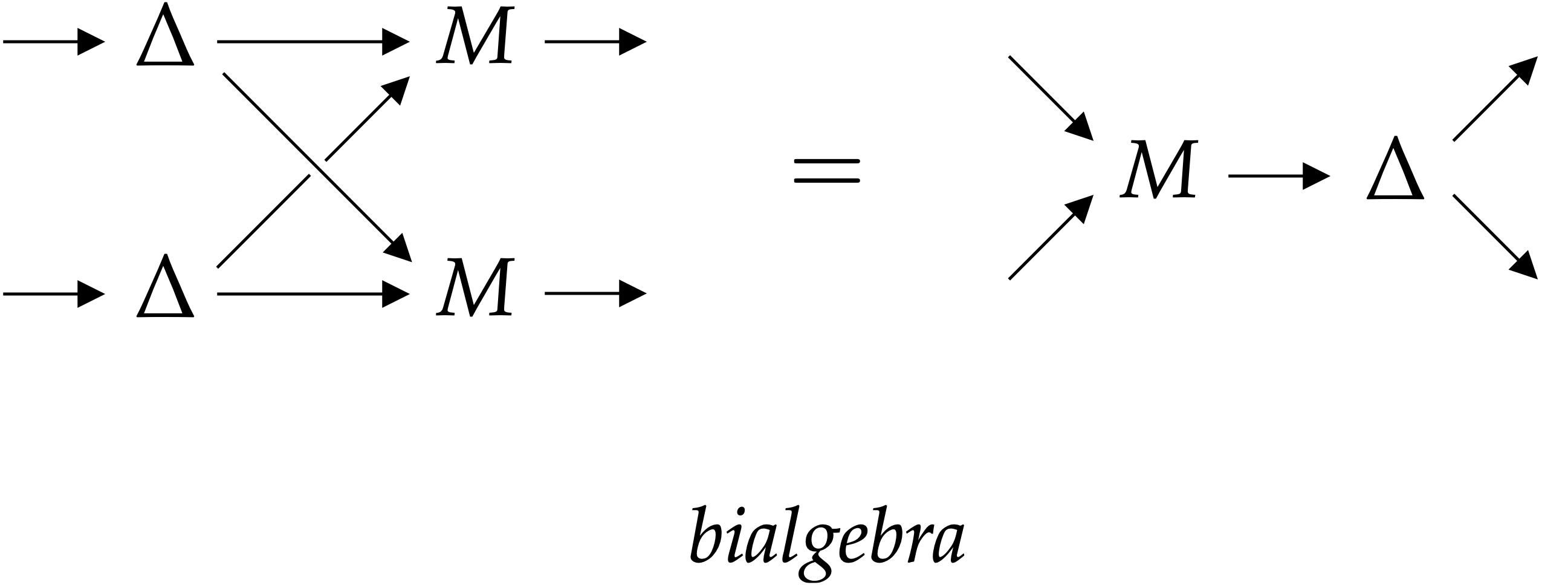}
\end{align*}
$$$$
\item (Antipode) The antipode $S$ must satisfy a standard \emph{antipode property}:
\vspace{.4cm}
\begin{align*}
    \includegraphics[scale=.4, valign = c]{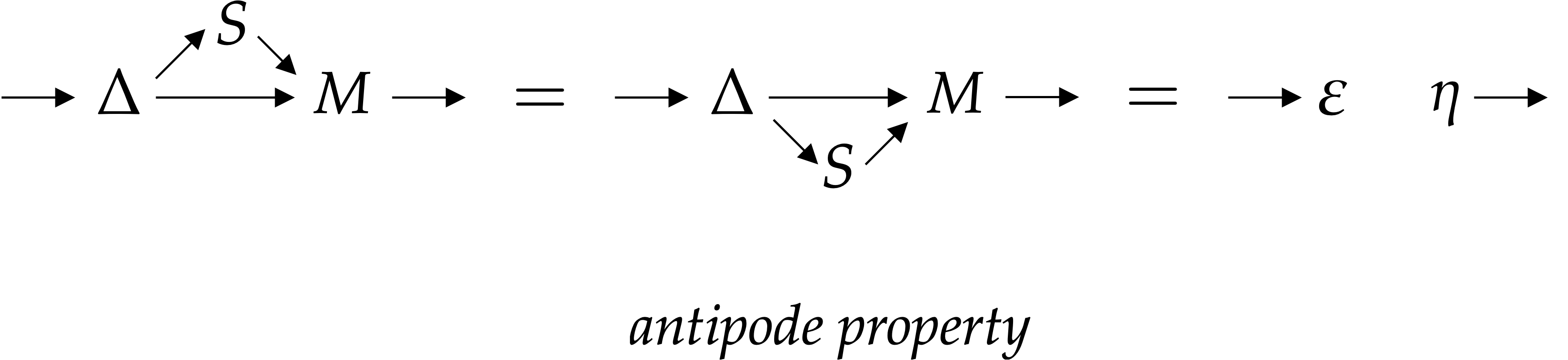}
\end{align*}
$$$$
\end{itemize}
A \emph{Hopf algebra map/morphism} $G \to H$ is a linear map that intertwines all of the structure tensors. The associativity and coassociativity axioms permit us to adopt the following abbreviated notation for iterated products and coproducts:
\vspace{.4cm}
\begin{align*}
    \includegraphics[scale=.4, valign = c]{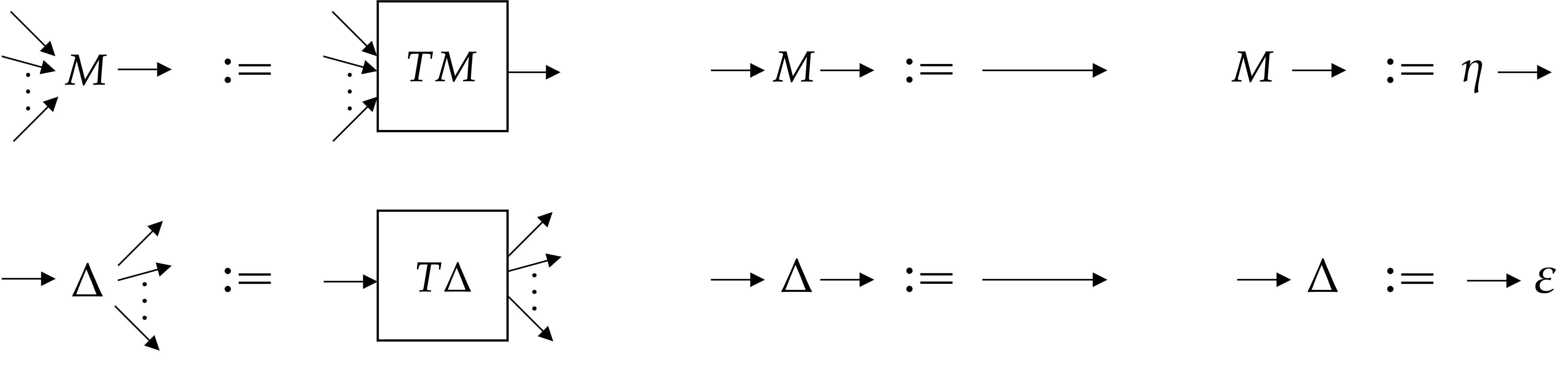}
\end{align*}
$$$$
Here $TM$ denotes an arbitrary tree with $i$ in edges, $1$ out edge, and only $M$ nodes.  Similarly $T\Delta$ denotes an arbitrary tree with $i$ out edges, $1$ in edge, and only $\Delta$ nodes.  \end{definition}

There are a number of constructions of Hopf algebras from other Hopf algebras that we will require below. 

\begin{definition} \label{def:dual_op_cop_algebras}
Let $H(M, \eta, \Delta, \epsilon, S)$ be a finite-dimensional Hopf algebra.
\begin{itemize}[leftmargin=.2in]
    \item The \emph{dual} $H^*$ is the linear dual of $H$ equipped with the structure tensors $H^*(\Delta^*, \epsilon^*, M^*, \eta^*, S^*)$. We will use the structure tensors of $H$ and interpret them as tensors for the dual space $H^*$. For instance, in $H^*$, the  product $\Delta^*$, coproduct $M^*$, and antipode $S^*$ are represented, respectively, by
\begin{align*}
    \includegraphics[scale=.4, valign = c]{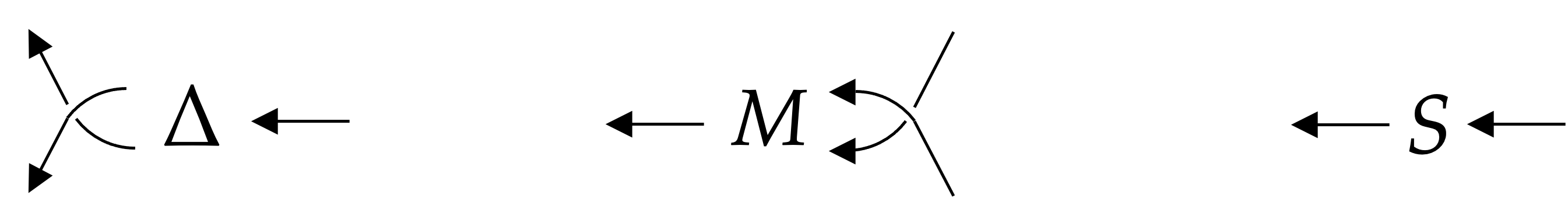}
\end{align*} 
The swapped order of the inputs and outputs in the product and coproduct tensors is necessary due to our input and output ordering convention.

\item The \emph{opposite Hopf algebra} $H^{\op{op}}$ is $H$ equipped with the structure tensors $H^{\op{op}}(M^{\op{op}}, \eta, \Delta, \epsilon, S^{-1})$ where the  product $M^{\op{op}}$ is given by the tensor
\begin{align*}
    \includegraphics[scale=.4, valign = c]{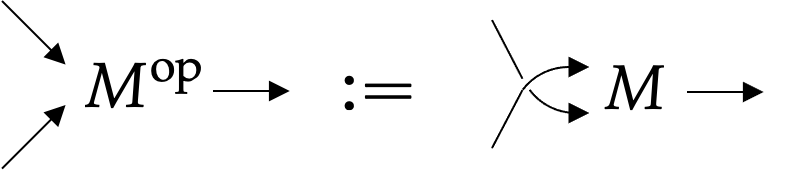}
\end{align*}
\item  The \emph{co-opposite Hopf algebra} $H^{\op{cop}}$ is $H$ equipped with the structure tensors $H^{\op{cop}}(M, \eta, \Delta^{\op{cop}}, \epsilon, S^{-1})$, where the coproduct $\Delta^{\op{cop}}$ is given by the tensor
\begin{align*}
    \includegraphics[scale=.4, valign = c]{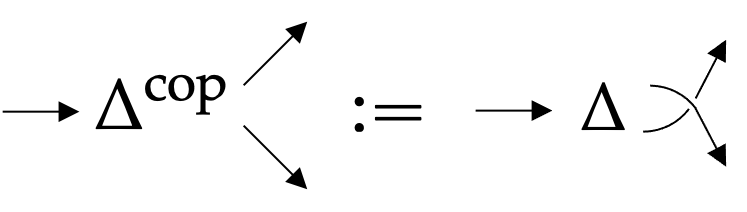}
\end{align*}
\end{itemize}

\end{definition}
Apparently, each of the three operations on Hopf algebras in the above definition is of order two, i.e.~each operation is its own inverse. We can also compose those operations, e.g.~to obtain $H^{*, \op{op}}$. The following lemma summarizes the commutation relations among these operations.
\begin{lemma}
    For a Hopf algebra $H$, we have the isomorphisms of Hopf algebras
    \begin{align*}
        H^{*, \op{cop}} &\overset{Id}{\longrightarrow} H^{\op{op},*},\\
        H^{*, \op{op}} &\overset{Id}{\longrightarrow} H^{\op{cop},*},\\
        H^{\op{op}, \op{cop}} &\overset{Id}{\longrightarrow} H^{\op{cop},\op{op}} \overset{S}{\longrightarrow} H.\\
    \end{align*}
\end{lemma}
\noindent We identity Hopf algebras in the above lemma whenever the ``identity'' map is an isomorphism.

\subsection{Integrals and phase} There are a number of special types of elements (and dual elements) of Hopf algebras that we will need for all of the constructions in this paper: (co)integrals and (co)group-likes.

\begin{definition}[Integrals] \label{def:integral_cointegral} A \emph{right integral} $u_R:H \to k$ and a \emph{right cointegral} $e_R: k \to H$ of a Hopf algebra $H$ over $k$ are maps satisfying
\begin{align*}
    \includegraphics[scale=.4, valign = c]{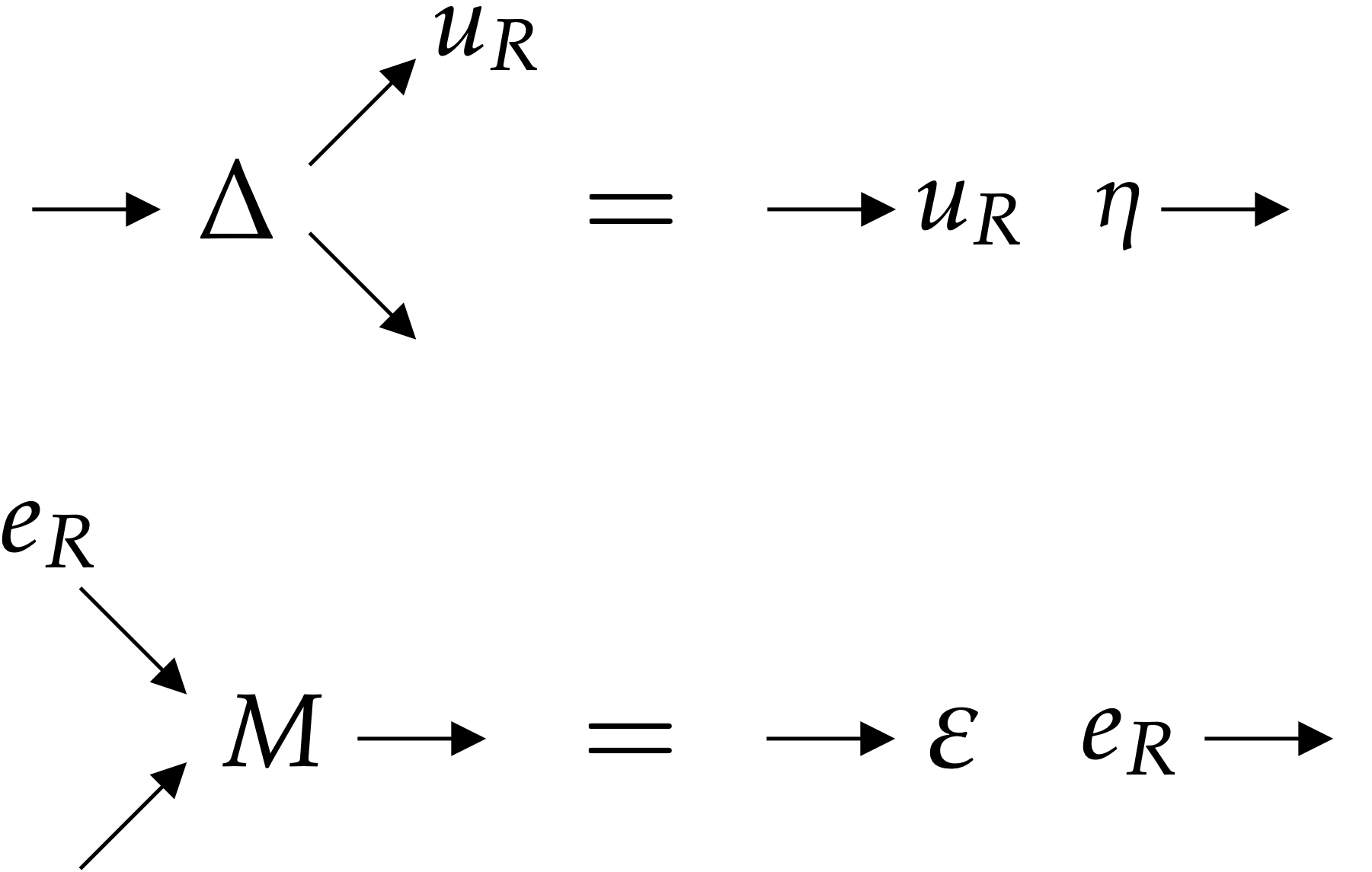}
\end{align*}
Similarly, a \emph{left integral} $u_L:H \to k$ and a \emph{left cointegral} $e_L: k \to H$ of a Hopf algebra $H$ over $k$ are maps satisfying
\begin{align*}
    \includegraphics[scale=.4, valign = c]{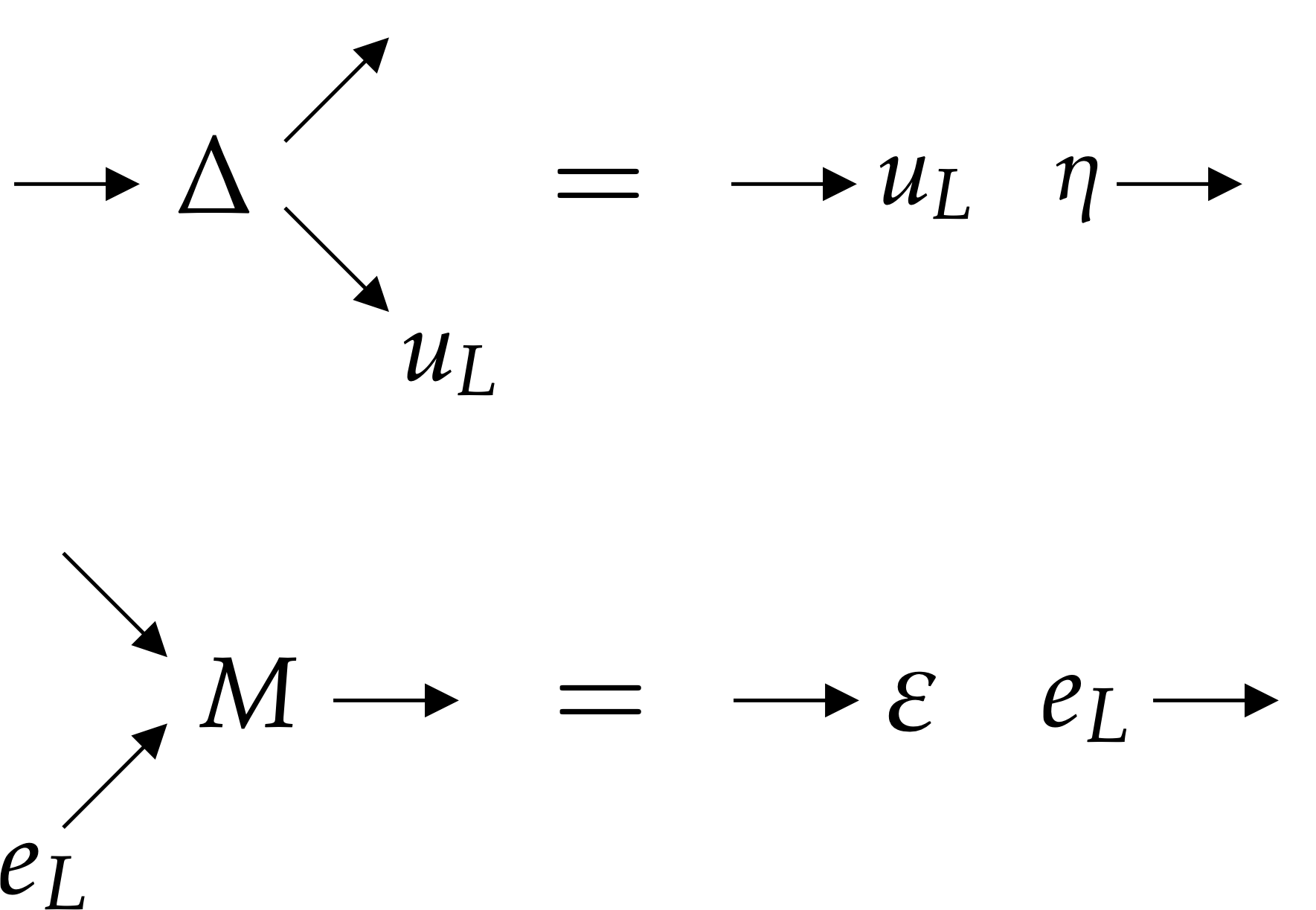}
\end{align*}
\end{definition}
Essentially, a left/right co-integral of $H$ is a left/right integral of $H^*$. It is a basic fact that for finite dimensional Hopf algebras the space of left integrals and the space of right integrals are both one-dimensional. The same holds for cointegrals. If $e_R$ is a right cointegral, then for any $x \in H$, $x e_R$ is also a right cointegral, and hence there exists an $\alpha \in H^*$ such that $x e_R = \alpha(x) e_R$. A similar argument applies to integrals which leads to the following definition.

\begin{definition}[Phase Elements]\label{def:phase_cophase}
The \emph{phase element} $a \blackarrowright $ and \emph{cophase element} $\blackarrowright \alpha$ are the unique elements such that for a right integral $u_R$ and a right cointegral $e_R$, the following identities hold:
\begin{align*}
    \includegraphics[scale=.4, valign = c]{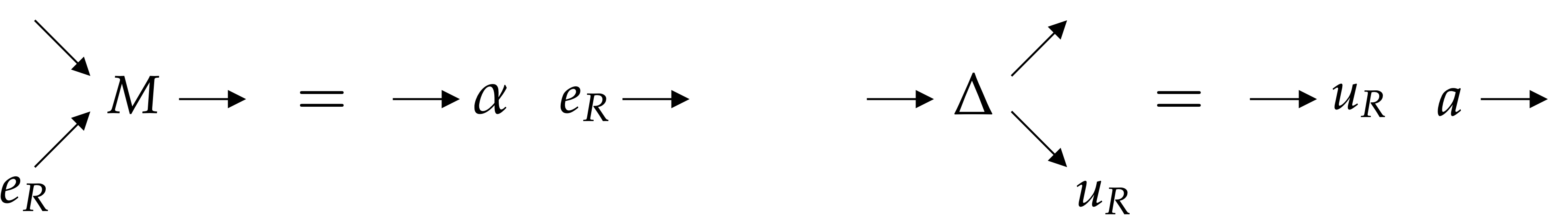}
    \end{align*}


\noindent The \emph{phase} $q$ of $H$ is the scalar $q := a \blackarrowright \alpha$. By normalizing $u_R(e_R) = 1$, the phase/cophase elements can be computed as
\begin{align*}
    \includegraphics[scale=.4, valign = c]{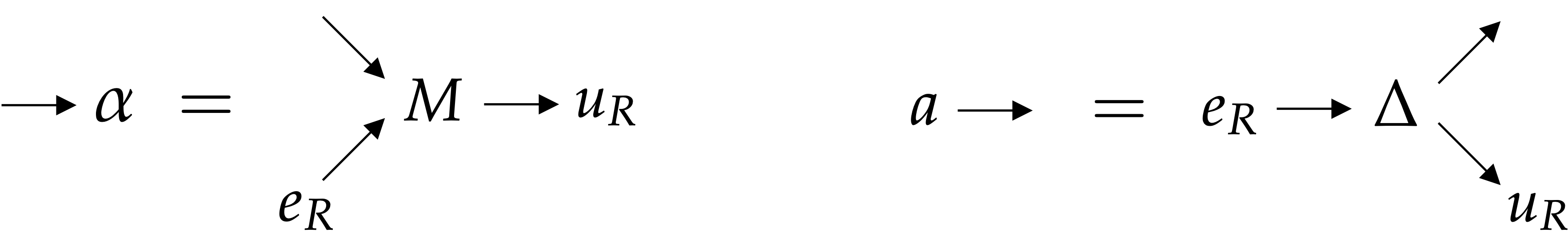}
    \end{align*}
\end{definition}

\begin{definition}[Group-Like] A \emph{group-like} $g \blackarrowright $ of a Hopf algebra $H$ is an element satisfying
\vspace{.4cm}
\begin{align*}
    \includegraphics[scale=.4, valign = c]{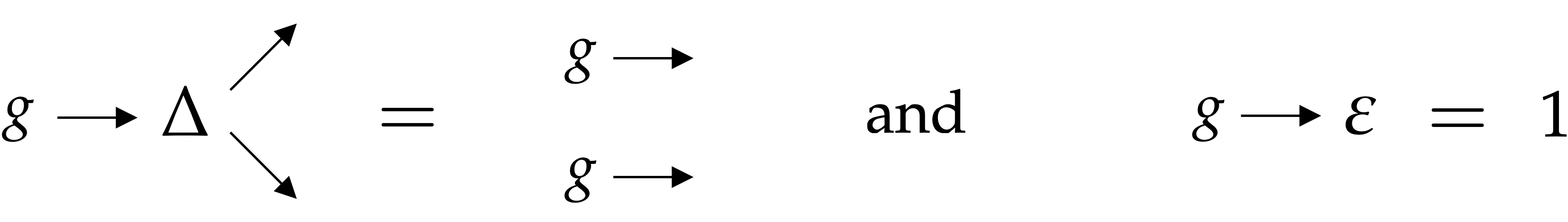}
\end{align*}
$$$$

\vspace{-10pt}

\noindent A \emph{cogroup-like} $\blackarrowright h$ is a group-like in the dual $H^*$, or equivalently, a \emph{cogroup-like} is an algebra morphism $H \to k$. \end{definition}

 The set of group-likes in $H$ form a group under the product in $H$, and for a group-like $g \in H$, $S(g) = g^{-1}$.
The phase element $a$ is  group-like and the cophase element $\alpha: H \to k$ is an algebra morphism. Hence we have $\alpha^m (a^n) = q^{mn}$. In a finite-dimensional Hopf algebra, it follows that $q$ is a root of unity since both $a$ and $\alpha$ have finite order. In the literature, $a$ and $\alpha$ are also called, respectively, the distinguished group-like elements of $H$ and $H^*$. Since the antipode maps a right cointegral to a left cointegral, and similarly for integrals, the following lemma follows from Definition \ref{def:phase_cophase}.
 \begin{lemma}\label{lem:phase_from_left_integral}
     For a left cointegral $e_L$ and a left integral $u_l$, the following identities hold:
     \vspace{.4cm}
     \begin{align*}
    \includegraphics[scale=.4, valign = c]{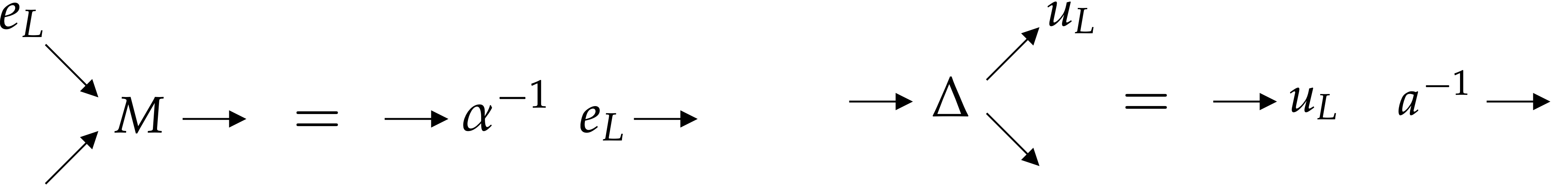}
    \end{align*}
    $$$$
 \end{lemma}

Below we list a few basic properties of integrals/co-integrals whose proofs, if not given here, can be found, for instance, in \cite{kuperberg1996noninvolutory}.

\begin{lemma}[\cite{kuperberg1996noninvolutory}]\label{lem:right_integral_cointegral_S}
The antipode can be expressed in terms of the product, coproduct, and integrals/co-integrals as
\vspace{.4cm}
 \begin{align*}
    \includegraphics[scale=.4, valign = c]{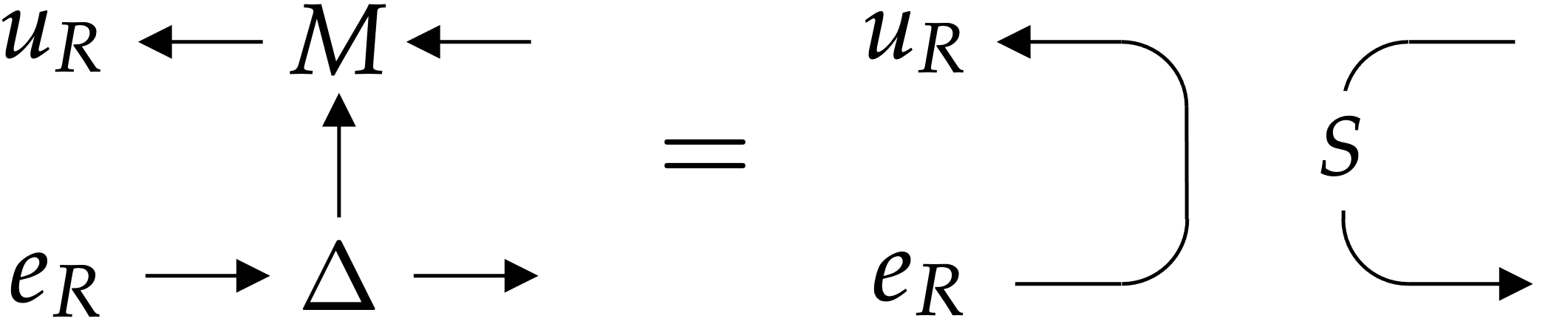}
    \end{align*}
    $$$$
\end{lemma}
By applying the above lemma to $H^{\op{op}}$ or $H^{\op{cop}}$, one can obtain similar identities,
\vspace{.4cm}
\begin{align}
\begin{split}
    &\includegraphics[scale=.4, valign = c]{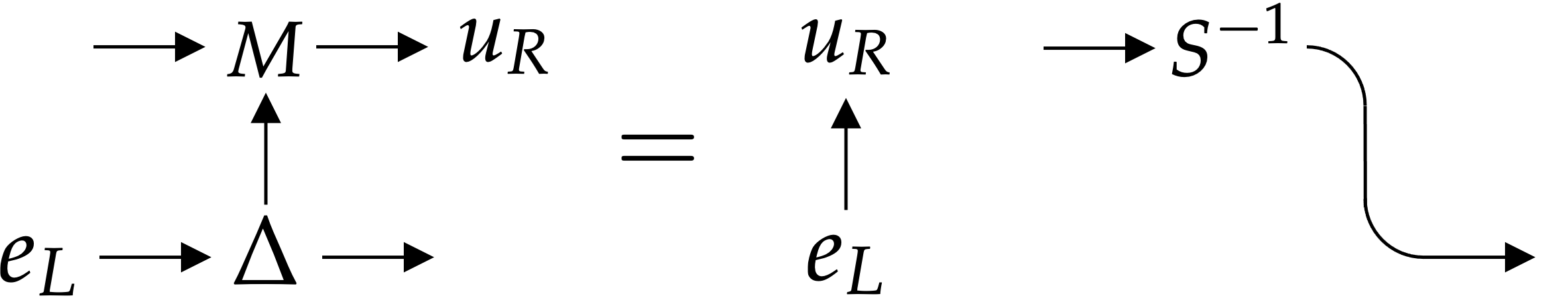} \\ \\
    &\includegraphics[scale=.4, valign = c]{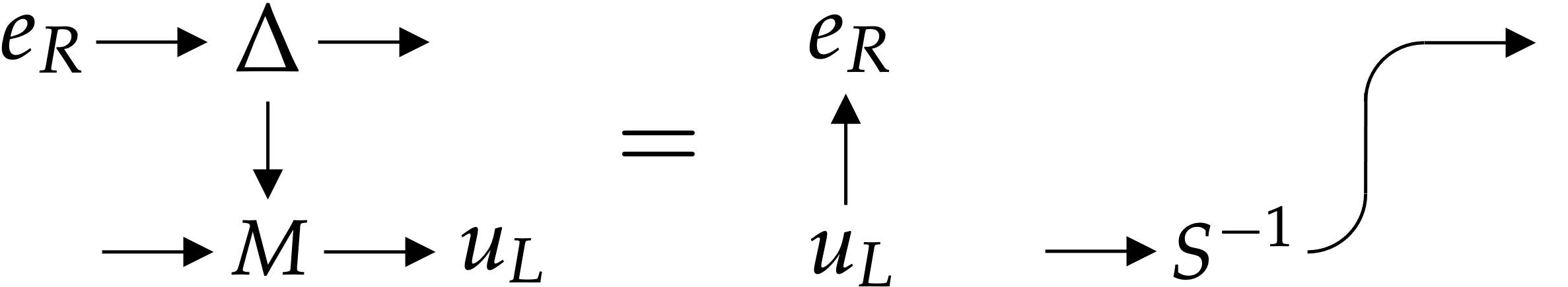}
\end{split}
\label{eqn:eR_muL_antipode}
\end{align}

$$$$

\begin{definition} \label{def:generalized_cointegral} Let $u_R$ be a right integral, $e_R$ be a right co-integral, and $n - \frac{1}{2} \in \Z + \frac{1}{2}$ be a proper half-integer. The generalized integrals/co-integrals are defined by
\begin{equation*}
    \includegraphics[scale=.4, valign = c]{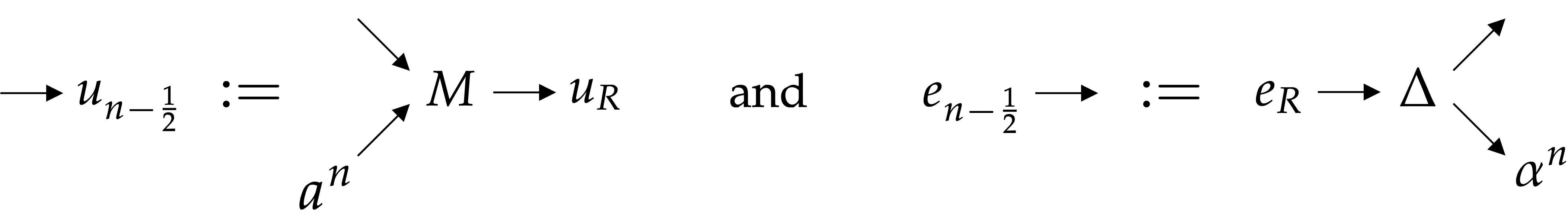}
\end{equation*}

\end{definition}

It follows from the definition that the generalized integrals/co-integrals of different indices are related by
\vspace{.4cm}
\begin{equation*}
    \includegraphics[scale=.4, valign = c]{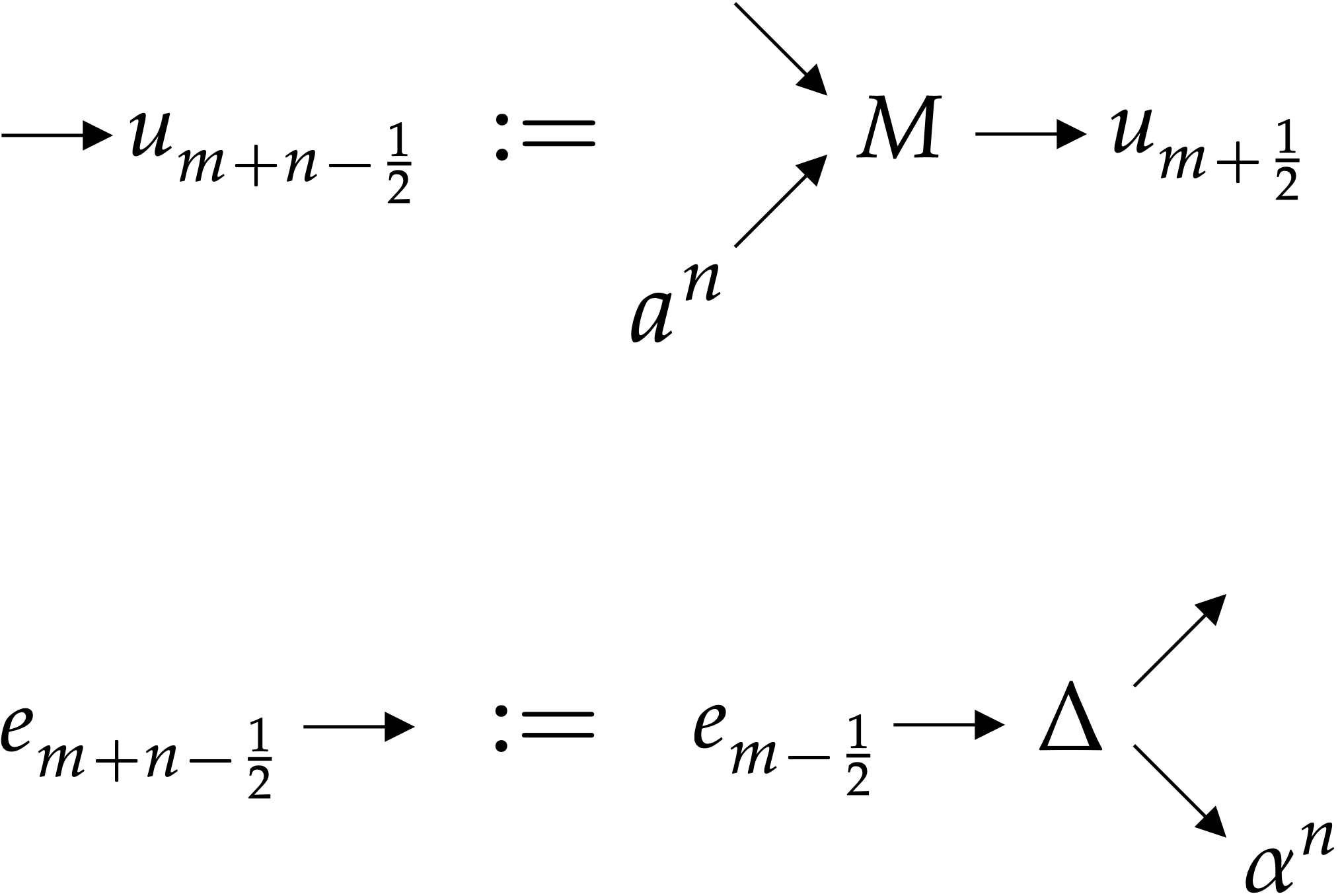}
\end{equation*}
$$$$
Note that the generalized integrals depend on the choice of a right integral. If the right integral $u_R$ is replaced by $\lambda u_R$ for a scalar $\lambda$, then the generalized integrals will also obtain a factor $\lambda$. The same applies to the generalized co-integrals.

\begin{lemma}\label{lem:generalized_integral_left_multiply}
    For any integer $n$, the generalized integrals/cointegrals satisfy the identities
    \vspace{.4cm}
    \begin{align*}
    &\includegraphics[scale=.4, valign = c]{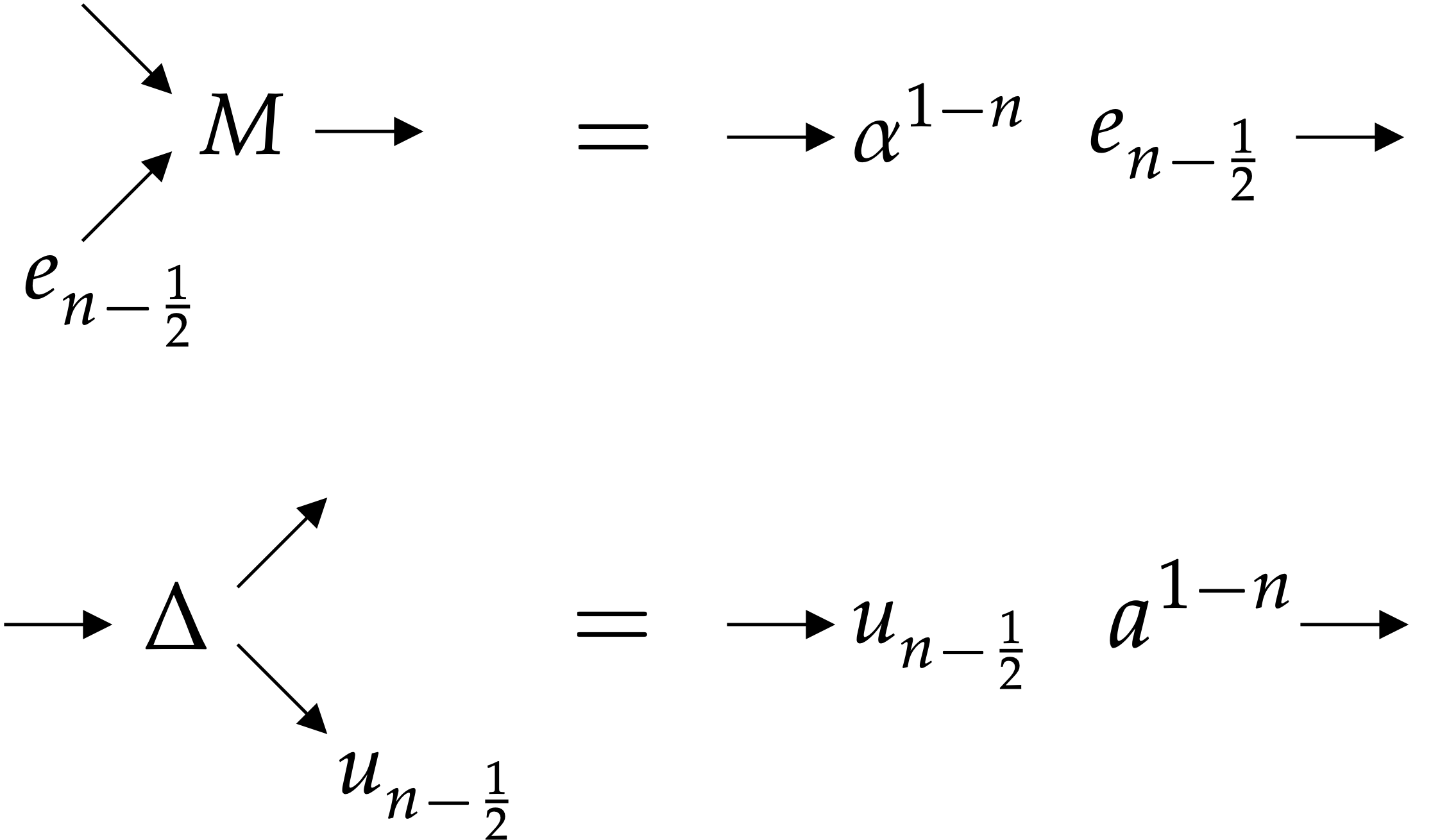}
\end{align*}
\begin{proof}
    We prove the first identity. The second one follows from the dualized version of the first identity.

    Note that since
    \vspace{.4cm}
    \begin{align*}
    &\includegraphics[scale=.4, valign = c]{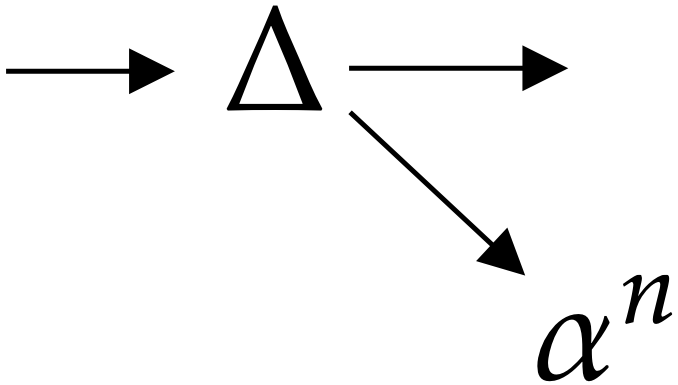}
\end{align*}
$$$$
    is an invertible map from the Hopf algebra to itself, the first identity is equivalent to
    \vspace{.4cm}
        \begin{align*}
    &\includegraphics[scale=.4, valign = c]{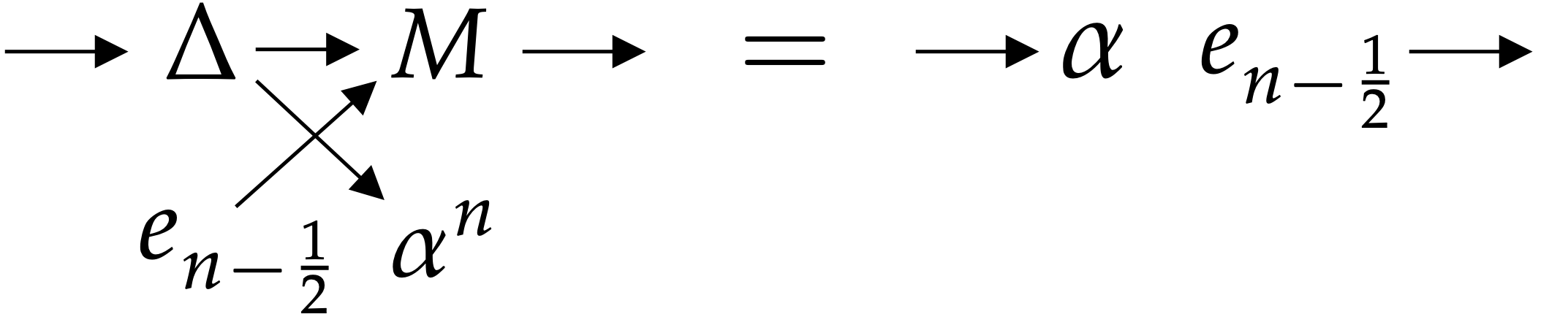}
\end{align*}
$$$$
Using the fact that $\alpha^n$ is an algebra morphism, the left hand side of the above identity can be expanded as
     \begin{align*}
    &\includegraphics[scale=.4, valign = c]{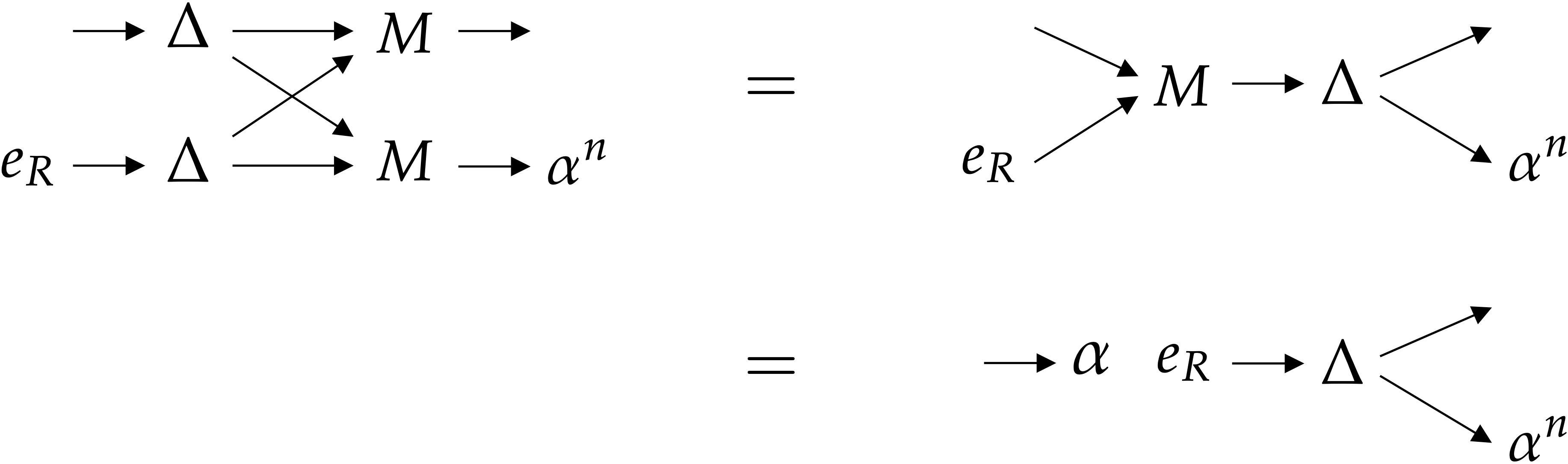}
\end{align*}
The last step above follows by the definition of $\alpha$.
\end{proof}
\end{lemma}

By Lemma \ref{lem:generalized_integral_left_multiply}, $e_{1/2}$ is a left cointegral and $ u_{1/2}$ is a left integral. Of course, $e_R = e_{-1/2}$ is a right cointegral and $u_R = u_{-1/2}$ is a right integral. For the reminder of this subsection, we always assume $e_L := e_{1/2}$ and $u_L := u_{1/2}$ are defined in such a way that
\vspace{.4cm}
\begin{equation*}
    \includegraphics[scale=.4, valign = c]{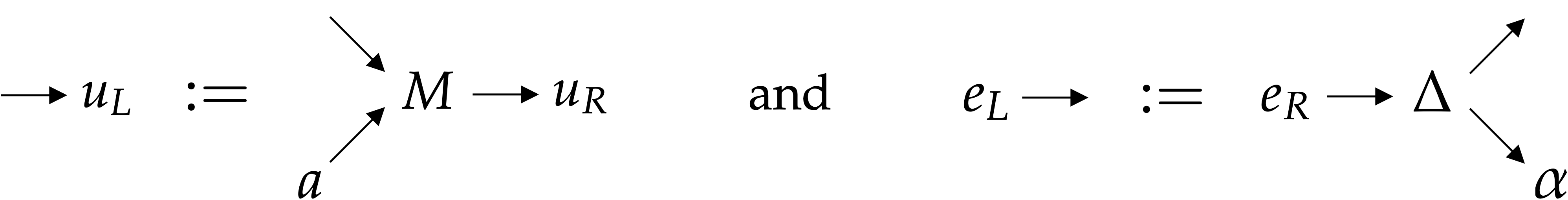}
\end{equation*}
$$$$

\noindent Then we have:
\begin{lemma}\label{lem:integral_cointegral}
    For a right integral $u_R$ and a right cointegral $e_R$ on $H$ such that $u_R(e_R) = \lambda$, we have
    \begin{align*}
        u_R(e_L) = \lambda, \quad u_L(e_R) = \lambda, \quad u_L(e_L) = q^{-1}\lambda\,.
    \end{align*}
    \begin{proof}
        The first two identities follow directly from the above expression for $u_L$ and $e_L$. The third one is similar:
        \vspace{.4cm}
         \begin{align*}
    \includegraphics[scale=.4, valign = c]{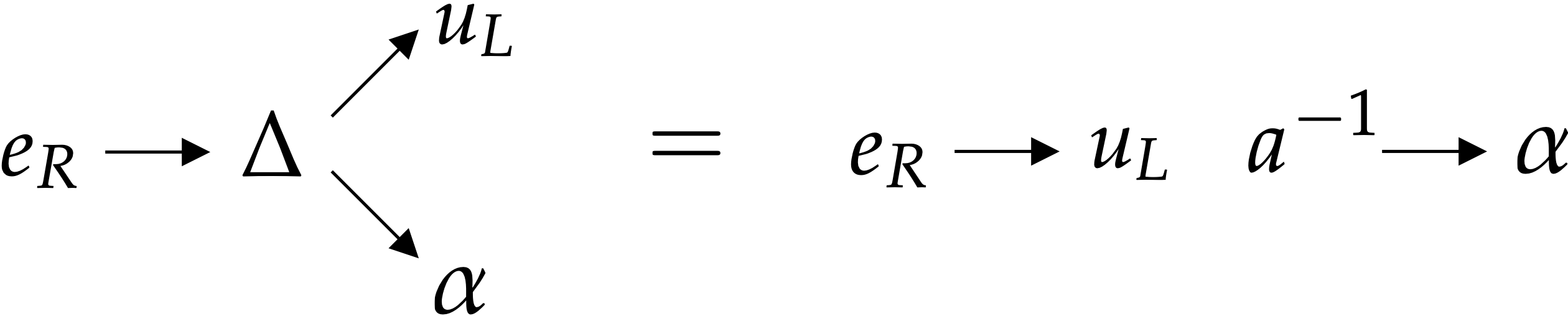}
    \end{align*}
    $$$$
        where the equality is due to Lemma \ref{lem:phase_from_left_integral}.
    \end{proof}
\end{lemma}

\begin{lemma}\label{lem:cophase_conjugation}
    Let $\alpha \in H^*$ be the cophase element. The map
    \vspace{.4cm}
    \begin{align*}
    \includegraphics[scale=.4, valign = c]{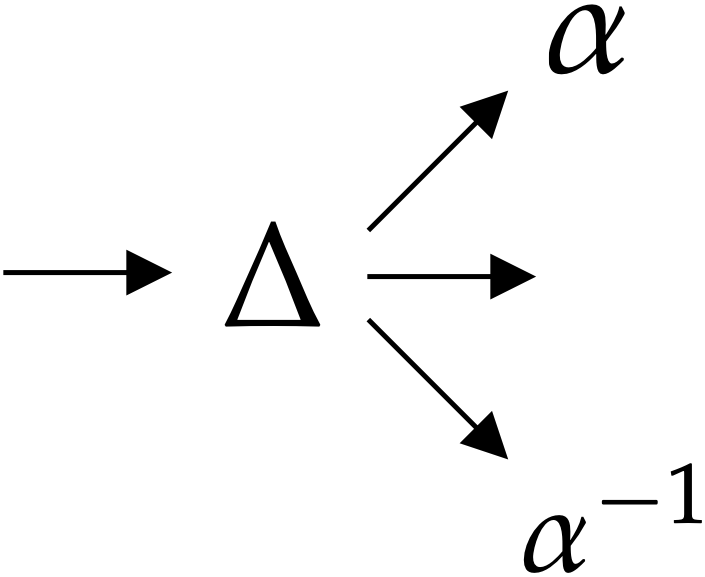}
    \end{align*}
    $$$$
    is a Hopf algebra automorphism. Furthermore, it has eigenvalue $q$ on all generalized integrals and generalized cointegrals. 
    \begin{proof}
        Denote the map in question by $f$.
        It is straightforward to check $f$ is a Hopf algebra automorphism, and we leave it as an exercise. For the second part, it suffices to check the statement for the right cointegral $e_R$ and the right integral $u_R$ by noting that $f$, as an automorphism, fixes the phase and cophase elements. 

        To check the case for $e_R$, it suffices to prove $u_R \circ f (e_R) = q u_R(e_R)$ since as an automorphism $f$ maps $e_R$ to a scalar multiple of it.  In fact, below we show $u_R \circ f  = q u_R$, which implies the statements for both $e_R$ and $u_R$:
        \vspace{.4cm}
          \begin{align*}
    \includegraphics[scale=.4, valign = c]{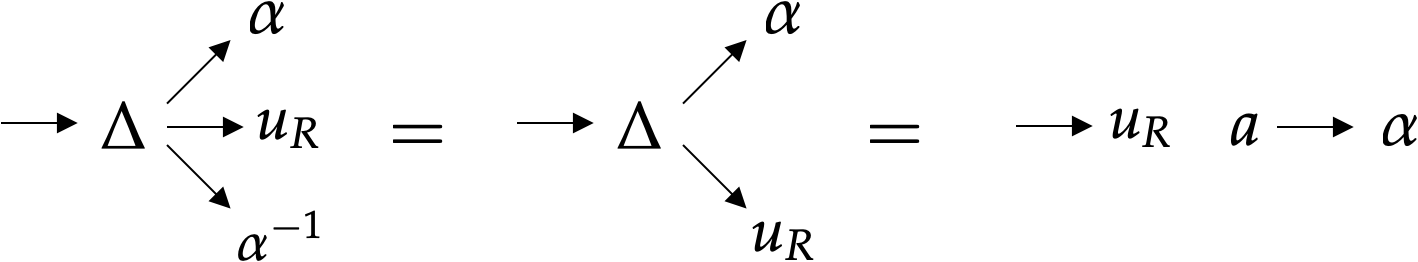}
    \end{align*}
    $$$$
    The first equality is by the definition of $u_R$ and the second is by the definition of $a$.
    \end{proof}
\end{lemma}
The antipode interchanges the left and right integrals/co-integrals with some factors.
\begin{lemma}\label{lem:antipode_integral}
    The following identities hold:
    \vspace{.4cm}
     \begin{align*}
    \includegraphics[scale=.4, valign = c]{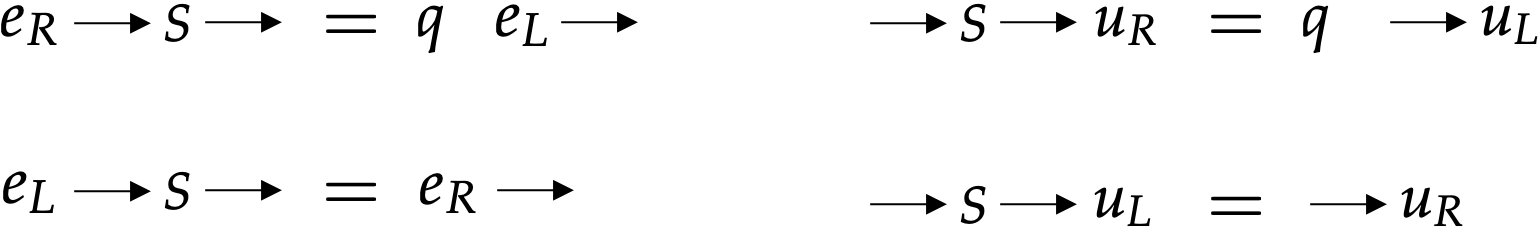}
    \end{align*}
    $$$$
    It follows that $S^2$ has eigenvalue $q$ on all generalized integrals and cointegrals.
    \begin{proof}
        It suffices to the prove the two identities involving cointegrals. Those involving integrals can be obtained from the corresponding identities of $H^*$. Since $S$ is an anti-algebra morphism, by the uniqueness of left/right cointegrals, $S$ automatically switches left and right cointegrals. We just need to calculate the relevant factors.

        We first verify $u_L \circ S (e_R) = q u_L(e_L)$, or equivalently $u_R(e_R)\; u_L \circ S (e_R) = q u_R(e_R)\; u_L(e_L) =  u_R(e_R)\; u_L(e_R)$ by Lemma \ref{lem:integral_cointegral}. Substituting $u_R(e_R)\;S$ with the tensor on the left hand side of the identity in Lemma \ref{lem:right_integral_cointegral_S}, we obtain
        \vspace{.4cm}
         \begin{align*}
    \includegraphics[scale=.4, valign = c]{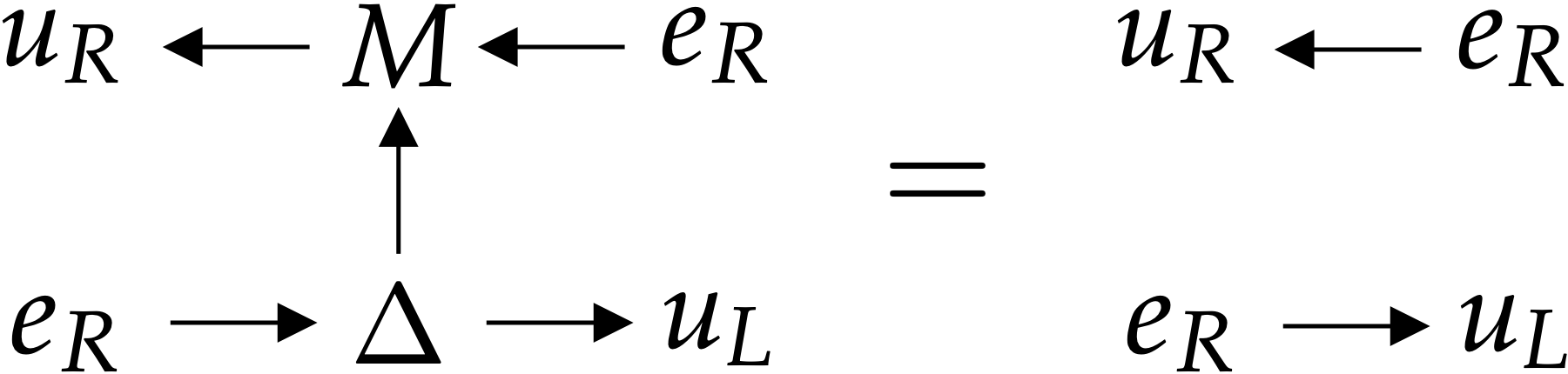}
    \end{align*}
        
    \end{proof}
\end{lemma}

Lemma \ref{lem:cophase_conjugation} and Lemma \ref{lem:antipode_integral} show that the conjugation map by $\alpha$ defined in Lemma \ref{lem:cophase_conjugation} and $S^2$ are ``almost indistinguishable'' from one another. They are both Hopf algebra automorphisms. They have the same action on all generalized integrals, generalized cointegrals, the phase element, and the cophase element. This motivates the definition of the tensor $T$, called the \emph{tilt} of $H$ \cite{kuperberg1996noninvolutory},
 \begin{align}
 \label{E:tilt1}
    \includegraphics[scale=.4, valign = c]{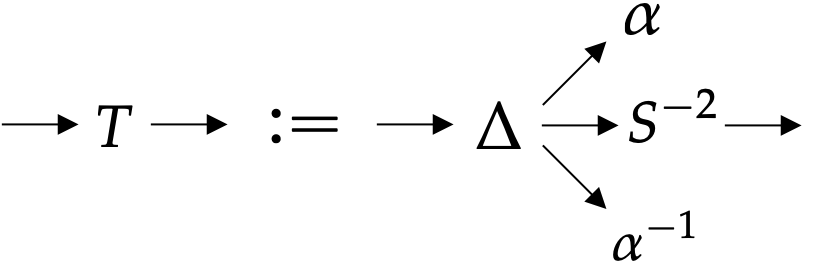}
 \end{align}
 The following property for $T$ is immediate. 
\begin{cor}\label{cor:tilt}
    The tilt $T$ is a Hopf algebra automorphism. Moreover, $T$ fixes all generalized integrals, generalized  cointegrals, the phase element, and the cophase element.
\end{cor}

\noindent The right integral is in general not co-commutative. The following lemma characterizes to what extent it fails to be so.
\begin{lemma}\label{lem:eR_cocommutativity}
    The right integral $e_R$ satisfies
     \begin{align*}
    \includegraphics[scale=.4, valign = c]{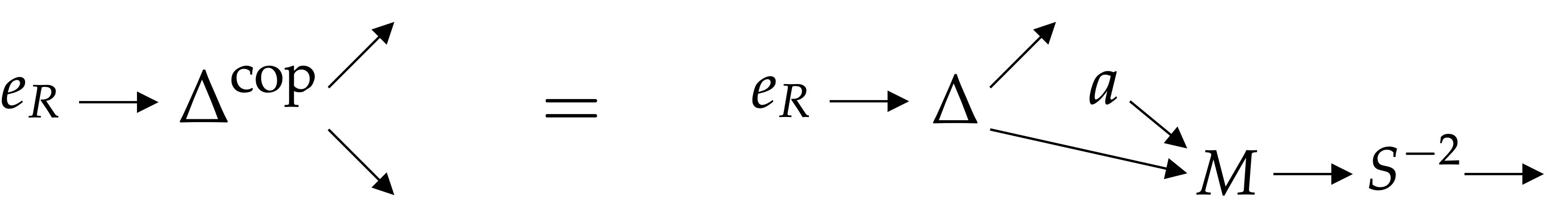}
    \end{align*}
    
    \begin{proof}
     By Lemma \ref{lem:right_integral_cointegral_S}, the tensor
    \begin{align*}
    \includegraphics[scale=.4, valign = c]{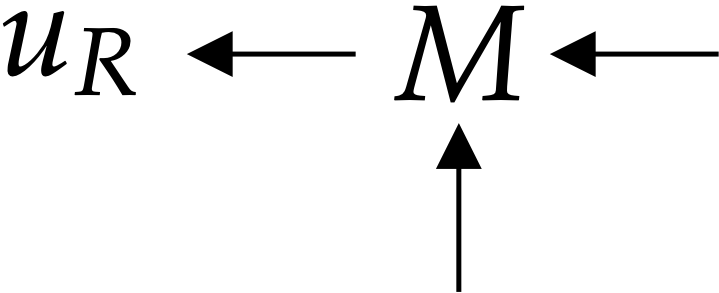}
    \end{align*}
    is non-degenerate. By contracting this non-degenerate tensor with the putative identity we desire to prove, we obtain an equivalent identity
    \vspace{.4cm}
     \begin{align}\label{eqn:lem:eR_cocommutativity_temp1}
    \includegraphics[scale=.4, valign = c]{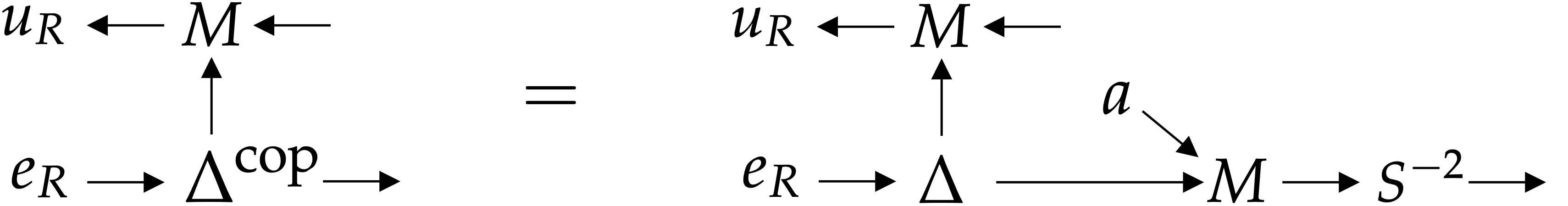}
    \end{align}
    $$$$
     Choose $u_R$ so that $u_R(e_R) = 1$. Again by Lemma \ref{lem:right_integral_cointegral_S}, the right hand side of Equation~\eqref{eqn:lem:eR_cocommutativity_temp1} becomes
     \vspace{.4cm}
     \begin{align}\label{eqn:lem:eR_cocommutativity_temp2}
    \includegraphics[scale=.4, valign = c]{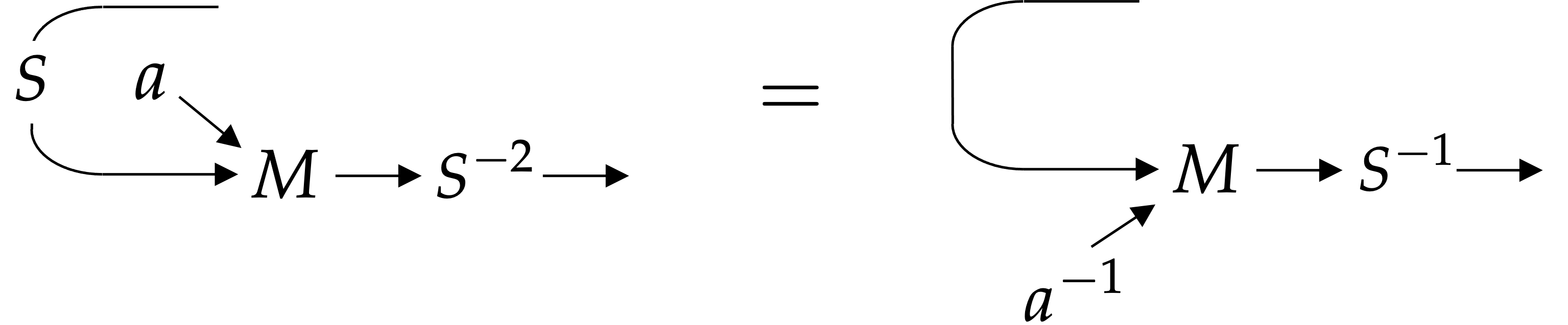}
    \end{align}
$$$$
     Now, the left hand side of Equation~\eqref{eqn:lem:eR_cocommutativity_temp1} equaling the right side of Equation~\eqref{eqn:lem:eR_cocommutativity_temp2} is equivalent to the following equation,
     \begin{align*}
    \includegraphics[scale=.4, valign = c]{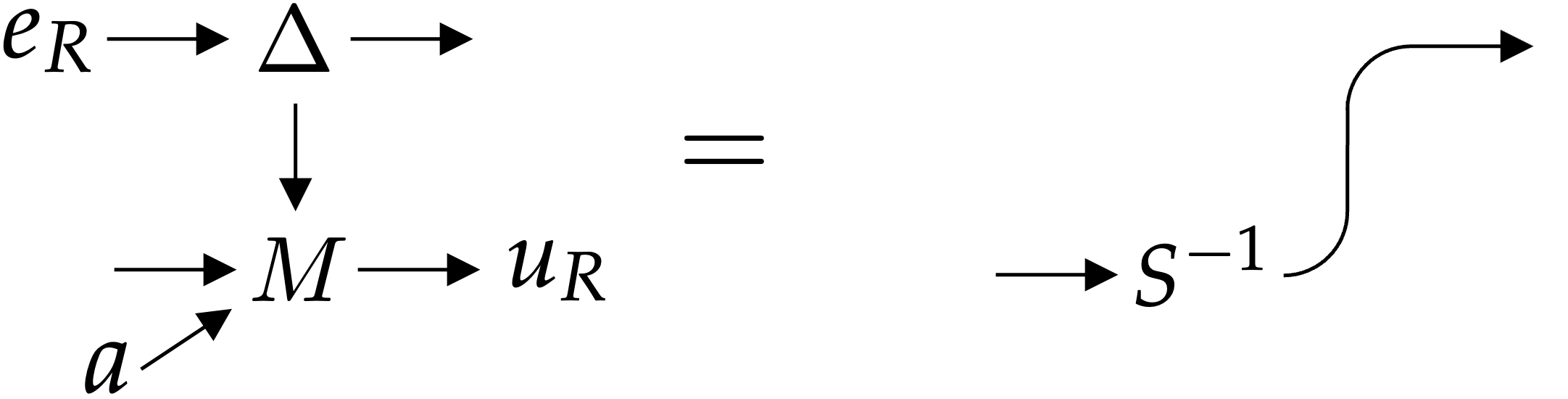}
    \end{align*}
    which is the same as Equation~\eqref{eqn:eR_muL_antipode}.
    \end{proof}
\end{lemma}
As an aside, it is a well-known and important fact that
\vspace{.4cm}
 \begin{align*}
    \includegraphics[scale=.4, valign = c]{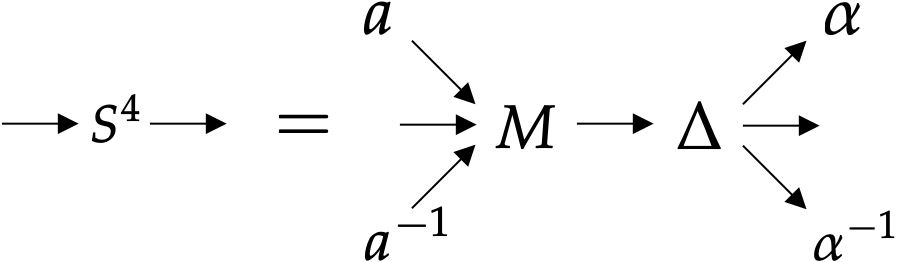}
 \end{align*}
although we will not use this identity in the present paper.
 
\begin{definition} A Hopf algebra $H$ is \emph{balanced} if the tilt map $T$ defined by~\eqref{E:tilt1} is the identity map. As a special case, $H$ is \emph{involutory} if $S^2 = \op{Id}$. \end{definition}
\noindent For a balanced Hopf algebra, $S^2$ equals the conjugation by the phase element, as well as the conjugation by the cophase element.

\subsection{More properties of the generalized cointegrals}

We need a few  more lemmas on the generalized cointegrals that will be used in the proof of the 4-manifold invariants.
\begin{lemma}
\label{lemma:etheta1}
    For any proper half-integer $\theta$, we have
    \vspace{.2cm}
    \begin{align*}
    \includegraphics[scale=.4, valign = c]{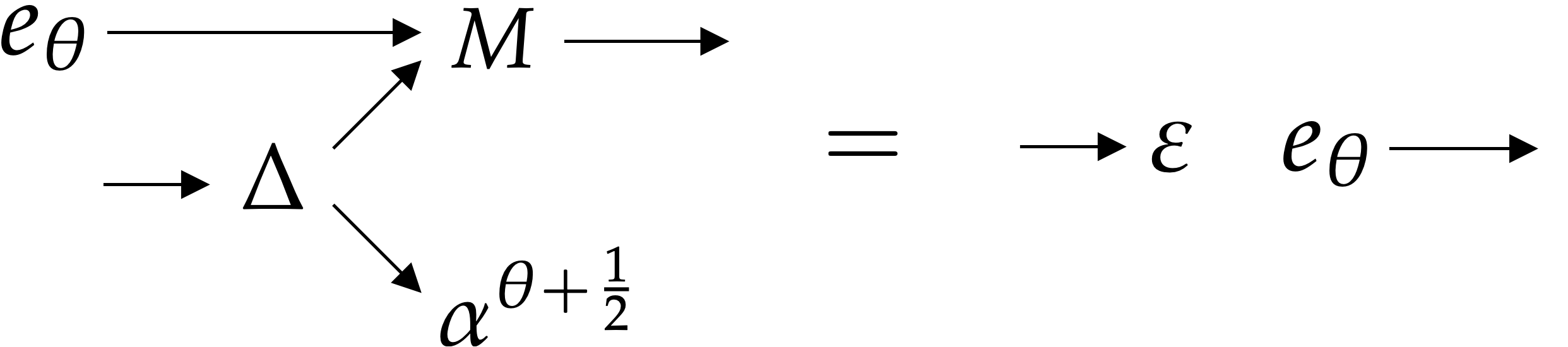}
    \end{align*}
    \begin{proof}
    Let $\theta = n - \frac{1}{2}$. By expanding $e_{\theta}$ explicitly, the left-hand side of the above equation becomes
    \vspace{.4cm}
    \begin{align*}
    \vspace{.4cm}
    \includegraphics[scale=.4, valign = c]{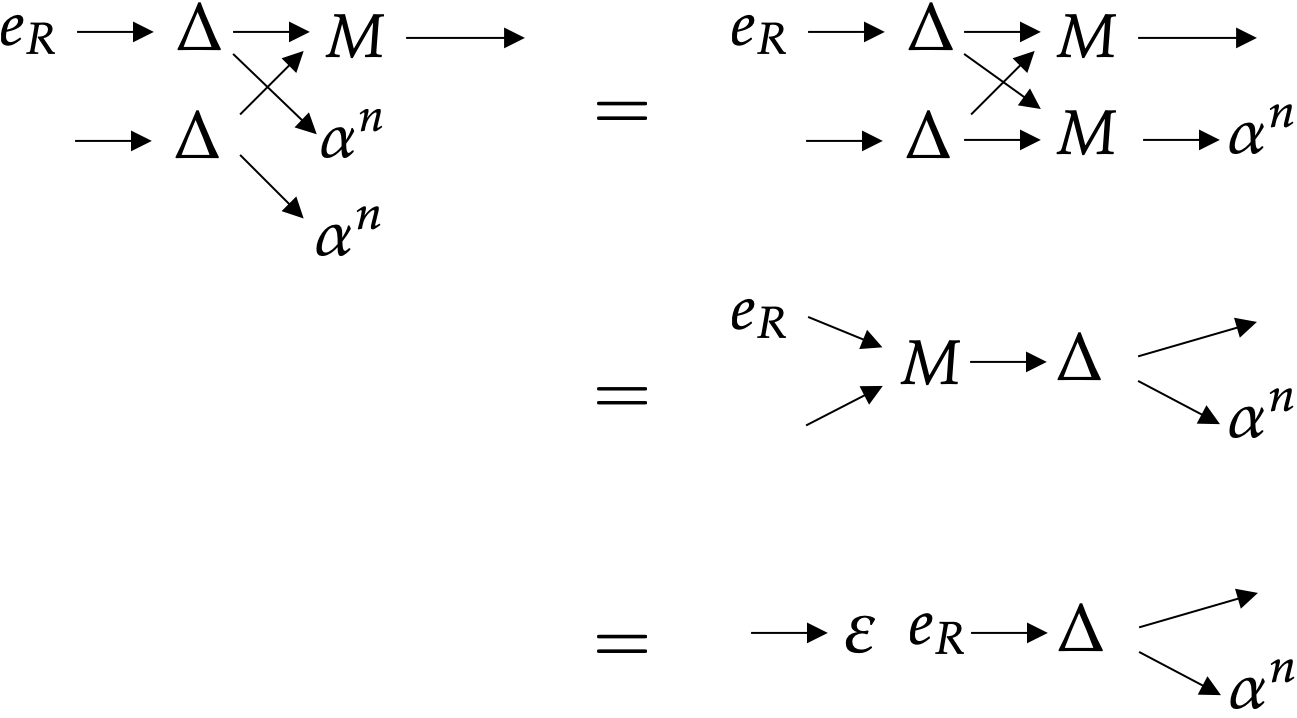}
    \end{align*}
    $$$$
    The first equality above follows from the fact that $\alpha^n$ is an algebra morphism. The second equality is a direct application of the compatibility condition between $\Delta$ and $M$. The last equality follows by definition of a right integral. 
    \end{proof}
\end{lemma}

The following result is about the co-commutativity of the generalized co-integrals extending Lemma \ref{lem:eR_cocommutativity}.
\begin{lemma}\label{lem:e_theta_cocommutativity}
    For any proper half-integer $\theta$, we have
    \begin{align*}
    \includegraphics[scale=.4, valign = c]{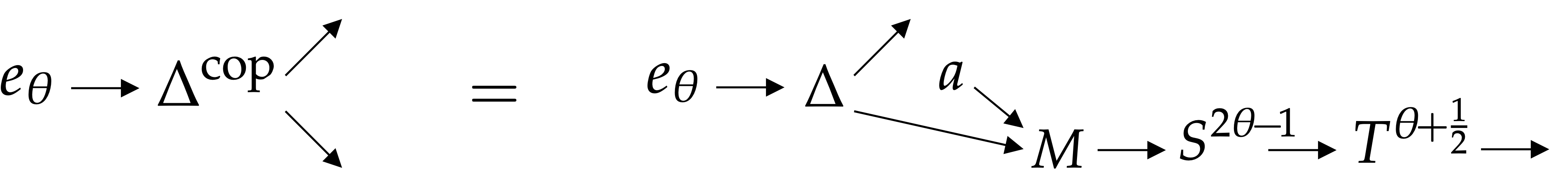}
\end{align*}
\begin{proof}
     Let $\theta = n -\frac{1}{2}$. The left hand side of the above identity can be expanded as
     \vspace{.2cm}
    \begin{align*}
    \includegraphics[scale=.4, valign = c]{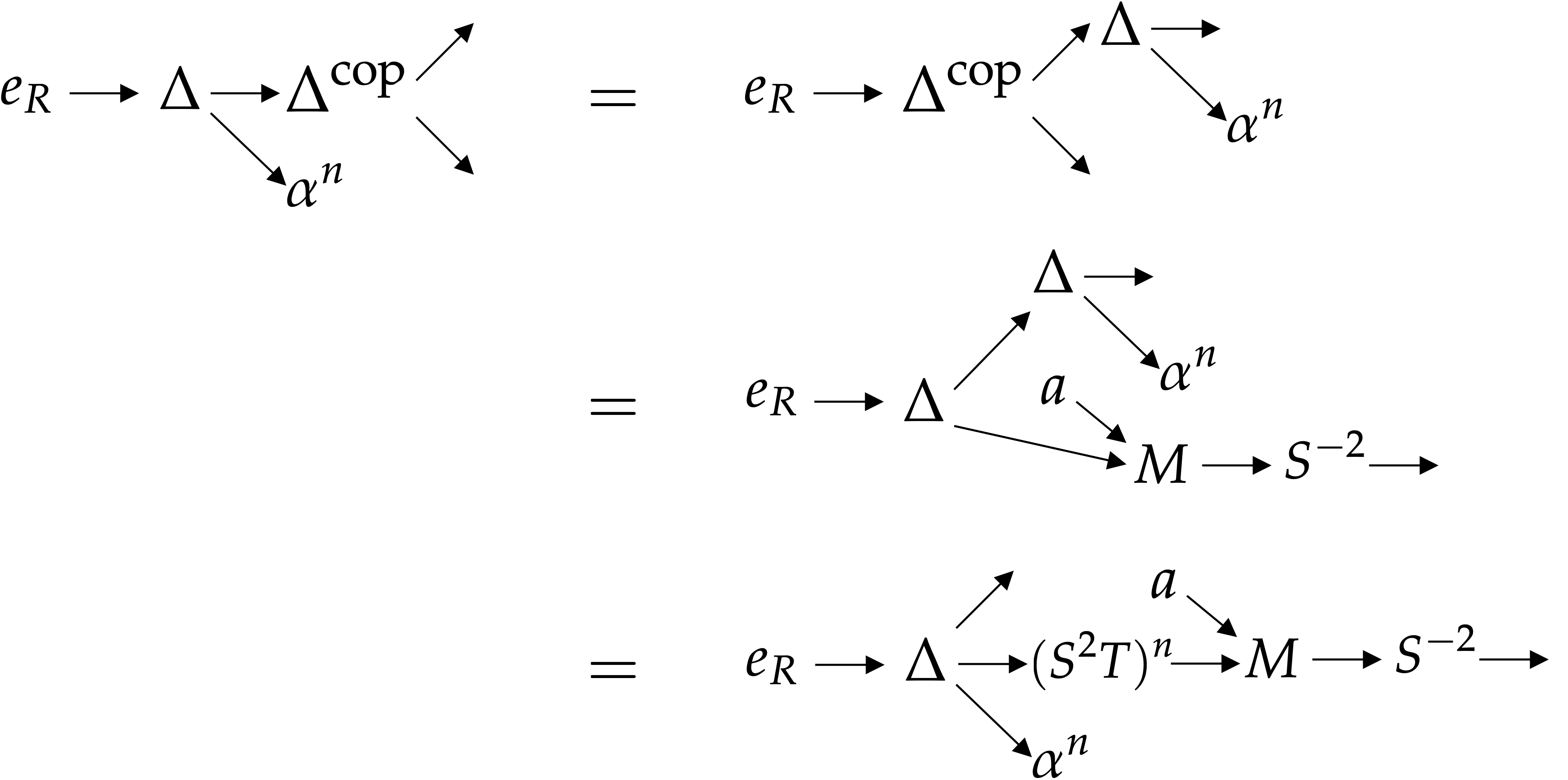}
\end{align*}
The second equality above is by Lemma \ref{lem:eR_cocommutativity}. The last equality follows from the definition of the tilt $T$. By Corollary \ref{cor:tilt}, $T$ fixes $a$ and commutes with the product and the antipode. The same holds for $S^2$. Then the last tensor diagram above reduces to the right-hand side of the desired identity. 

\end{proof}
\end{lemma}

\begin{lemma}\label{lem:S_fli_e_theta}
    For any proper half-integer $\theta$, we have
    \begin{align*}
    \includegraphics[scale=.4, valign = c]{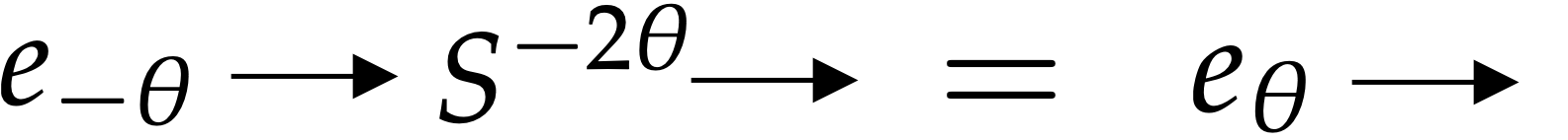}
\end{align*}
\begin{proof}
Let $\theta = n -\frac{1}{2}$. By Lemma \ref{lem:antipode_integral}, $S^{-2n}$ has eigenvalue $q^{-n}$ on the generalized co-integral $e_{-n + \frac{1}{2}}$. Hence, the left hand side of the above identity can be expanded as below with subsequent derivations; explanations of the derivations are provided at the end:
\vspace{.4cm}
\begingroup
\allowdisplaybreaks
    \begin{align*}
    &\includegraphics[scale=.4, valign = c]{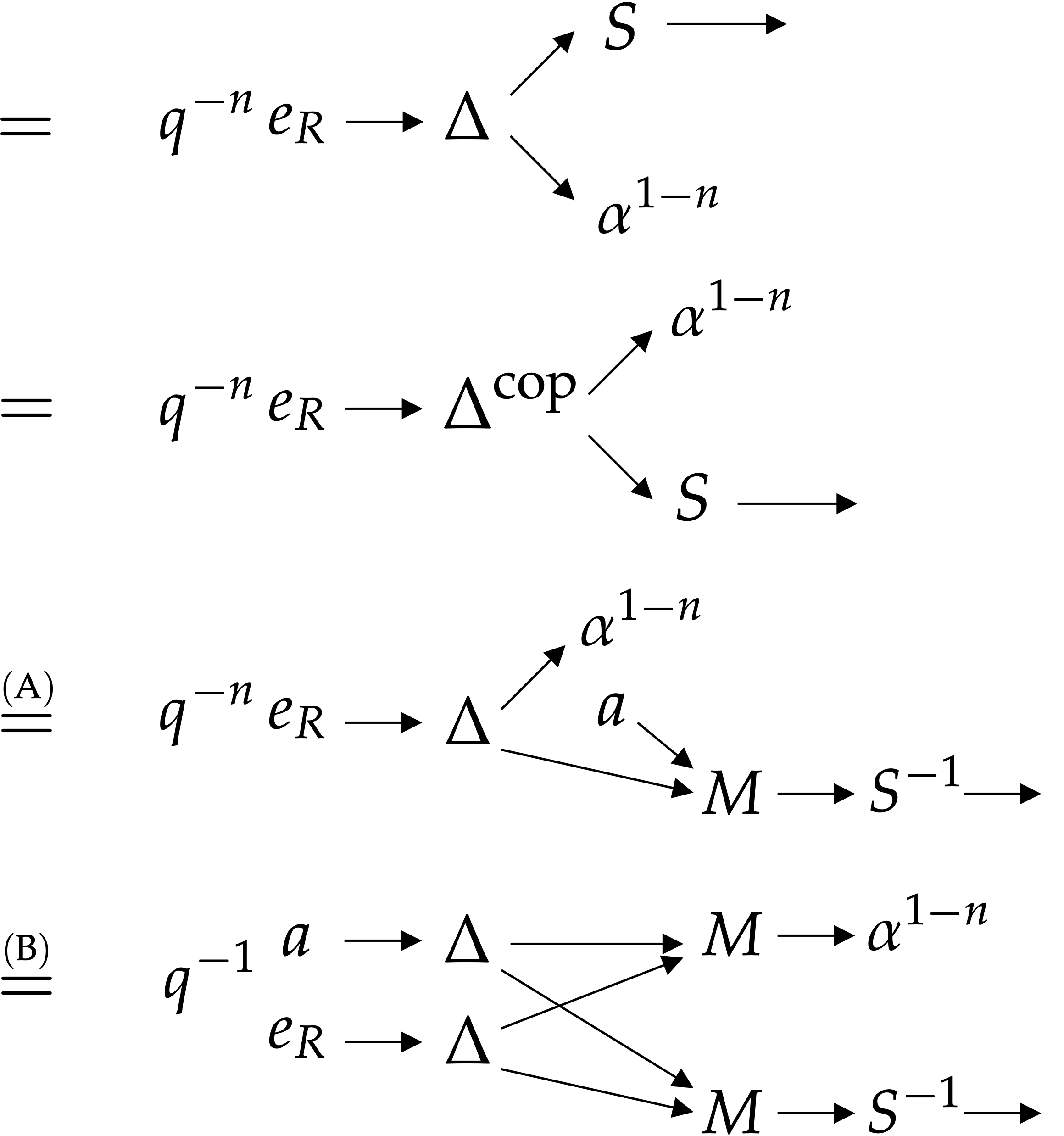}\\
    &\includegraphics[scale=.4, valign = c]{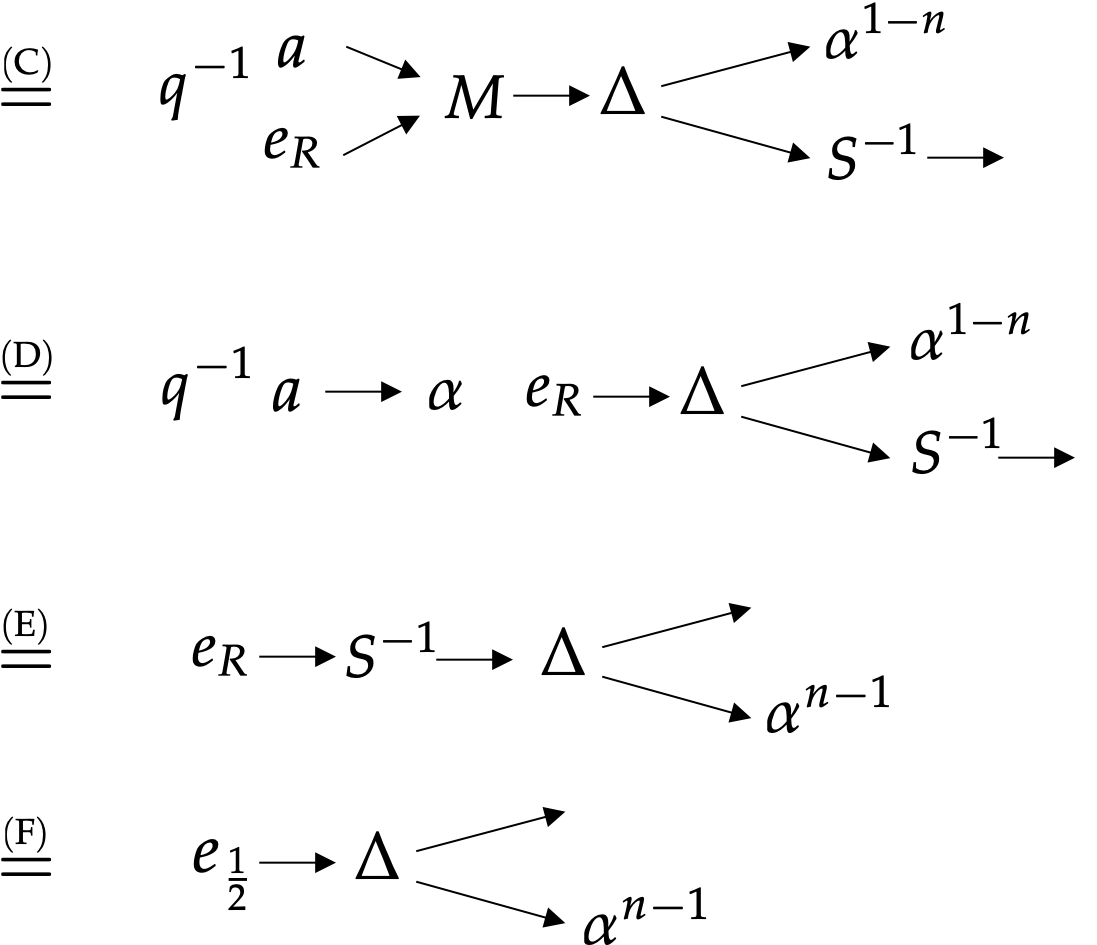}
\end{align*}
\endgroup
$$$$
Identity (A) is by Lemma \ref{lem:eR_cocommutativity}. Identity (B) uses the facts that $\alpha^{1-n}$ is an algebra morphism, $a$ is group-like, and $q = \alpha(a)$. Identity (C) is due to the compatibility condition between $\Delta$ and $M$. Identity (D) is by Definition~\ref{def:phase_cophase}. Identity (E) follows from the facts that $S^{-1}$ is an anti-coalgebra morphism and $\alpha^{-n} \circ S^{-1} = \alpha^{n}$. Identity (F) is by Lemma \ref{lem:antipode_integral}. The right hand side of Identity $(F)$ is apparently $e_{\theta}$.
\end{proof}

\end{lemma}

\subsection{Drinfeld pairings} We next discuss a certain type of pairing on Hopf algebras.

\begin{definition} \label{def:hopf_pairing} A \emph{Drinfeld pairing} on two Hopf algebras $G$ and $H$ is a bilinear form $P: G \otimes H \to k$
\begin{equation*}
\xblackrightarrow[\footnotesize{$G$}]{P} \xblackleftarrow[\footnotesize{$H$}]{}
\end{equation*}
such that the induced map $G \to H^{*,\op{cop}}$ is a Hopf algebra morphism, or equivalently, the induced map $H \to G^{*, \op{op}}$ is a Hopf algebra morphism. 
\end{definition}
In this paper, given two Hopf algebras $G$ and $H$, we will not simultaneously consider more than one Drinfeld pairing on them, and hence a pairing $P: G \otimes H \to k$ is simply represented by
\vspace{.4cm}
\begin{align*}
    \includegraphics[scale=.4, valign = c]{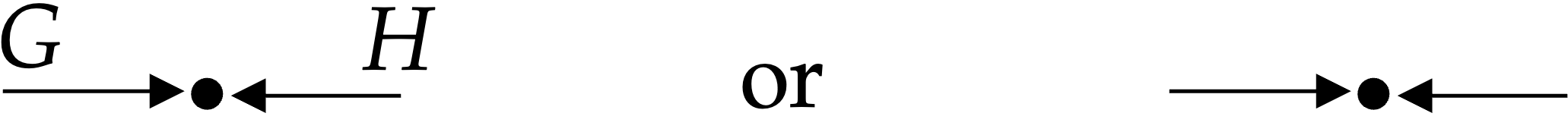}
\end{align*}
$$$$
when no confusion arises. We also add a subscript to structure tensors (e.g.~$\Delta_{H}$) to indicate which Hopf algebra they are from, and as always, omit the subscript whenever it is clear from the context. Unpacking the definition of a Drinfeld pairing leads to the following identities of tensor diagrams:
\vspace{.4cm}
\begin{align*}
    \includegraphics[scale=.4, valign = c]{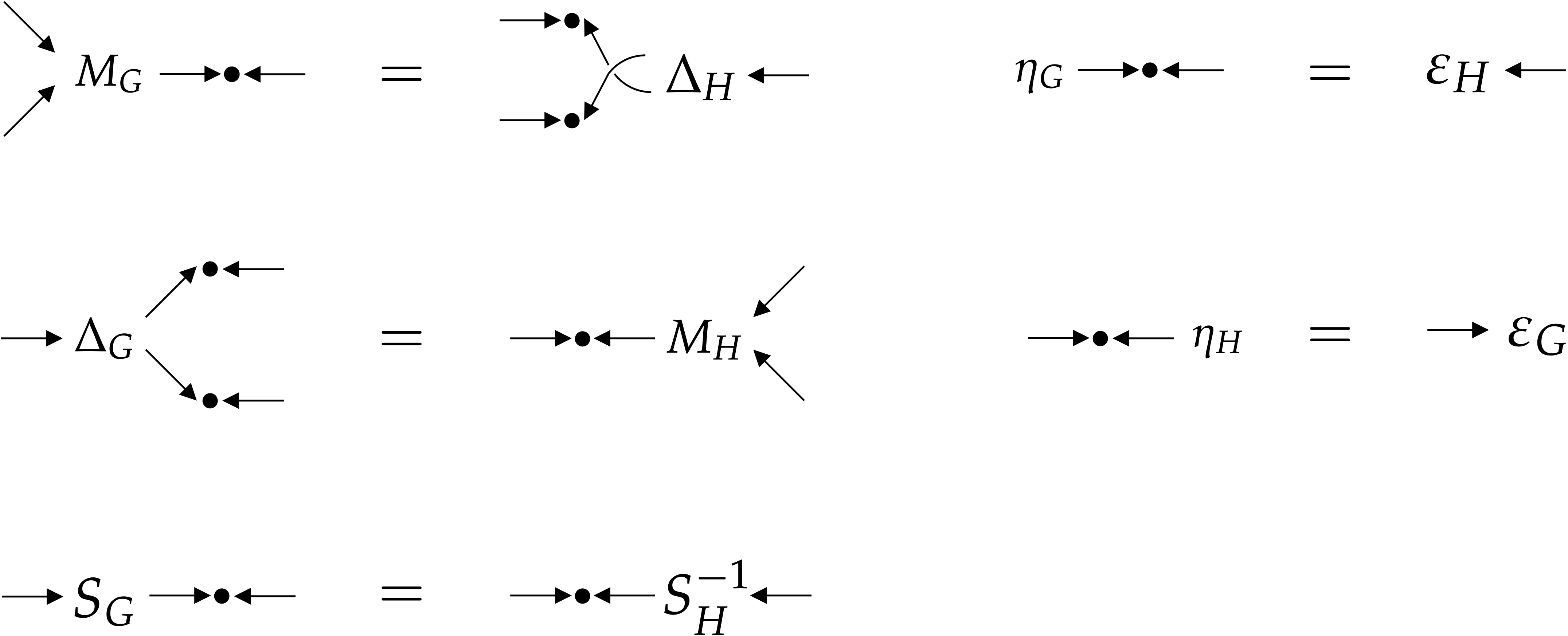}
\end{align*}
$$$$

\noindent The following lemma is straightforward.
\begin{lemma} Let $P$ be a Drinfeld pairing for the pair $(G,H)$. Then $P$ is also a Drinfeld pairing for
\[
(G^{\op{op}},H^{\op{cop}})\,, \qquad (G^{\op{cop}},H^{\op{op}})\,, \qquad (H^{\op{op},\op{cop}},G)  
\]
\end{lemma}

\begin{definition}[Drinfeld Double]\label{def_generalized_double}
The \emph{Drinfeld double} associated to a Drinfeld pairing $P$ for $G$ and $H$ is the Hopf algebra
\[
D(G,H;P) \qquad \text{or more simply} \qquad D_P
\]
defined as follows. As a coalgebra, $D_P$ is $G \otimes H$ with the tensor product coalgebra structure. That is, the coproduct and counit are given by
\begin{align*}
    \includegraphics[scale=.4, valign = c]{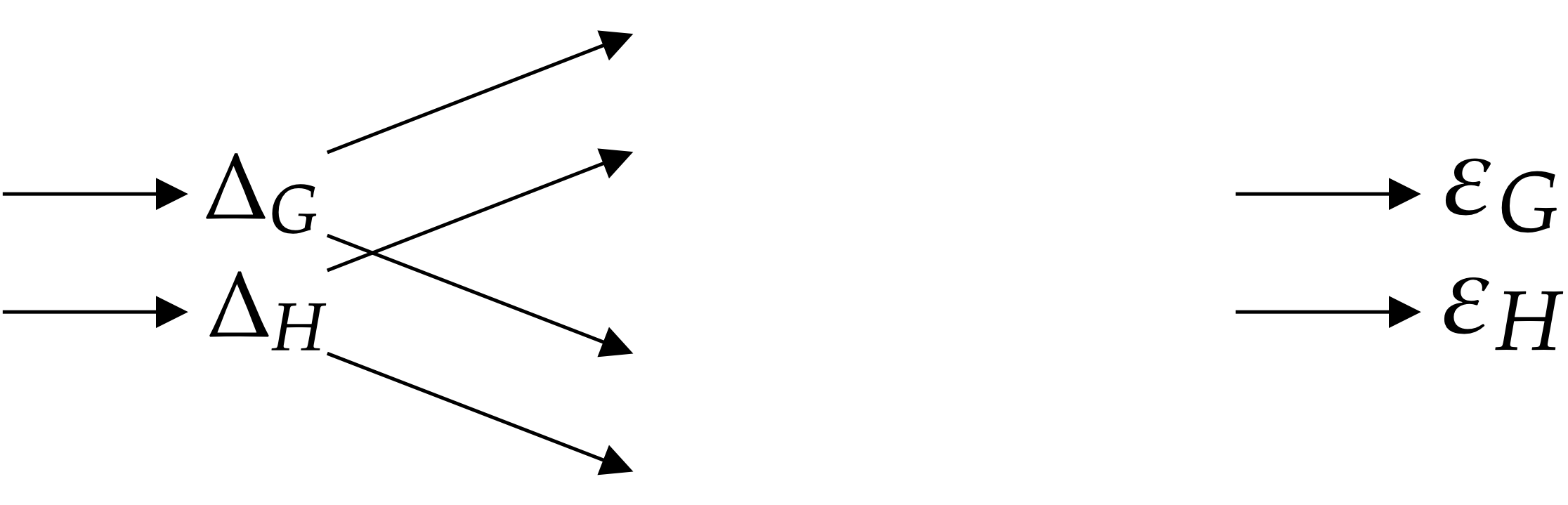}
\end{align*}
Denote by $U: H \otimes G \to G \otimes H$ the tensor
\vspace{.4cm}
\begin{align*}
    \includegraphics[scale=.4, valign = c]{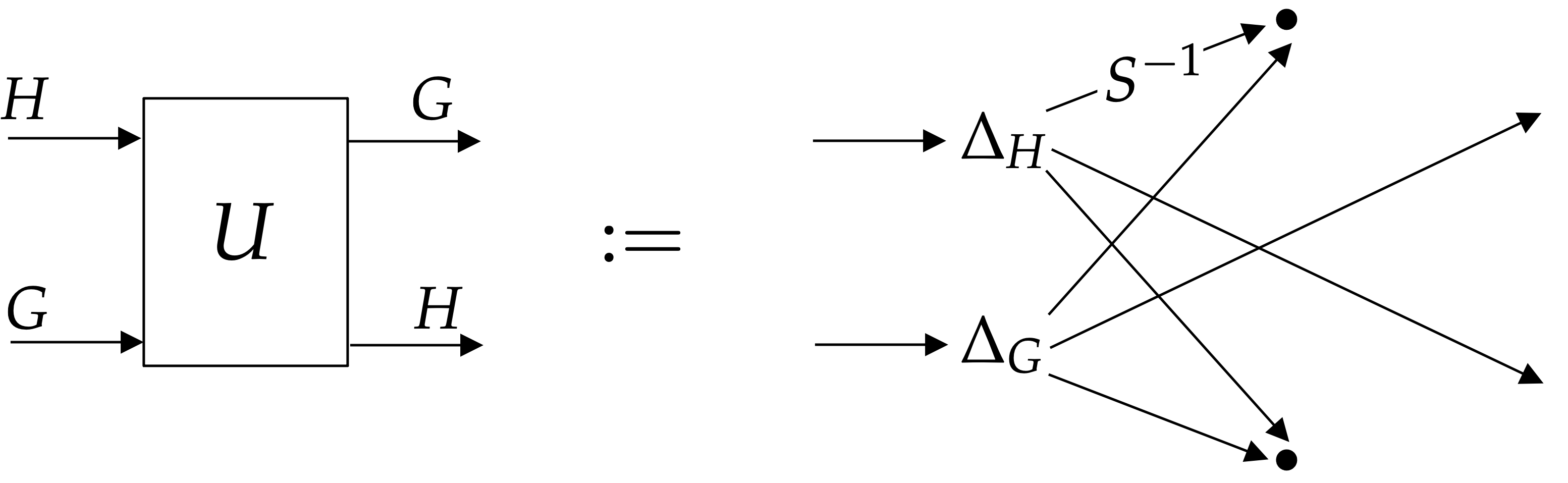}
\end{align*}
$$$$
Then the product and unit on $D_P$ are given by
\vspace{.4cm}
\begin{align*}
    \includegraphics[scale=.4, valign = c]{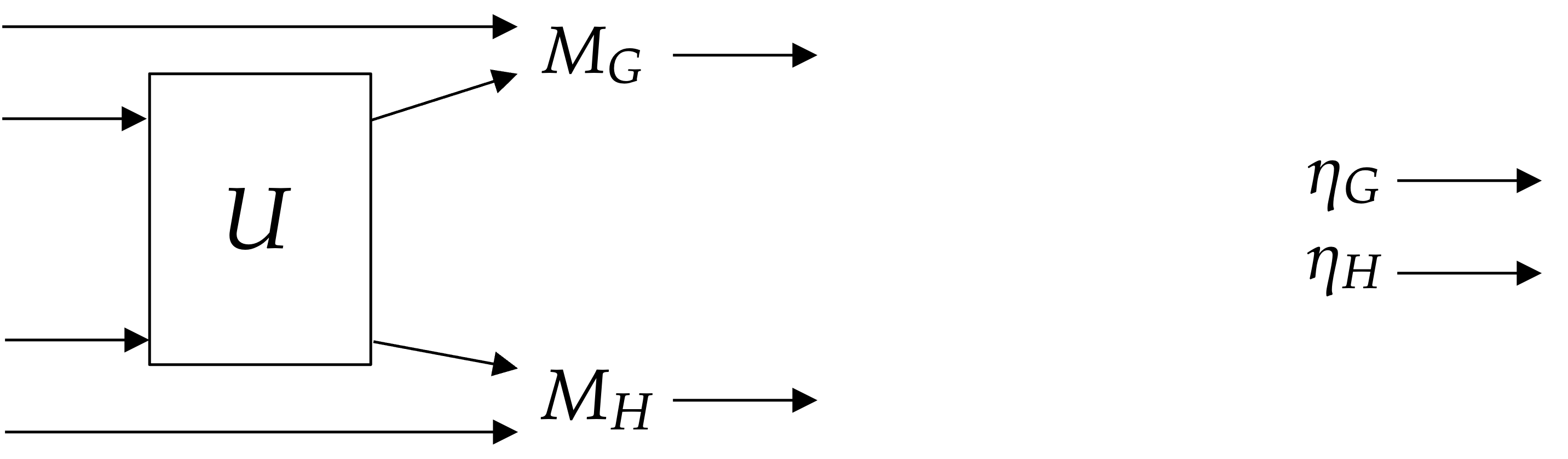}
\end{align*}
$$$$
Lastly, the antipode is
\vspace{.4cm}
\begin{align*}
    \includegraphics[scale=.4, valign = c]{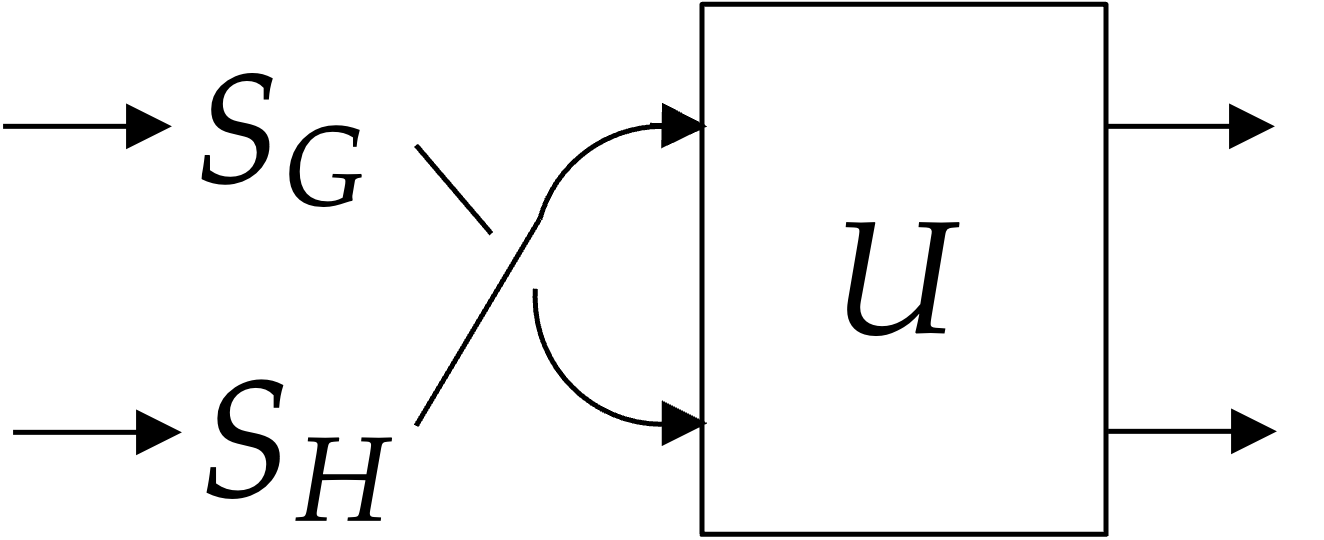}
\end{align*}
$$$$
\end{definition}

It can be checked the embeddings $G \to D_P$ mapping $g$ to $g \otimes 1$ and  $H \to D_P$ mapping $h$ to $1 \otimes h$ are Hopf algebra morphisms. We identify $G$ and $H$ with their images in $D_P$ by the above maps. Under these identifications, for $g \in G,\ h \in H$, 
\begin{align*}
    g \cdot h  &= g \otimes h \\
    h \cdot g  &= U(h \otimes g)
\end{align*}
 Hence, as an algebra, $D_P$ is generated by the subalgebras $G$ and $H$ subject  to the above relations.

 \begin{example}[Drinfeld double of a Hopf algebra]
     Given a Hopf algebra $H$, the canonical pairing $P_{\text{can}}$ on $H^{*,\op{cop}} \otimes H$ is a Drinfeld pairing, and $D(H^{*,\op{cop}} , H; P_{\text{can}}) = H^{*, \op{cop}} \otimes H$ is the usual  Drinfeld double of $H$.
 \end{example}

 \begin{example}\label{example:quasi_triangular_pairing}
     Let $H$ be a quasi-triangular Hopf algebra with the $R$-matrix $R \in H \otimes H$ which is represented by the tensor
      \begin{align*}
    \includegraphics[scale=.4, valign = c]{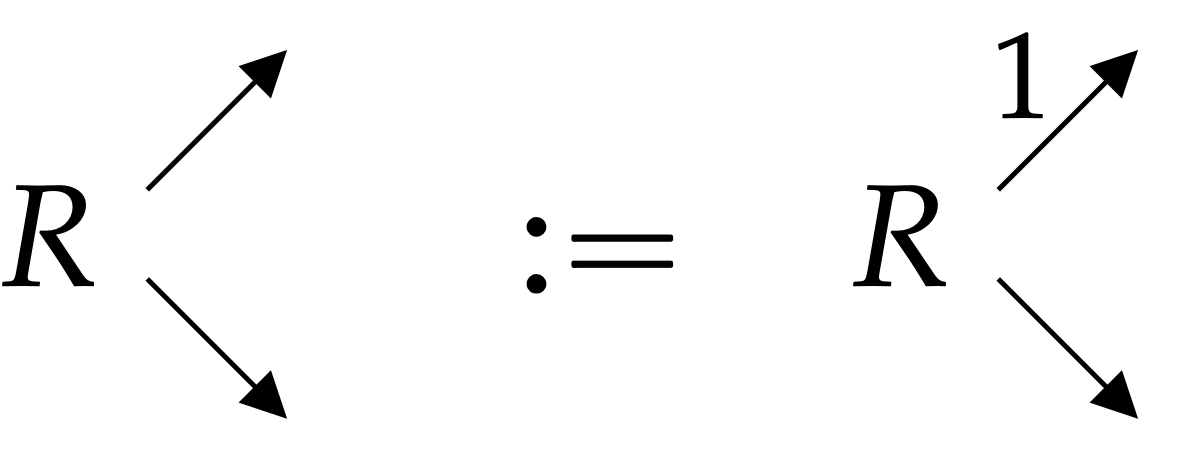}
     \end{align*}
     The edge labelled by $1$ denotes the first factor of $R$. Here we will only draw the $R$-tensor in a similar manner as above and hence will omit the label `1'. With this notation, the $R$-matrix satisfies the identities
     \vspace{.4cm}
     \begin{align}\label{eqn:R-matrix_condition1}
    \includegraphics[scale=.4, valign = c]{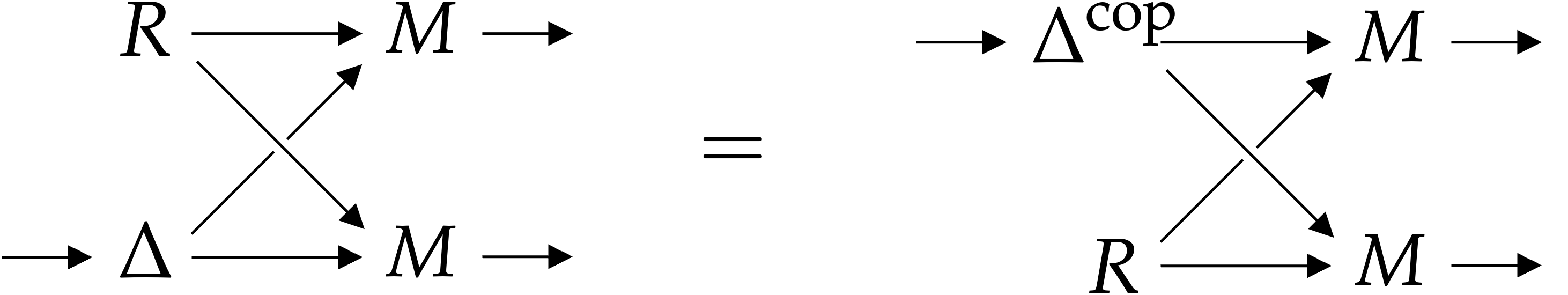}
     \end{align}
      \begin{align*}
    \includegraphics[scale=.4, valign = c]{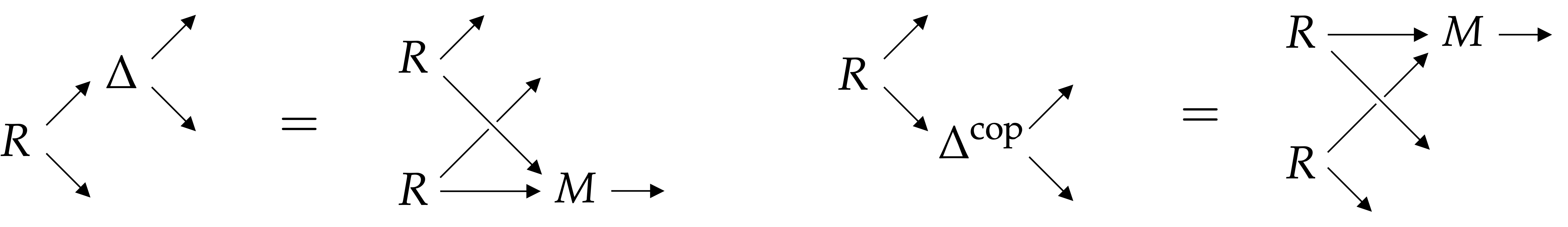}
     \end{align*}
     $$$$
     Define the map $f_R: H^* \to H^{\op{cop}}$ by $f_R(p):= (p \otimes \text{Id})(R)$. By  Proposition $12.2.11$ of \cite{radford2011hopfalgebras}, $f_R$ is a Hopf algebra morphism. Denote by $R^*$ the bilinear form on $H^* \otimes H^*$ defined by $R^*(p \otimes q):= (p \otimes q)(R)$. Then $R^*$ is a Drinfeld pairing on $H^* \otimes H^*$.
 \end{example}

\section{Hopf Triplets} \label{sec:Hopf_triplets}
We now introduce Hopf triplets which are the algebraic data to define the invariants.

\begin{definition}[Hopf Triplet] \label{def:hopf_triplet}
A \emph{Hopf triplet} $\mathcal{H} = (H_\alpha,H_\beta,H_\kappa,\langle-\rangle)$ over a field $k$ consists of three  Hopf algebras $H_\alpha,\ H_\beta$, and $H_\kappa$, and three pairings
\[
\langle-\rangle_{\alpha\beta}:H_\alpha \otimes H_\beta \to k \qquad \langle-\rangle_{\beta\kappa}:H_\beta \otimes H_\kappa \to k \qquad \langle-\rangle_{\kappa\alpha}:H_\kappa \otimes H_\alpha \to k \qquad
\]
whose tensor diagrams are denoted, respectively, by
\begin{align*}
    \includegraphics[scale=.4, valign = c]{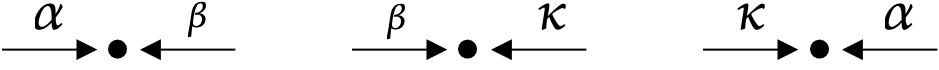}
\end{align*}
such that the following conditions hold.
\begin{itemize}[leftmargin=.2in]
    \item[(a)] For $\mu\nu \in \{\alpha\beta, \beta\kappa, \kappa\alpha\}$, $\langle-\rangle_{\mu\nu}$ is a Drinfeld pairing; that is, $\langle-\rangle_{\mu\nu}$ induces a Hopf algebra morphism $H_{\mu} \to H_{\nu}^{*, \text{cop}}$.
    \item[(b)] The above induced map $H_{\mu} \to H_{\nu}^{*, \text{cop}}$ and its dual $H_{\nu}^{\text{op}} \to H_{\mu}^{*}$ preserve the phase element,  
    \begin{align*}
    \includegraphics[scale=.4, valign = c]{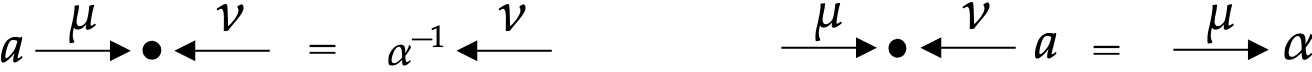}
    \end{align*}
    \item[(c)] The following equality holds:
    \begin{align}\label{eqn:triplet_def_triangle_1}
    \includegraphics[scale=.4, valign = c]{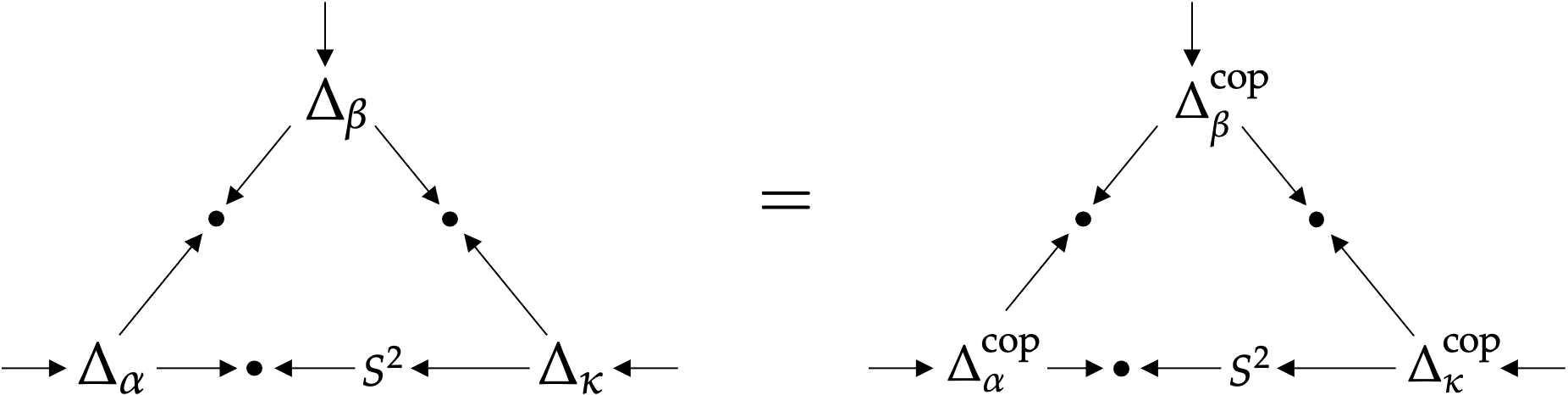}
    \end{align}
\end{itemize}
\end{definition}
\begin{remark}
    If all three Hopf algebras of $\mathcal{H}$ are involutory (i.e., semisimple), then Part (b) is automatically satisfied, and the above definition reduces to the definition of involutory Hopf triplet in \cite{chaidez20194manifold}. 
\end{remark}
\begin{remark}
    In Part (c) of the above definition and throughout the paper, we drop the indices from certain tensors whenever they are clear from the context. Also, the equality in Part (c) seems asymmetric in the sense that only the $S^2$ tensor from $H_{\kappa}$ is involved. However, the lemma below shows otherwise.
\end{remark}
\begin{lemma}\label{lem:triplet_triangle_condition}
    Assuming the condition in Part (a) of Definition \ref{def:hopf_triplet} holds, then Equation~\eqref{eqn:triplet_def_triangle_1} in Part (c) is equivalent to the identity below which is formally obtained from Equation~\eqref{eqn:triplet_def_triangle_1} by cyclically permuting the indices of all involved tensors in the order $\alpha \mapsto \beta \mapsto \kappa \mapsto \alpha$\,:
    \begin{align*}
    \includegraphics[scale=.4, valign = c]{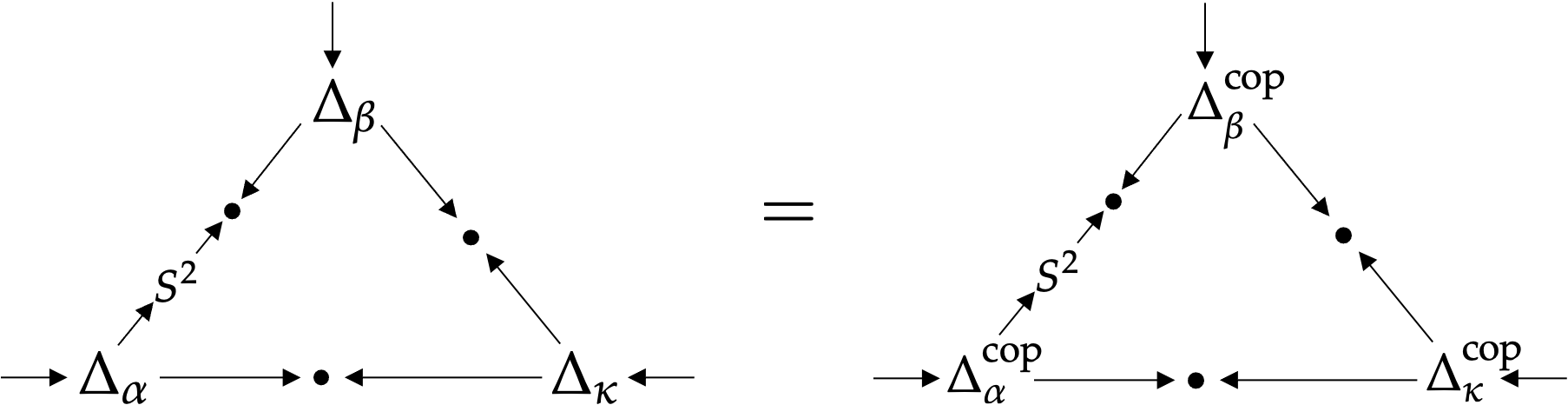}
    \end{align*}
    By permuting the indices once again, one obtain a third identity which is also equivalent to  Equation~\eqref{eqn:triplet_def_triangle_1}.
    \begin{proof}
    The proof is straightforward by using the conditions that the pairings induce Hopf algebra morphisms and that $S^2$ is a Hopf algebra isomorphism.
    \end{proof}
\end{lemma}
We call Equation~\eqref{eqn:triplet_def_triangle_1} the Triangle Identity. Lemma \ref{lem:triplet_triangle_condition} shows there is a symmetry to the Triangle Identity by cyclically permuting the indices $\alpha \mapsto \beta \mapsto \kappa \mapsto \alpha$.
Below we provide some additional equivalent formulations of the Triangle Identity which will be useful in the study of the 4-manifold invariants.

Still, assume the condition in Part (a) of Definition \ref{def:hopf_triplet} holds for the tuple of Hopf algebras $(H_\alpha,H_\beta,H_\kappa,\langle-\rangle)$. Then, the pairing-induced map $H_{\alpha} \to H_{\beta}^{*, \text{cop}}$ can also be thought of as a Hopf algebra morphism,
\begin{align}
    \langle-\rangle_{\alpha\beta}\colon &H_{\alpha}^{\text{op}} \to H_{\beta}^{*, \text{cop},\text{op}} = \left(H_{\beta}^{\text{cop}}\right)^{*, \text{cop}},
\end{align}
or equivalently, 
\begin{align}\label{eqn:triplet_alpha_beta}
    \langle-\rangle_{\alpha\beta}\colon &H_{\alpha}^{\text{op}} \otimes H_{\beta}^{\text{cop}} \to k
\end{align}
is a Drinfeld pairing. Hence, the Drinfeld double $D(H_{\alpha}^{\text{op}}, H_{\beta}^{\text{cop}}, \langle-\rangle_{\alpha\beta})$ is defined (see Definition \ref{def_generalized_double}). Similarly, the map $H_{\beta} \to H_{\kappa}^{*, \text{cop}}$ induces the Hopf algebra morphism,
\begin{align}\label{eqn:triplet_beta_kappa}
    \langle-\rangle_{\beta\kappa}\colon &H_{\beta}^{\text{cop}} \to H_{\kappa}^{*},
\end{align}
and $H_{\kappa} \to H_{\alpha}^{*, \text{cop}}$ induces the Hopf algebra morphism,
\begin{align}
    \langle-\rangle_{\kappa\alpha}\colon &H_{\alpha}^{*, \text{cop},*} = H_{\alpha}^{\text{op}} \to H_{\kappa}^{*}.
\end{align}
Take the composition
\begin{align}\label{eqn:triplet_kappa_alpha}
    (S^2)^* \circ \langle-\rangle_{\kappa\alpha}\colon&H_{\alpha}^{\text{op}}\  \overset{\langle-\rangle_{\kappa\alpha}}{\longrightarrow} \  H_{\kappa}^{*}\  \overset{(S^2)^*}{\longrightarrow}\  H_{\kappa}^{*}
\end{align}
The two maps in Equations~\eqref{eqn:triplet_beta_kappa} and~\eqref{eqn:triplet_kappa_alpha} determine a linear map
\begin{align}
    \Phi_{\alpha\beta}\colon& D(H_{\alpha}^{\text{op}}, H_{\beta}^{\text{cop}}, \langle-\rangle_{\alpha\beta}) \to H_{\kappa}^{*},
\end{align}
which is defined by
\begin{align*}
    \Phi_{\alpha\beta}(x \otimes y) := \left((S^2)^* \circ \langle\cdot , x\rangle_{\kappa\alpha} \right) \bullet \left(\langle y, \cdot \rangle_{\beta\kappa}\right).
\end{align*}
In terms of the tensor diagrams, $\Phi_{\alpha\beta}$ is presented by
\begin{align*}
    \includegraphics[scale=.4, valign = c]{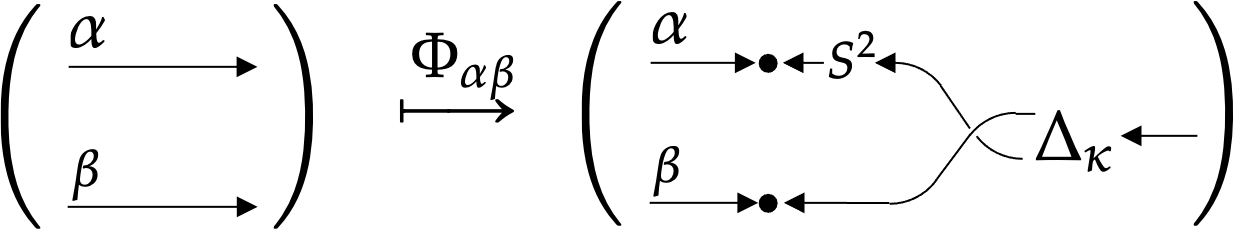}
\end{align*}
By the construction of the Drinfeld double, $\Phi_{\alpha\beta}$ is a coalgebra morphism and preserves the unit. Then by the cyclic symmetry of the indices, the coalgebra maps $\Phi_{\beta\kappa}$ and $\Phi_{\kappa\alpha}$ are also defined similarly,
\begin{align*}
    \Phi_{\beta\kappa}\colon& D(H_{\beta}^{\text{op}}, H_{\kappa}^{\text{cop}}, \langle-\rangle_{\beta\kappa}) \to H_{\alpha}^{*}\,,\\
    \Phi_{\kappa\alpha}\colon& D(H_{\kappa}^{\text{op}}, H_{\alpha}^{\text{cop}}, \langle-\rangle_{\kappa\alpha}) \to H_{\beta}^{*}\,.
\end{align*}
\begin{lemma}\label{lem:doublet_invertible_tensor}
    For a Drinfeld pairing $(H_{\alpha}, H_{\beta}, \langle - \rangle)$, the following identities hold:
    \begin{align*}
          \includegraphics[scale=.4, valign = c]{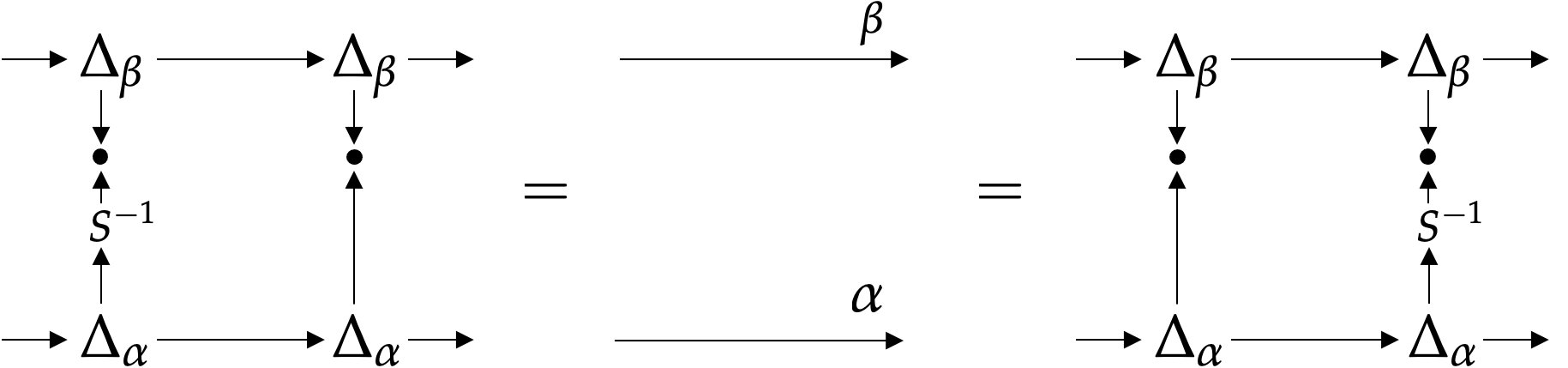}
    \end{align*}
    \begin{proof}
        We leave it as an exercise.
    \end{proof}
\end{lemma}
\begin{thm}\label{thm:triangle_identity_versions}
    Given a tuple of Hopf algebras $(H_\alpha,H_\beta,H_\kappa,\langle-\rangle)$ such that the condition in Part (a) of Definition \ref{def:hopf_triplet} holds, the followings are equivalent.
    \begin{enumerate}
        \item The Triangle Identity in Equation~\eqref{eqn:triplet_def_triangle_1} holds.
        \item $\Phi_{\alpha\beta}: D(H_{\alpha}^{\text{op}}, H_{\beta}^{\text{cop}}, \langle-\rangle_{\alpha\beta}) \to H_{\kappa}^{*}$ is a Hopf algebra morphism.
        \item The following different version of the Triangle Identity holds:
        \begin{align}\label{eqn:triplet_def_triangle_2}
          \includegraphics[scale=.4, valign = c]{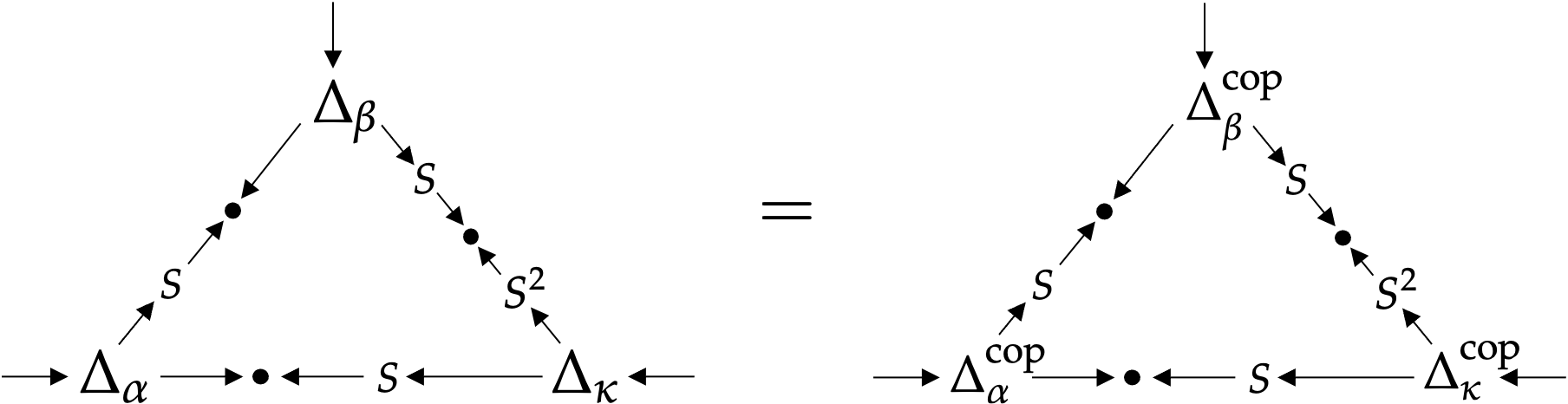}
        \end{align}
    \end{enumerate}
    By the cyclic symmetry of the Triangle Identity (Lemma \ref{lem:triplet_triangle_condition}), one can cyclically permute the indices $\alpha \mapsto \beta \mapsto \kappa \mapsto \alpha$ in each of the above three statements to obtain new identities which are also equivalent to one another. 
    \begin{proof}
        By definition, $\Phi:= \Phi_{\alpha\beta}$ are Hopf algebra morphisms when restricting to the two Hopf subalgebras   $H_{\alpha}^{\text{op}} $ and $H_{\beta}^{\text{cop}}$, and we have $\Phi(x_{\alpha}y_{\beta}) = \Phi(x_{\alpha})\Phi(y_{\beta})$ for $x_{\alpha} \in H_{\alpha}^{\text{op}}, \ y_{\beta} \in H_{\beta}^{\text{cop}}$ where we have identified $x_{\alpha}$ with $x_{\alpha} \otimes 1$ and $y_{\beta}$ with $1 \otimes y_{\beta}$. It follows that $\Phi$ is automatically a co-algebra morphism since the co-algebra structure on $D(H_{\alpha}^{\text{op}}, H_{\beta}^{\text{cop}}, \langle-\rangle_{\alpha\beta})$ is the tensor product of that of its two factors. Hence, for $\Phi$ to be a Hopf algebra morphism, it suffices to be an algebra morphism for the only remaining condition to check is
        \begin{align}\label{eqn:Phi_algebra_morphism_0}
            \Phi(y_{\beta}x_{\alpha}) = \Phi(y_{\beta})\Phi(x_{\alpha}).
        \end{align}

        Lemma \ref{lem:doublet_invertible_tensor} shows that these two tensors
        \begin{align}\label{eqn:invertible_tensor}
          \includegraphics[scale=.4, valign = c]{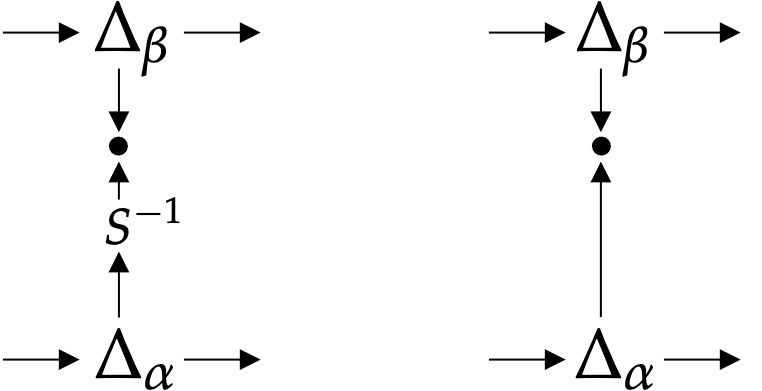}
        \end{align}
        are 2-sided inverses to each other. 
        
        \textbf{(i) $\Leftrightarrow$ (ii).} Now take the above tensor on the left, and contract it with both sides of the Triangle Identity in Equation~\eqref{eqn:triplet_def_triangle_1} (in the unique way). We obtain the following identity equivalent to Equation~\eqref{eqn:triplet_def_triangle_1},
        \begin{align}\label{eqn:Phi_algebra_morphism_1}
          \includegraphics[scale=.4, valign = c]{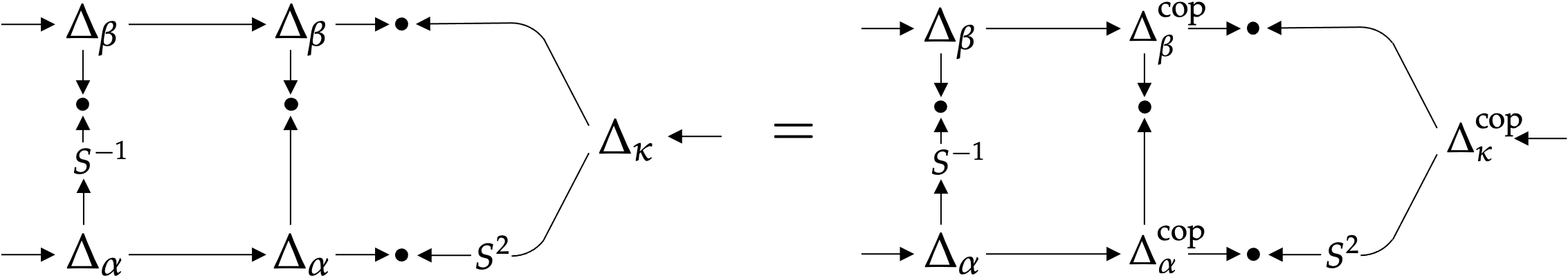}
        \end{align}
        Applying Lemma \ref{lem:doublet_invertible_tensor} to the left-hand tensor and re-write the right-hand tensor using $\Delta$ instead of $\Delta^{\text{cop}}$, we obtain
        \begin{align}\label{eqn:Phi_algebra_morphism_2}
          \includegraphics[scale=.4, valign = c]{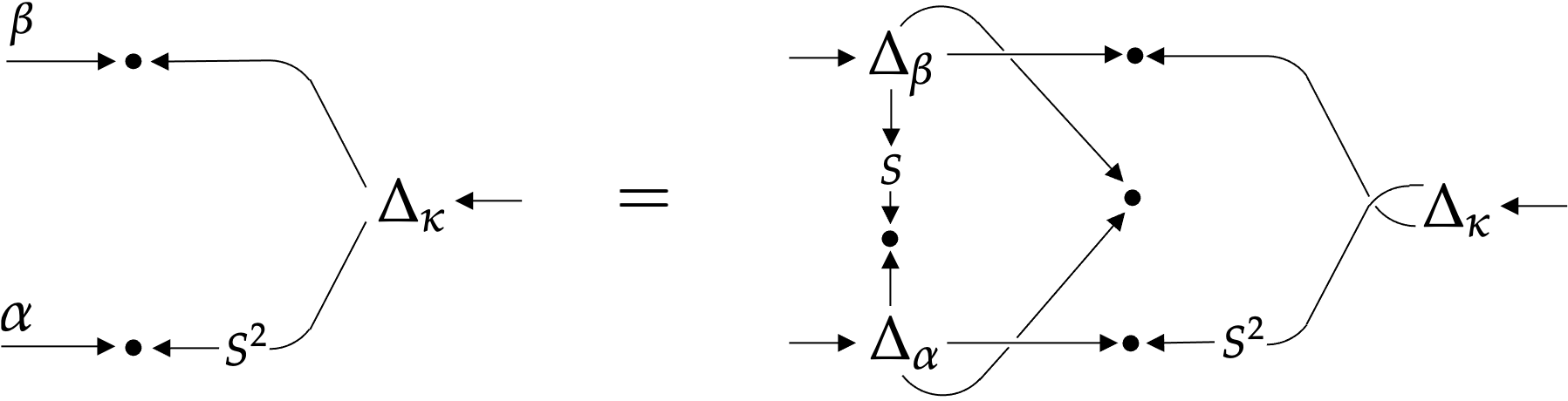}
        \end{align}
        Now, take the antipode $S_{\kappa}$ and contract it with both sides of the above identity ($S_{\kappa}$ appears on the right end of the tensor diagrams on both sides of the identity). Then, slide $S_{\kappa}$ through the tensor diagrams until the tensor $S_{\beta} \otimes S_{\alpha}$ appears on the left end. During the process we need to use the fact that we can create a pair of antipodes, one on each side of a pairing,
        \begin{align*}
          \includegraphics[scale=.4, valign = c]{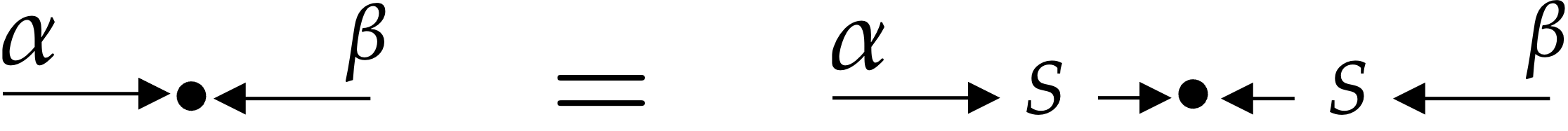}
        \end{align*}
        We obtain the following equivalent identity:
        \begin{align}\label{eqn:Phi_algebra_morphism_3}
          \includegraphics[scale=.4, valign = c]{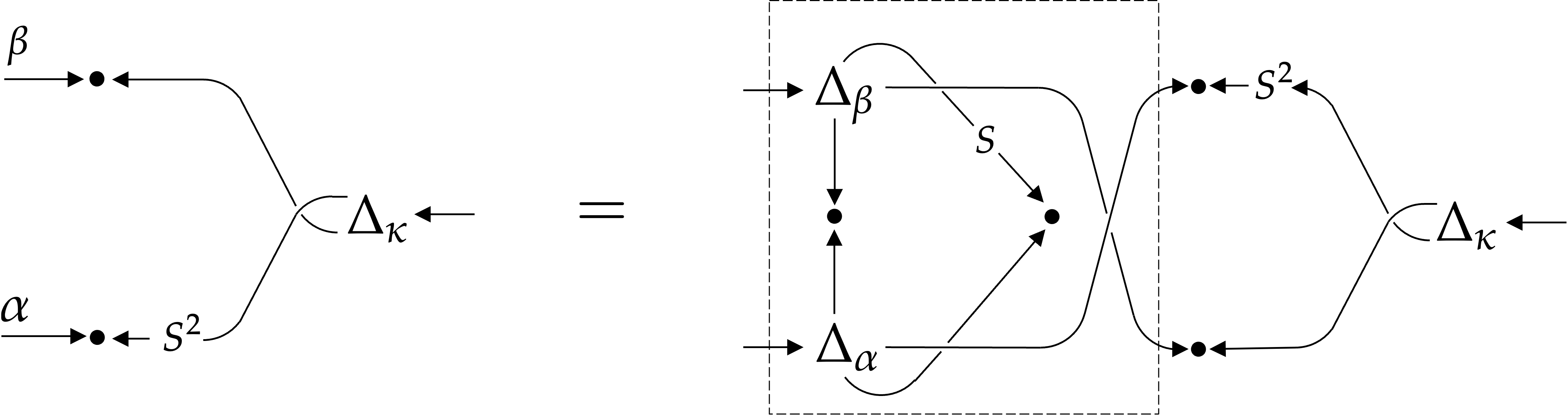}
        \end{align}
        Note that the tensor diagram within the dashed box on the right side is precisely the multiplication tensor in $D(H_{\alpha}^{\text{op}}, H_{\beta}^{\text{cop}}, \langle-\rangle_{\alpha\beta})$ which multiplies an element in $H_{\beta}^{\text{cop}}$ with an element in $H_{\alpha}^{\text{op}}$:
        \begin{align*}
          \includegraphics[scale=.4, valign = c]{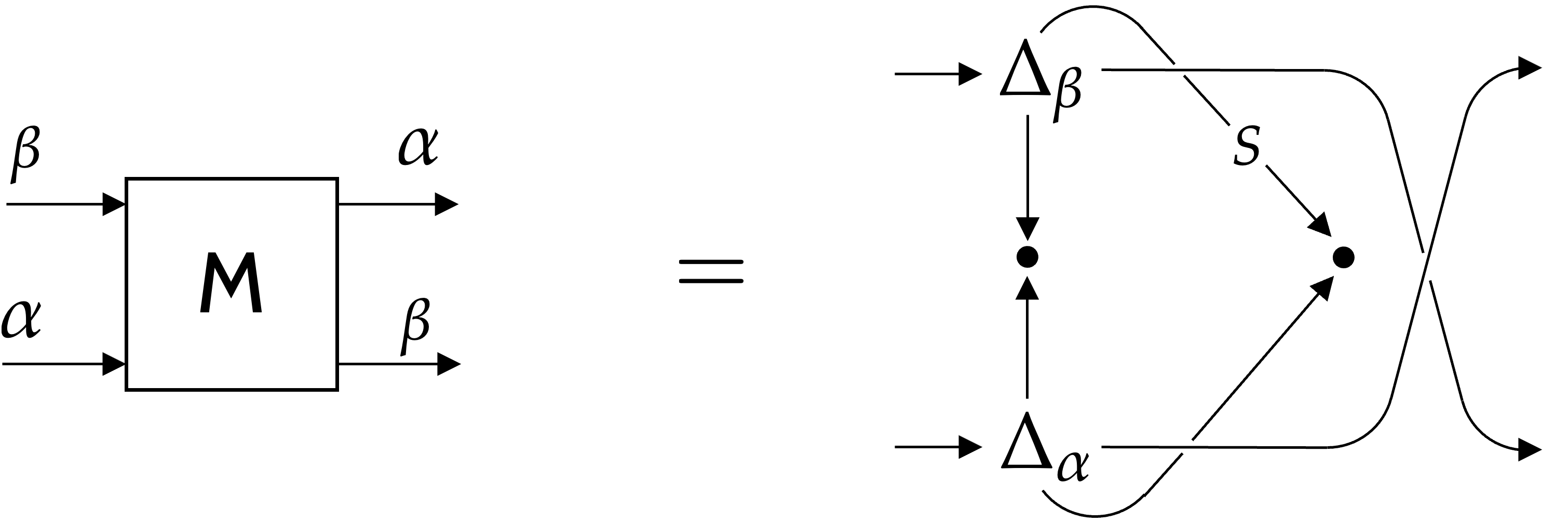}
        \end{align*}
        Then Equation~\eqref{eqn:Phi_algebra_morphism_3} is equivalent to 
        \begin{align}\label{eqn:Phi_algebra_morphism_4}
          \includegraphics[scale=.35, valign = c]{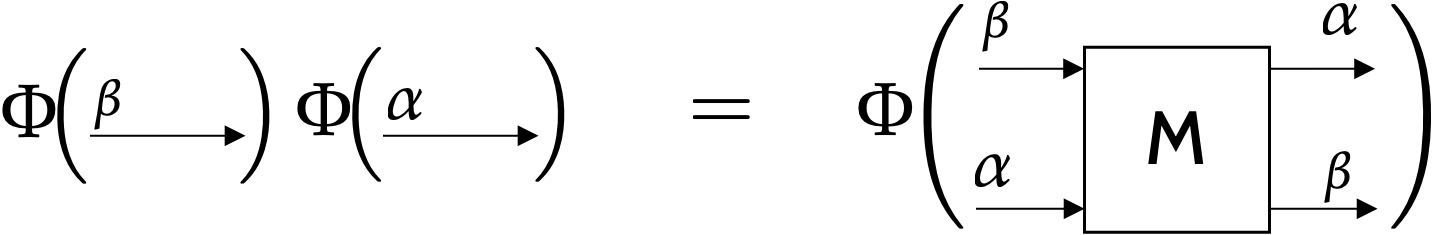}
        \end{align}
        which is the same as Equation~\eqref{eqn:Phi_algebra_morphism_0}. Note that the Triangle Identity and Equations~\eqref{eqn:Phi_algebra_morphism_1}-\eqref{eqn:Phi_algebra_morphism_4} are equivalent to one another. This shows the equivalence of (i) and (ii).

        \textbf{(ii) $\Leftrightarrow$ (iii).} The proof is similar to, and in fact slightly easier than, the previous case. By some straightforward manipulations of the antipodes, Equation~\eqref{eqn:triplet_def_triangle_2} is equivalent to the following identity:
        \begin{align*}
          \includegraphics[scale=.4, valign = c]{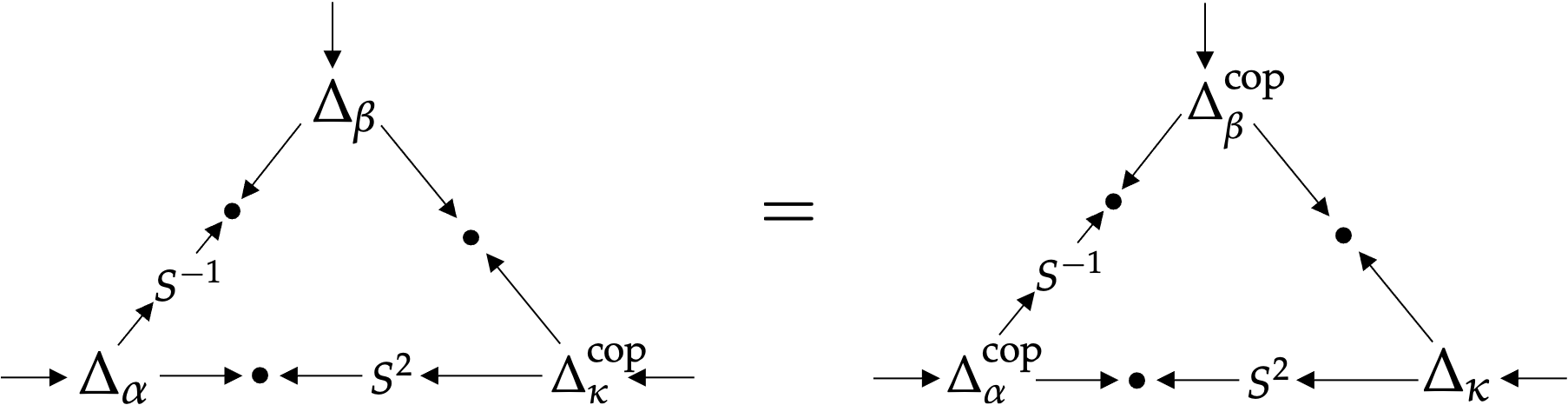}
        \end{align*}
        Now take the tensor on the right of Equation~\eqref{eqn:invertible_tensor} and contract it with both sides of the above tensor diagrams; this yields exactly the identity in Equation~\eqref{eqn:Phi_algebra_morphism_3}. 
    \end{proof}
\end{thm}

Given a Hopf triplet $\mathcal{H} = (H_\alpha,H_\beta,H_\kappa,\langle-\rangle)$, for $\mu \in \{\alpha, \beta, \kappa\}$, denote by $q_{\mu}$ the phase element of $H_{\mu}$.  For two adjacent indices $(\mu, \nu)$ in the cyclically ordered list  $\alpha \mapsto \beta \mapsto \kappa \mapsto \alpha$, Part (b) of Definition \ref{def:hopf_triplet} implies $q_{\mu} = q_{\nu}^{-1}$. It follows that $q_{\alpha} = q_{\beta} = q_{\kappa} = \pm 1$.
\begin{definition}\label{def:triplet_phase_balanced}
    Let  $\mathcal{H} = (H_\alpha,H_\beta,H_\kappa,\langle-\rangle)$ be  a Hopf triplet.
    \begin{enumerate}
        \item The phase of $\mathcal{H}$ is defined as $q_{\mathcal{H}}:= q_{\alpha} = \pm 1$.
        \item $\mathcal{H}$ is called balanced (resp.~involutory) if all of its Hopf algebras are balanced (resp. involutory).
    \end{enumerate}
\end{definition}

\begin{example}
    Let $(H,R)$ be a quasi-triangular Hopf algebra satisfying the condition
    \begin{align}\label{eqn:quasi_R_condition_1}
          \includegraphics[scale=.4, valign = c]{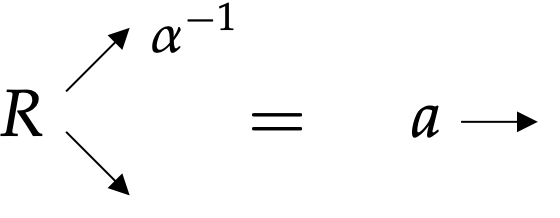}
    \end{align}
    Using the notation in Example \ref{example:quasi_triangular_pairing}, the above condition means $f_R(\alpha^{-1}) = a$.
    Note that Equation~\eqref{eqn:quasi_R_condition_1} holds whenever $\alpha = \epsilon$ and $a = 1$, and in particular when $H$ is semisimple. We define a Hopf triplet $\mathcal{H} = (H_{\alpha}, H_{\beta}, H_{\kappa}; \langle-\rangle)$ as follows:
    \begin{align*}
        H_{\alpha} &:= H^*,  \\
        H_{\beta}  &:= H^{\op{cop}},\\
        H_{\kappa} &:= H^*.
    \end{align*}
    For $p_{\alpha} \in H_{\alpha}, x_{\beta} \in H_{\beta}, p_{\kappa} \in H_{\kappa}$, the pairings are defined by
    \begin{align*}
        \langle p_{\alpha}, x_{\beta}\rangle_{\alpha\beta} &:= p_{\alpha}(S(x_{\beta})),\\
        \langle x_{\beta}, p_{\kappa}\rangle_{\beta\kappa} &:= p_{\kappa}(x_{\beta}),\\
        \langle p_{\kappa}, p_{\alpha}\rangle_{\kappa\alpha} &:= (p_{\alpha} \otimes p_{\kappa})\circ (S \otimes S^{-2})(R)\,.
    \end{align*}
    It is straightforward to check that $(H_{\alpha}, H_{\beta};\langle-\rangle)$ and $(H_{\beta}, H_{\kappa};\langle-\rangle)$ are Drinfeld pairings and they preserve the phase elements. To see that $(H_{\kappa}, H_{\alpha};\langle-\rangle)$ is a Drinfeld pairing, notice that the pairing-induced map $H_{\alpha} (= H^*) \to H_{\kappa}^{*,\op{op}} (= H^{\op{op}})$ can be expressed as the composition
    \begin{equation*}
        \begin{tikzcd}
            H^* \arrow[r, "S^*"] & H^{*,\op{cop}, \op{op}}\arrow[r, "f_R"] & H^{\op{op}}\arrow[r, "S^{-2}"] & H^{\op{op}},
        \end{tikzcd}
    \end{equation*}
   where each map is a Hopf algebra morphism. 

    That the pairing on $H_{\kappa} \otimes H_{\alpha}$ preserves the phase element is equivalent to Equation~\eqref{eqn:quasi_R_condition_1}. We leave this as an exercise.

    Lastly, the Triangle Equation~\eqref{eqn:triplet_def_triangle_1} becomes the identity
    \begin{align*}
          \includegraphics[scale=.4, valign = c]{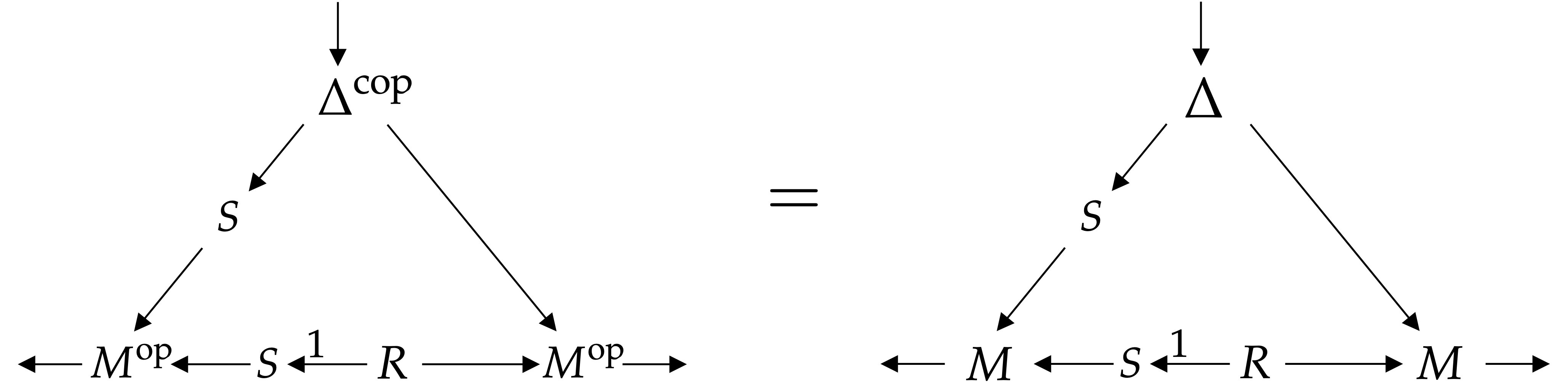}
    \end{align*}
    which is equivalent to Equation~\eqref{eqn:R-matrix_condition1} satisfied by the $R$-matrix.
\end{example}

\begin{example}
    One can also numerically search for Hopf triplets from low-dimensional Hopf algebras. For example, Hopf algebras of dimension up to 11 are classified in \cite{cstefan1999hopf}. Take the 8-dimensional Hopf algebra $A:= A_{\text{C}_2 \times \text{C}_2}$ as an illustration (see Theorem 3.5 \cite{cstefan1999hopf}). As an algebra, $A$ is the quotient of the free algebra  $k[G, H, X]$ by the 2-sided ideal generated by,
    $$G^2 - 1, \  H^2 - 1,\  X^2,\  GX + XG,\  HX + XH,\  GH - HG\,.$$
    Moreover, $G$ and $H$ are group-like elements, and $X$ is $(G, 1)$-primitive, i.e.~$\Delta(X)  = X \otimes G + 1 \otimes X$. The phase element $a$ of $A$ is $G$, and the phase element $\alpha$ of $A^*$ is given by $\alpha(G) = \alpha(H) = -1$, and $\alpha(X) = 0$. The Hopf algebra $A$ is balanced but not semisimple. Choose an ordered basis of $A$ as
    $$(1, G, H, GH, X, GX, HX, GHX)\,.$$

    We look for Hopf triplets of the form $(A, A, A; \langle-\rangle)$ where the three Hopf algebras are all equal to $A$, and the three pairings are assumed to be the same as well. Under the above basis, the pairing is represented by an $8 \times 8$ matrix. By direct computer calculations, the following is a solution to the pairing where $t$ is an arbitrary parameter:
    \begin{equation*}
    \left(
\begin{array}{cccccccc}
 1 & 1 & 1 & 1 & 0 & 0 & 0 & 0 \\
 1 & -1 & -1 & 1 & 0 & 0 & 0 & 0 \\
 1 & -1 & -1 & 1 & 0 & 0 & 0 & 0 \\
 1 & 1 & 1 & 1 & 0 & 0 & 0 & 0 \\
 0 & 0 & 0 & 0 & t & -t & -t & t \\
 0 & 0 & 0 & 0 & t & t & t & t \\
 0 & 0 & 0 & 0 & t & t & t & t \\
 0 & 0 & 0 & 0 & t & -t & -t & t \\
\end{array}
\right)
\end{equation*}
\end{example}

\section{Trisections and Combings} \label{sec:trisections_and_combings} In this section, we review the theory of 4-manifold trisections and trisection diagrams. We also discuss singular combings of trisection diagrams. 

\subsection{Open books} \label{subsec:open_books} We begin with a brief discussion of open books on $3$-manifolds. A relative trisection determines a 4-manifold equipped with an open book on the boundary.

\begin{definition}\label{def:marked_open_book}
A \emph{marked open book} $\pi = (B,C,\pi)$ of type $(p,b)$ on a closed $3$-manifold $Y$ consists of the following data.
\begin{itemize}
    \item A closed $1$-manifold $B \subset Y$ with $b$ components, called the \emph{binding} of $\pi$.
    \vspace{2pt}
    \item A distinguished component $C \subset B$ of the binding, called the \emph{marking} of $\pi$.
    \vspace{2pt}
    \item A fibration $\pi:Y \setminus B \to S^1$. The fiber $\Sigma = \pi^{-1}(p)$ is called the \emph{page} and must be genus $p$.
\end{itemize}
The fibration must be trivial near the binding, in the sense that there must exist a tubular neighborhood $B \times D^2 \simeq U$ of $B$ in $Y$ such that $\pi|_U$ is given by the projection
\[\pi(x,r,\theta) = \theta \in S^1 \qquad \text{for}\qquad (x,r,\theta) \in B \times (D^2 \setminus 0) \simeq U \setminus B\] 
\end{definition}

\begin{example}[Trivial open book] \label{ex:trivial_OB} The standard open book $(B,\pi)$ on $S^3$ is given as follows. Identify $S^3$ with the unit sphere in $\C^2$. Then we define the binding by
\[B = (0 \times \C) \cap S^3\]
Then we have $S^3 \setminus B \subset \C^\times \times \C$ and we define $\pi$ as the restriction of
\[
\pi:\C^\times \times \C \to S^1 \qquad \text{given by}\qquad \pi(z_1,z_2) = \frac{z_1}{|z_1|} \in S^1 \subset \C
\]\end{example}

\subsection{Surface diagrams} \label{subsec:curves_on_surfaces} We next discuss general surface diagrams and their moves. 

\begin{definition} A \emph{surface diagram} $D = (\Sigma;\Gamma_1,\dots,\Gamma_k)$ consists of a compact oriented surface with boundary $\Sigma$ and embedded, closed $1$-manifolds
\[\Gamma_i \subset \Sigma \qquad\text{for each }i = 1,\dots,k\]
that are pairwise transverse (so $\Gamma_i \pitchfork \Gamma_j$) and without triple intersections (so $\Gamma_i \cap \Gamma_j \cap \Gamma_k = \emptyset$). A component $\eta \subset \Gamma_i$ is called a \emph{curve} and two components of $\Gamma_i$ are said to be the same \emph{color}. 

\vspace{3pt}

An \emph{isomorphism} or \emph{diffeomorphism} $\phi:D \simeq D'$ is simply a diffeomorphism $\phi:\Sigma \simeq \Sigma'$ of the underlying surfaces that restricts to a diffeomorphism $\phi:\Gamma_i \simeq \Xi_i$ of each embedded $1$-manifold. \end{definition}

There are a number of natural ways to modify surface diagrams. First, we can move the $1$-manifolds in a surface diagram around via isotopy. 

\begin{definition} An \emph{isotopy} $D_t$ of surface diagrams is a $1$-parameter family of diagrams
\[
D_t = (\Sigma;\Gamma_1^t,\dots,\Gamma_k^t) \qquad\text{for}\qquad t \in [0,1]
\]
such that $\Gamma_i^t \subset \Sigma$ varies smoothly in $t$ for each $i$. \end{definition}

\noindent Next, we can change the crossings of a surface diagram by standard regular homotopies (that are essentially Reidemeister moves).

\begin{definition} A \emph{two-point} move on a surface diagram $D$ is the local modification
\begin{figure}[h!]
\centering
\includegraphics[width=.5\textwidth]{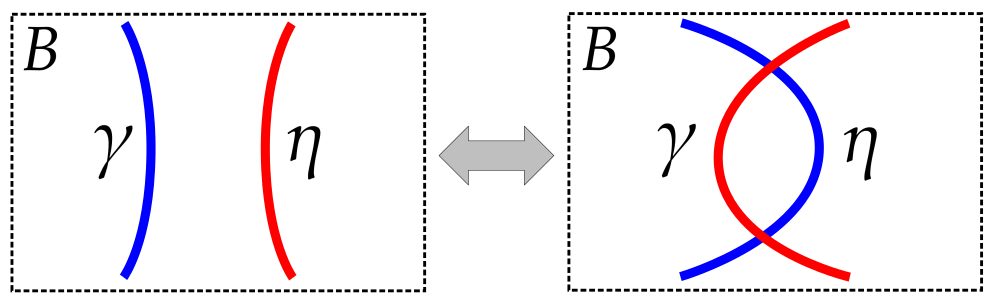} 
\label{fig:surface_diagram_two_point}
\end{figure} 

\noindent A \emph{three-point} move on a surface diagram $D$ is the local modification
\begin{figure}[h!]
\centering
\includegraphics[width=.5\textwidth]{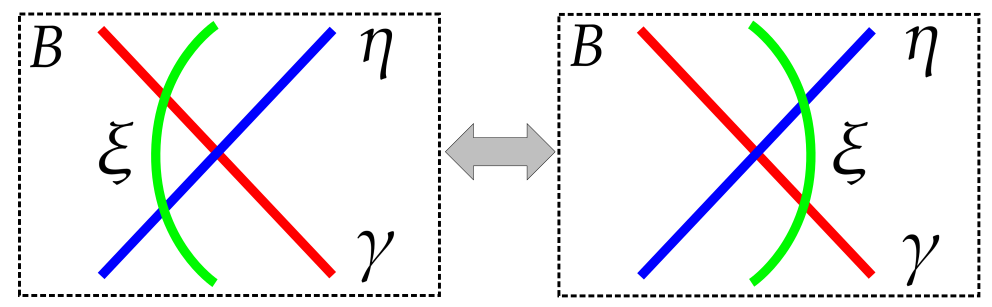} 
\label{fig:surface_diagram_three_point}
\end{figure} 
\end{definition}

\noindent Note that any pair of surface diagrams $D$ and $D'$ connected by a $1$-parameter family of surface diagrams
\[(D;\Gamma_1^t,\dots,\Gamma_k^t) \qquad\text{for}\qquad t \in [0,1]\]
can be connected by a sequence of two point moves, three point moves and diffeomorphisms (see \cite{chaidez20194manifold}). Finally, we can modify a surface diagram by ambiently summing curves of the same color.

\begin{definition} A \emph{handle-slide} of a curve $\eta$ over a curve $\xi$ (of the same color) in a surface diagram $D$ is a surface diagram $D'$ acquired from $D$ by replacing $\eta$ with a band sum curve
\[\eta \; \#_\gamma \; \xi\]
Precisely, let $\gamma$ be an arc connecting $\eta$ and $\xi$ in the surface $\Sigma$ of $D$, and let $U$ be a band neighborhood of $\eta \cup \gamma \cup \xi$. Then $\eta \#_\gamma \xi$ is the unique component of $\partial U$ that is not homotopic to $\eta$ or $\xi$ in $U$.
\begin{figure}[h!]
\centering
\includegraphics[width=\textwidth]{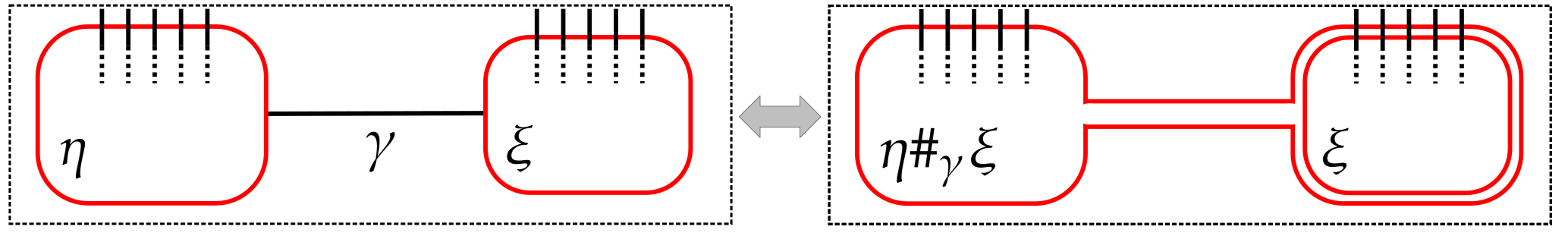} 
\label{fig:handle_slide_point}
\end{figure} 
\end{definition}

There is also a general procedure for summing two surface diagrams.

\begin{definition} The \emph{boundary sum} $D \natural D'$ of two surface diagrams 
\[D = (\Sigma;\Gamma_1,\dots,\Gamma_k) \qquad\text{and}\qquad D' = (\Sigma';\Xi_1,\dots,\Xi_k)\]
along components $C \subset \partial \Sigma$ and $C' \subset \partial \Sigma$ is the surface diagram with underlying surface given by the boundary sum $\Sigma \natural \Sigma'$ along $C$ and $C'$, and $1$-manifolds given by $\Gamma_1 \cup \Xi_1,\dots,\Gamma_k \cup \Xi_k$. 
\end{definition}

\subsection{Trisection diagrams} \label{subsec:trisections} Trisections and trisection diagrams were introduced by Gay-Kirby \cite{gk2016} (in the closed case) and Castro-Gay-Pinzon-Caicedo \cite{cgpc2018relativetrisections} (in the case with boundary). A trisection diagram is a type of surface diagram that specifies a handlebody decomposition of a 4-manifold. 

\vspace{3pt}

In this paper, we will work with a variant of the relative trisection theory of \cite{cgpc2018relativetrisections} that equips a trisection with a distinguished boundary component.

\begin{definition} \label{def:trisection_diagram} A \emph{marked trisection diagram} $T$ is a surface diagram $(\Sigma;\alpha,\beta,\kappa)$ and a marked component $C \subset \partial \Sigma$ such that the surface diagrams $D_i$ given by
\[D_1 = (\Sigma;\alpha,\beta) \qquad D_2 = (\Sigma;\beta,\kappa) \quad\text{and}\quad D_3 = (\Sigma;\kappa,\alpha)\]
are diffeomorphic (after handleslides) to the following standard diagram for each $i = 1,2,3$. 
\begin{figure}[h!]
\centering
\includegraphics[width=\textwidth]{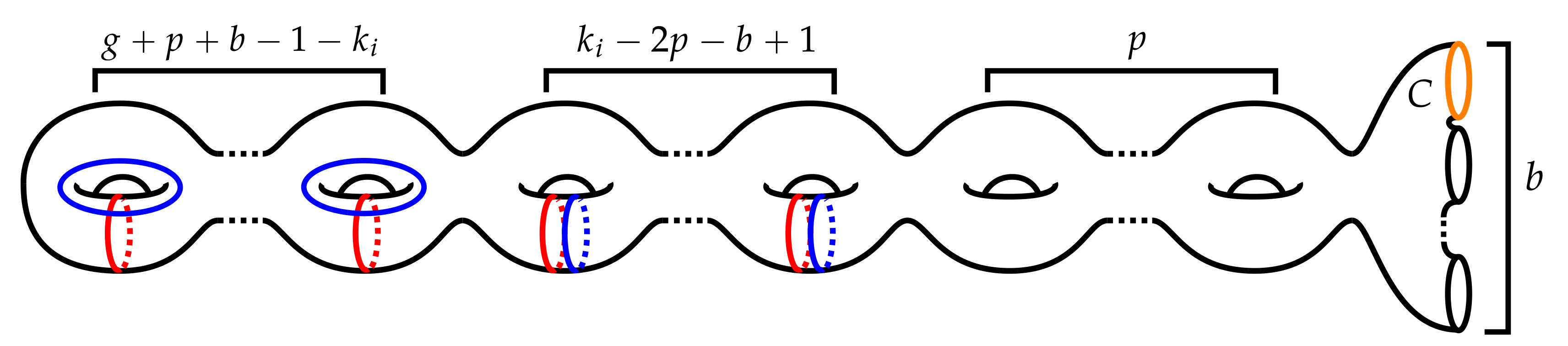} 
\end{figure}

\noindent The quantities $g, k_i, p$ and $b$ will be referred to as the \emph{genus}, \emph{indices}, \emph{page genus} and \emph{boundary number} of the trisection $T$, respectively. We will use the notation
\[g(T)\,, \quad k_i(T)\,, \quad p(T)\,, \quad\text{and}\quad b(T)\,,\]
when we must emphasize the specific trisection. The \emph{type} of the trisection $T$ is the tuple
\[(g,k_1,k_2,k_3;p,b)\]\end{definition}

All of the standard operations on surface diagrams can be applied to trisection diagrams. Additionally, we adopt the convention that the boundary sum
\[S \natural T \quad\text{of two marked trisection diagrams $S$ and $T$}\]
always occurs along the marked boundary component, unless explicitly stated otherwise. The type of $S \natural T$ is specified as follows.
\[g(S \natural T) = g(S) + g(T)\,,\qquad k_i(S \natural T) = k_i(S) + k_i(T)\,, \qquad p(S \natural T) = p(S) + p(T)\,,\]
\[b(S \natural T) = b(S) + b(T) - 1\,.\]

\begin{definition}\label{def:stabilization}
A \emph{stabilization} $S$ and a \emph{destabilization} $D$ of a marked trisection $T$ are, respectively, marked trisections that satisfy
\[S \simeq T \; \natural \; T^{\op{st}}_i \qquad\text{and}\qquad T \simeq D \;\natural \; T^{\op{st}}_i\]
Here $T^{\op{st}}_i$ is any of the three standard marked trisection diagrams $T^{\op{st}}_1,T^{\op{st}}_2$ and $T^{\op{st}}_3$ given below.

\begin{figure}[h!]
\centering
\includegraphics[width=\textwidth]{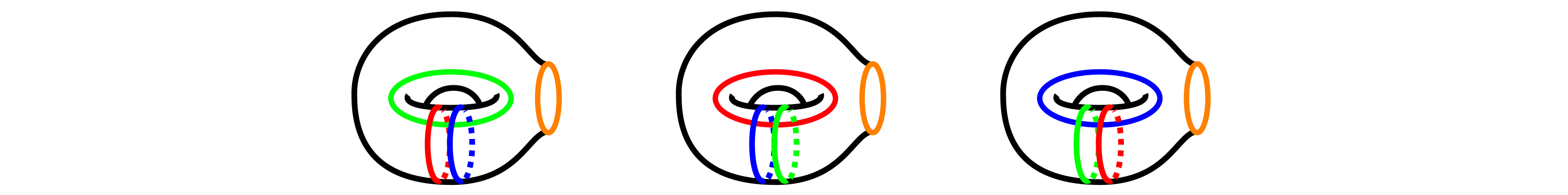} 
\label{fig:stabilized_sphere_trisection}
\end{figure}

\noindent Each stabilization trisection is a trisection for the $4$-ball $B^4$ equipped with an open book with page diffeomorphic to the disk. The type of $T^{\op{st}}_j$ is given by
\[(1,\delta_{1j},\delta_{2j},\delta_{3j};0,1)\]
Here $\delta_{ij}$ is the usual indicator, i.e. $\delta_{ij} = 1$ if $i = j$ and $0$ otherwise.\end{definition}

A \emph{trisection move} is any stabilization, destabilization, isotopy or diffeomorphism operation. The following result relates marked trisections modulo trisection moves to 4-manifolds. 

\begin{thm} \cite{gk2016,cgpc2018relativetrisections} There is a bijection between marked trisection diagrams $T$ (modulo trisection moves) and connected, compact $4$-manifolds $X$ with a marked open book on $\partial X$ (modulo diffeomorphism).
\[\frac{\Big\{\text{marked trisections $T$}\Big\}}{\text{trisection moves}} \xrightarrow{\sim}  \frac{\Big\{\begin{array}{c} \text{connected, compact 4-manifolds $X$}\\\text{with a marked open book $\pi$ on $\partial X$}\end{array}\Big\}}{\text{diffeomorphism}}\]
If the type of $T$ is $(g,b;p,k)$, then the type of the induced open book $\pi$ is $(p,b)$. \end{thm}

\noindent The construction of a 4-manifold with a boundary open book from a relative trisection diagram uses a standard (partial) handlebody decomposition associated to the diagram. The details are described by Castro-Gay-Pinzon-Caicedo in \cite[\S 3]{cgpc2018relativetrisections}. 

\vspace{3pt}

Trisections for closed 4-manifolds can be viewed as a strict subset of trisections of 4-manifolds with non-trivial boundary, as follows. Consider a closed, oriented, connected $4$-manifold $M$. Choose a ball $B \subset X$ and let $(X,\pi)$ be the pair
\[
X = M \setminus B \qquad\text{and}\qquad \pi = \pi_{\op{std}} \text{ on }B \simeq S^3
\]
Here $\pi_{\op{std}}$ is the standard open book on $S^3$ with disk-like pages (see Example \ref{ex:trivial_OB}). Let $T = (\Sigma,\alpha,\beta,\kappa)$ be a trisection for $(X,\pi)$. Note that this trisection has a single boundary component, since the corresponding open book only has one binding component. In particular, there is a canonical marking of $(X,\pi)$ and we may regard it as marked.

\begin{lemma} \label{lem:rel_trisection_to_trisection} The diagram $T \cup D^2$ acquired by gluing a disk $D^2$ onto the boundary of $T$ is a trisection diagram for the closed 4-manifold $M$.
\end{lemma}

\begin{proof} This follows immediately from the general gluing formula of Castro-Ozbagci \cite[Prop. 2.12]{co2017lefschetztrisections}, which is itself a special case of the gluing formula of Castro \cite{castro2017cobordisms}. In particular, $M$ may be viewed as a gluing
\[X \cup_f B\]
where $f:\partial X \simeq \partial B$ is the identity diffeomorphism, intertwining the standard open book. The trisection diagram $S$ for $(B,\pi_{\op{std}})$ is simply the disk (with no curves). Thus, by \cite[Prop. 2.12]{co2017lefschetztrisections}, a trisection diagram for $M$ is given by
\[T \cup_{\partial T} S = (\Sigma \cup_{\partial \Sigma} D^2, \alpha,\beta,\kappa)\]
Typically, \cite[Prop. 2.12]{co2017lefschetztrisections} prescribes that the addition of some new $\alpha,\beta$ and $\kappa$ curves to $T'$, but in this case there are no such curves since
\[g(T') = g(T) \qquad\text{and}\qquad k_i(T') = k_i(T) \qedhere\] 
\end{proof}

\begin{remark} Conversely, any trisection diagram $T'$ of $M$ can be turned into a trisection of $(X,\pi)$ by removing a disk.
\end{remark}

We will require a description of the homology of a $4$-manifold in terms of its trisections. Given a closed $1$-manifold $\nu \subset \Sigma$, let
\[L_\nu := \op{span}\big([\eta] \; : \; \eta \text{ is a connected component of }\nu\big) \subset H_1(\Sigma)\]

\begin{prop} \label{prop:trisection_homology} \cite[Thm 2]{tanimoto2021homology} Let $T = (\Sigma;\alpha,\beta,\kappa)$ be a trisection for the $4$-manifold $X$. Then the complex
\[
0 \to \Z \xrightarrow{0} (L_\alpha \cap L_\beta) \oplus (L_\beta \cap L_\kappa) \oplus (L_\kappa \cap L_\alpha) \xrightarrow{\jmath} L_\alpha \oplus L_\beta \oplus L_\kappa \xrightarrow{\iota} H_1(\Sigma) \xrightarrow{0} \Z \to 0
\]
computes the homology groups of $X$. Here the maps $\jmath$ and $\iota$ are given by
\[\jmath(x,y,z) = (x - y,y - z, z - x) \qquad\text{and}\qquad \iota(x,y,z) = x + y + z\]
\end{prop}

\subsection{Combings} Next, we introduce the notion of a combing of trisection diagram. 

\vspace{3pt}

We must first recall a notion of degree. Let $v$ be a tangent vector-field on a surface $\Sigma$ defined along a closed immersed curve $C$. There is a canonical (isotopy class of) bundle isomorphism
\[
\tau:T\Sigma|_C \simeq \R^2
\]
which takes a non-vanishing vector-field $w$ tangent to $C$ to $(1,0) \in \R^2$. The \emph{degree} $\op{deg}(C,v)$ of $v$ along $C$ is the degree of the map
\[
C \to S^1 \qquad \text{given by}\qquad p \mapsto \frac{\tau \circ v(p)}{|\tau \circ v(p)|} \in S^1 
\]

\begin{definition}\label{def:combing_on_diagram}
A \emph{(singular) combing} $v$ on a marked trisection $T$ of type $(g,k;p,b)$ is a vector-field
\[v:\Sigma \to T\Sigma \qquad\text{on the trisection surface }\Sigma\]
that satisfies the following properties.
\begin{itemize}
\item[(a)] $v$ has exactly one zero $p_\eta$ of index $-1$ on each curve $\eta$ in $T$, and no additional zeros. 
\item[(b)] $v$ is tangent to $\eta$ and pointing out of $p_\eta$ in a neighborhood of $p_\eta$. 
\item[(c)] $v$ has degree $\chi(\Sigma) + 3(g - p)$ on the marked component of $\partial \Sigma$ and degree $0$ on the others.  
\end{itemize}
An \emph{isotopy} of combings is a $1$-parameter family of vector-fields $v_t$ satisfying the criteria above.
\end{definition}

A pair of singular combings $u$ on $S$ and $v$ on $T$ can be boundary summed to form a combing
\[
u \natural v \quad\text{on the boundary sum diagram}\quad S \natural T
\]
The sum is acquired by putting $u$ and $v$ into a standard form near the marked boundary components of $S$ and $T$, as in the following picture:

\vspace{4pt}

\begin{center}
\includegraphics[width=.7\textwidth]{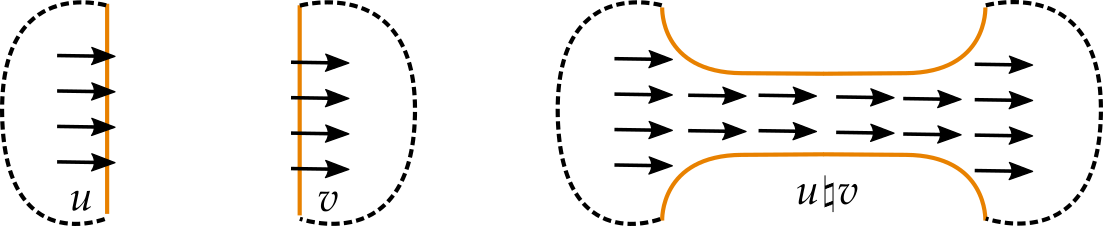}
\label{fig:boundary_sum_combings}
\end{center}

\vspace{3pt}

\noindent We will be primarily interested in equivalence classes of combings up to certain moves.

\begin{definition}\label{def:combing_move}
A \emph{combing move} on a combing $v$ of a marked trisection diagram $T$ is any of the following local modifications in a box $B$ around a singularity $p_\eta$ of $v$.
\begin{itemize}
\item A \emph{basepoint isotopy} that moves the basepoint of $v$ across an intersection point.
\vspace{2pt}
\begin{center}
\includegraphics[width=.5\textwidth]{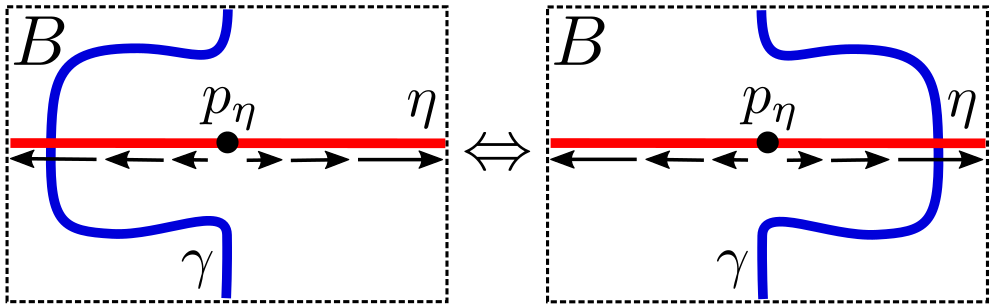} 
\label{fig:basepoint_crossing}
\end{center}
\vspace{2pt}

\item A \emph{basepoint spiral} that twists the combing $v$ near the basepoint.
\vspace{2pt}
\begin{center}
\includegraphics[width=.5\textwidth]{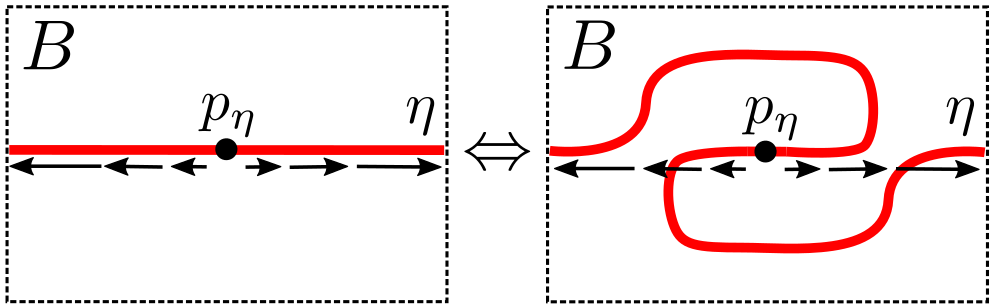} 
\label{fig:basepoint_spiral}
\end{center}
\vspace{2pt}
\end{itemize}
The set of equivalence classes $[v]$ of combings on $T$ up to trisection moves is denoted by
\[\text{Comb}(T)\]
\noindent Note that both of these moves are depicted as modifications of the trisection curves supported in $B$, that leave $v$ unchanged. However, they are equivalent (up to diffeomorphism supported in $B$) to modifications of $v$ supported in $B$ that leave the trisection unchanged.
\end{definition}

We would like to classify combings modulo isotopy, basepoint isotopies, and basepoint spirals. For this purpose, we give the following definition.

\begin{definition} \label{def:combing_difference} Let $u$ and $v$ be a pair of combings of a marked trisection $T$. The \emph{difference}
\[[v - u] \in H_1(\Sigma)/(L_\alpha + L_\beta + L_\kappa) \simeq H_1(X)\]
is defined as follows. Apply isotopies and basepoint isotopies so that
\begin{itemize}
    \item[(a)] $u$ and $v$ have the same singularity $p_\eta$ on each curve $\eta$.
    \item[(b)] $u$ and $v$ agree in a neighborhood of each singularity $p_\eta$ and the boundary $\partial \Sigma$.
\end{itemize}
Choose a bundle isomorphism $\tau:T\Sigma|_U \simeq \C$ of $T\Sigma$ over the open set $U = \Sigma \setminus (\cup_\eta p_\eta)$, and let
\[\phi:U \to \R/\Z \simeq S^1 \qquad\text{given by}\qquad \phi(s) = \op{arg}\big(\frac{\tau \circ v(s)}{\tau \circ u(s)}\big) \in \R/\Z\]
Since $\phi$ vanishes near $p_\eta$, it extends to a map $\phi:\Sigma \to S^1$, and we acquire a cohomology class
\[\phi^*E \in H^1(\Sigma) \qquad\text{where}\qquad E = \op{PD}[S^1] \in H^1(S^1)\]
Since $\phi = 0$ near $\partial \Sigma$, we have $\phi^*E \in H^1(\Sigma,\partial\Sigma) \subset H^1(\Sigma)$. We let
\[[v - u]_\Sigma := \op{PD}(\phi^*E) \in H_1(\Sigma) \qquad\text{and}\qquad [v - u] := [v - u]_\Sigma + (L_\alpha + L_\beta + L_\kappa)\]\end{definition}

\begin{lemma} The difference class $[v - u] \in H_1(\Sigma)/(L_\alpha + L_\beta + L_\kappa)$ is well-defined and
\[[w - u] = [w - v] + [v - u]\]
Moreover, $[v - u] = 0$ if and only if $v$ and $u$ are isotopic after a sequence of basepoint isotopies and spirals. 
\end{lemma}

\begin{proof} Let $u$ and $v$ be two combings of $T$. We will say that $u$ and $v$ are in \emph{good position} if they satisfy Definition \ref{def:combing_difference}(a-b). 

\vspace{3pt}

There are several properties of $[v - u]_\Sigma$ that are immediate from the definition. First, note that if $u_t$ and $v_t$ are arbitrary isotopies of vector-fields that are in good position at each $t$, then
\[[v_t - u_t]_\Sigma \in H_1(\Sigma)\]
is independent of $t$. Second, note that if $u_0,\dots,u_k$ are in pairwise good position, then
\begin{equation} \label{eqn:additive_difference_sigma} 
[u_k - u_0]_\Sigma = [u_k - u_{k-1}]_\Sigma + \dots + [u_1 - u_0]_\Sigma\end{equation}
Finally, note that $[v - u]_\Sigma = -[u - v]_\Sigma$. We now prove the desired claims in four steps.

\vspace{3pt}

{\bf Step 1: Basepoint isotopy.} If $u$ and $v$ are in good position and related by a sequence of basepoint isotopies carrying a basepoint $p_\eta$ once around the curve $\eta$, then
\[[v - u]_\Sigma = [\eta] \in H_1(\Sigma)\]
Indeed, if $\gamma$ is any closed curve in $\Sigma$, then
\[[v - u]_\Sigma \cdot [\gamma] = \langle \phi^*E,[\gamma]\rangle = \#(v \cap T\gamma) - \#(u \cap T\gamma)\]
Here $\#(v \cap T\gamma)$ and $\#(u \cap \gamma)$ are signed counts of points where $v$ and $u$ are, respectively, tangent to $\gamma$ in the direction parallel to the orientation. By reversing the basepoint isotopies from $u$ to $v$ as an ambient isototopy and applying it to $\gamma$, we find that
\[
\#(v \cap T\gamma) = \#(u \cap T\gamma') 
\]  
Here $\gamma'$ is a curve acquired by isotoping each strand of $\gamma$ that intersects $\eta$ over the singularity $p_\eta$ once. In particular,
\[
 \#(u \cap T\gamma') = \#(u \cap T\gamma) + [\eta] \cdot [\gamma]
\]

\vspace{3pt}

{\bf Step 2: Basepoint spiral.} If $u$ and $v$ are in good position and related by a single basepoint spiral at a basepoint $p_\eta$, then
\[[v - u]_\Sigma = 0 \in H_1(\Sigma)\]
Indeed, there is an isotopy $u_t$ of vector-fields from $u$ to $v$ so that $u_t$ is in good position with $u$ and the angle between $u_t$ and $u$ is constant in a neighborhood of $p_\eta$ (but varies with $t$). Thus, there is an isotopy of difference functions
\[\phi_t:\Sigma \to S^1\]
Accordingly $[u_t - u]_\Sigma = \phi_t^*E$ is constant in $t$, and must vanish.

\vspace{3pt}

{\bf Step 3: Well-definedness.} Finally, suppose that $u$ and $v$ are arbitrary. Apply isotopies and basepoint isotopies to acquire pairs of combings $u_0,v_0$ and $u_1,v_1$ that are in good position.  To prove well-definedness, we must show that
\[[v_0 - u_0]_\Sigma - [v_1 - u_1]_\Sigma \in L_\alpha + L_\beta + L_\kappa\]
By simultaneously isotoping the pair $(u_1,v_1)$, we may assume that all of the vector-fields $u_0,u_1,v_0$ and $v-1$ are in pairwise good position. Then by (\ref{eqn:additive_difference_sigma}), we have
\begin{equation} \label{eqn:combing_difference_welldefined_2}
[v_1 - u_1]_\Sigma = [v_1 - v_0]_\Sigma + [v_0 - u_0]_\Sigma + [u_0 - u_1]_\Sigma
\end{equation}
By Steps 1 and 2, $[v_1 - v_0]_\Sigma$ and $[u_1 - u_0]_\Sigma$ are in $L_\alpha + L_\beta + L_\kappa$ since they are related by basepoint isotopies and spirals. This proves well-definedness of $[v - u]$, and additivity follows from (\ref{eqn:additive_difference_sigma}).

\vspace{3pt}

{\bf Step 4: Kernel.} Finally, if $u$ and $v$ are related by basepoint isotopies and spirals, then this is still the case when they are put in good position. Thus $[v - u] = 0$ by Steps 1 and 2. 

\vspace{3pt}

Conversely, if $[v - u] = 0$ then $[v - u]_\Sigma \in L_\alpha + L_\beta + L_\kappa$. After a sequence of basepoint isotopies applied to $v$, we acquire a $v'$ with $[v' - u]_\Sigma = 0$. The circle map $\phi':\Sigma \to S^1$ corresponding to $v'$ and $u$ is homotopic to a constant map through maps
\[\phi_t:\Sigma \to S^1 \qquad\text{with}\qquad \phi_0 = 0 \,\text{ and }\,\phi_1 = \phi\]
After applying basepoint spirals to $v'$, we can assume that $\phi_t = 0$ near each basepoint $p_\eta$. Finally, it is simple to see that there is a unique family of vector-fields $u_t$ in good position to $u$ that recovers $\phi_t$, i.e. such that
\[
\phi_t = \op{arg}\big(\frac{\tau \circ u_t}{\tau \circ u}\big)
\]
This $v'$ is isotopic to $u$, and $v$ is isotopic to $u$ after basepoint isotopies and spirals. \end{proof}

\begin{cor} \label{cor:unique_combings} If $T$ is a marked trisection diagram for a $4$-manifold $X$ with $H_1(X) = 0$ then any two combings are isotopic after a sequence of basepoint isotopies and basepoint spirals.
\end{cor}

\subsection{Lefschetz fibrations} We conclude this section by discussing a rich source of examples of trisection diagrams: Lefschetz fibrations and their diagrams. Our main references are Ozbagci-Stipsicz \cite{ozbagci2004surgery} and Castro-Ozbagci \cite{castro2019trisections}.

\begin{definition} \cite[Def.~10.1.1]{ozbagci2004surgery} An \emph{achiral Lefschetz fibration} $\pi$ on a compact, connected $4$-manifold with boundary $X$ is a smooth map $f:X \to S$ to a smooth surface with boundary $S$ such that:
\begin{itemize}
    \item The set of critical points $P = \{p_1,\dots,p_m\}$ of $f$ is finite and contained in $\op{int}(X)$.
    \vspace{2pt}
    \item Each critical point $p$ has a neighborhood $V$ and a chart $U \simeq V$ from a subset $U \subset \C^2$ with
    \[
    f(z_1,z_2) = z_1^2 + z_2^2 \qquad\text{in the chart }U
    \]
\end{itemize}
The \emph{fiber} $F$ of the Lefschetz fibration is the surface with boundary given by the fiber of $f$ at a non-critical value. If $S$ is diffeomorphic to $D^2$, there is an associated open book on the boundary
\[(B,\pi)  \qquad\text{with}\qquad B = f^{-1}(0) \cap \partial X\quad\text{and}\quad\pi = \op{arg} \circ f\]
Here $\op{arg}$ is the argument map sending a non-zero complex number to its corresponding angle.

\end{definition}

\begin{remark} We will also typically assume that the critical values of different critical points are distinct, and that $0$ is not a critical value. Any Lefschetz fibration is equivalent to such a fibration after a small isotopy. \end{remark}

A Lefschetz fibration is specified (up to isotopy and isomorphism) by the data of a certain type of diagram. We fix the following terminology.

\begin{definition} An \emph{achiral Lefschetz diagram} $(F,L,s)$ consists of an oriented surface with boundary $F$ equipped with:
\begin{itemize}
    \item A collection of simple, closed curves $L = L_1,\dots,L_m$.
    \vspace{3pt}
    \item An associated sign $s_i \in \{+1,-1\}$ for each closed curve $L_i$.
\end{itemize}
An (ordinary) Lefschetz diagram is an achiral Lefschetz fibration with only positive signs.\end{definition}

\begin{thm} \label{thm:Lef_diagram_to_Lef_fibration} Any achiral Lefschetz diagram $(F,L,s)$ with $m$ curves determines an achiral Lefschetz fibration $\pi:X \to D^2$ over the disk with fiber $F$ and $m$ critical points. 
\end{thm}

\begin{figure}[h!]
\centering
\includegraphics[width=.7\textwidth]{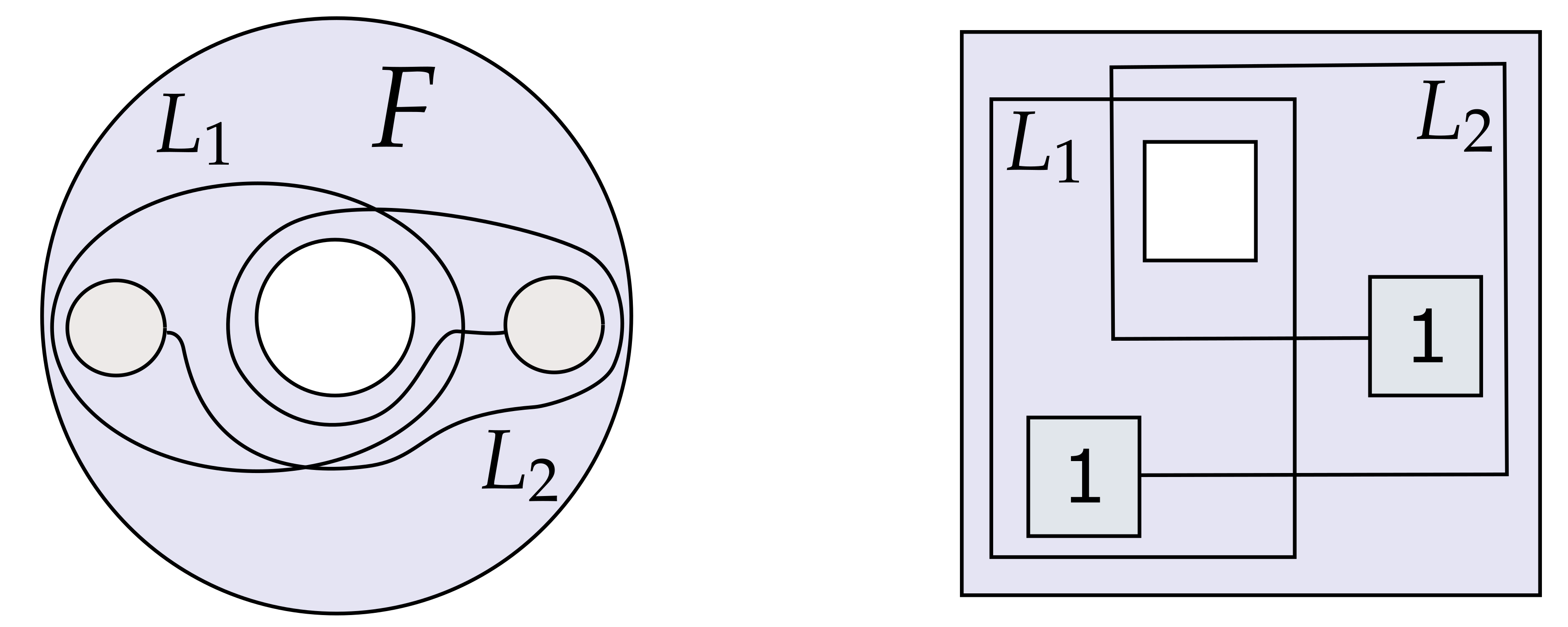} 

\caption{Two pictures of the same Lefschetz fibration. The underlying surface $F$ is a genus $1$ surface with two boundary components (where the grey circles represent the handle, as in a Heegaard diagram). There are two curves $L_1$ and $L_2$. The image on the right is a more rectilinear picture where the two ends of the handle are labelled.}\label{fig:simple_lefschetz}
\end{figure} 

\begin{remark} Let us remark on the topology behind Theorem \ref{thm:Lef_diagram_to_Lef_fibration}. For a more complete discussion, see Ozbagci-Stipsicz \cite[\S 10.1]{ozbagci2004surgery}.

\vspace{3pt}

An achiral Lefchetz fibration over the disk $D^2$ with critical values $Q \subset D^2$ is determined up to isomorphism by the fiber $F$ and the monodromy homomorphism
\[
\Phi_f:\pi_1(D^2 \setminus Q,0) \to \op{MC}(F)
\]
sending a loop based at $p$ to the monodromy of the fibration around the loop (as an element of the mapping class group). The group $\pi_1(D^2 \setminus Q,0)$ is generated by a set of simple loops around each critical point, and the monodromy around each such loop is a Dehn twist, whose handedness depends on the orientation of the singularity chart \cite[Prop 10.1.5]{ozbagci2004surgery} . 

\vspace{3pt}

A Lefschetz fibration is thus determined by a surface, a set of curves and a specification of righthanded or lefthanded Dehn twists along each curve. This is the data of a Lefschetz diagram.\end{remark}

There is an important algorithm for converting a Lefschetz diagram into a trisection diagram introduced by Castro-Ozbagci \cite{castro2019trisections} (and in the relative case, discussed by Castro-Gay-Pinzon-Caicedo \cite{cgpc2018relativetrisections}).

\begin{algorithm}[Lefschetz to trisection] \label{alg:Lef_to_Tri} Fix a smooth 4-manifold $X$ with boundary $\partial X$, and let $f:X \to D^2$ be a Lefschetz fibration inducing an open book $\pi$ on $\partial X$. Fix a Lefschetz diagram
\[
(F,L,s) \qquad\text{with $m$ curves and fiber of genus $p$ with $b$ boundary components}
\]
We now describe an algorithm for converting the diagram $(F,L,s)$ into a (balanced) relative trisection representing $(X,\pi)$, denoted by
\[T = (\Sigma,\alpha,\beta,\kappa) \qquad\text{of type}\qquad (3m + p, 2p + b - 1; p,b) \]
We will apply each step of this algorithm to the Lefshetz diagram in Figure~\ref{fig:simple_lefschetz} as an example.

\vspace{5pt}

{\bf Step 1.} Divide each curve $L_i$ into two arcs, a red arc $R_i$ and a green arc $G_i$, meeting at along their boundary points. On the Lefschetz diagram in Figure \ref{fig:simple_lefschetz}, this yields the following picture:

\begin{center}
\includegraphics[width=.3\textwidth]{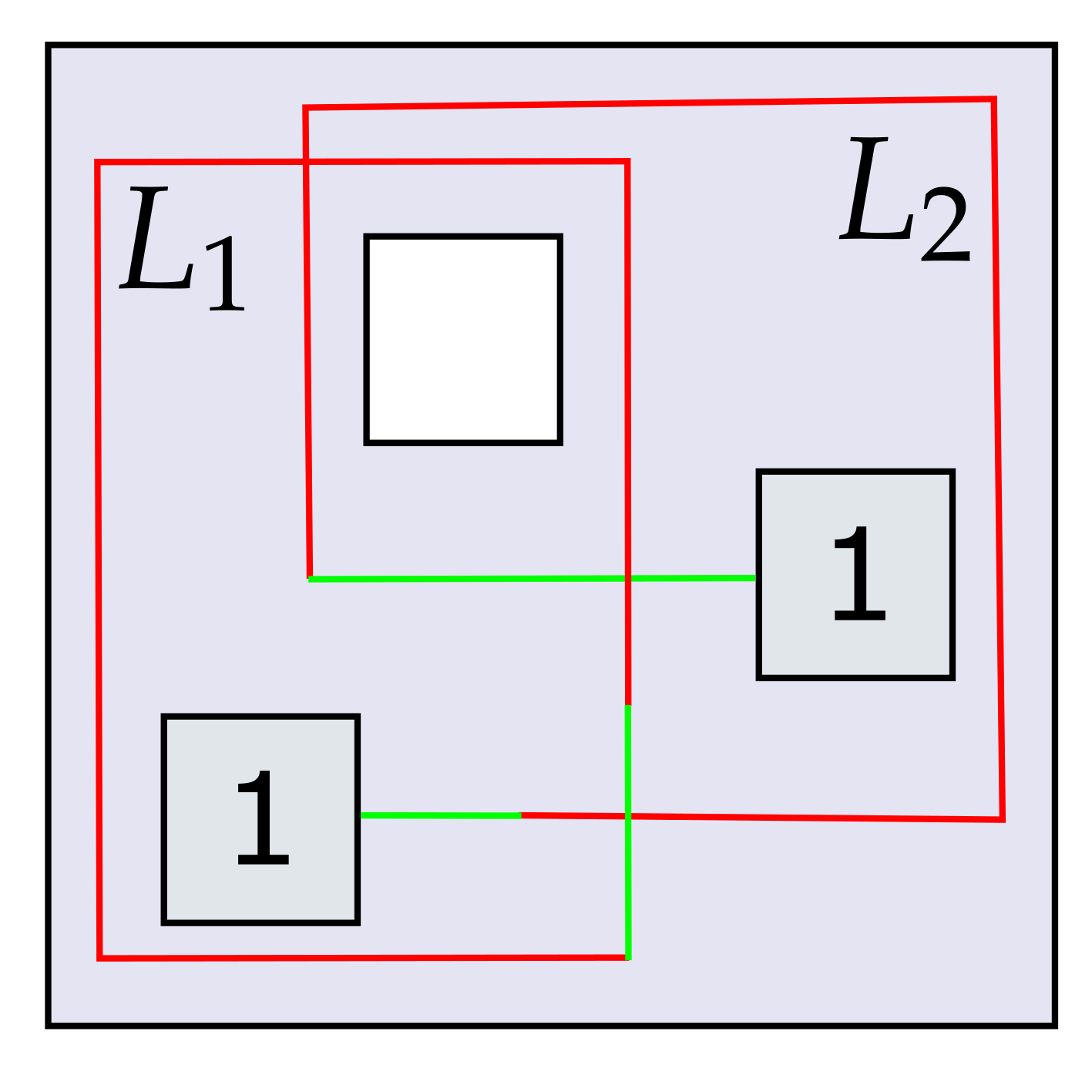}
\end{center}

{\bf Step 2.} Inductively, isotope each curve $L_i$ so that all intersections with the curve $L_j$ for $j < i$ are contained in $R_i \cap G_j$, i.e. are intersections between the red part of $L_i$ and the green part of $L_j$. On the Lefschetz diagram in Figure \ref{fig:simple_lefschetz}, this yields the following picture:

\begin{center}
\includegraphics[width=.3\textwidth]{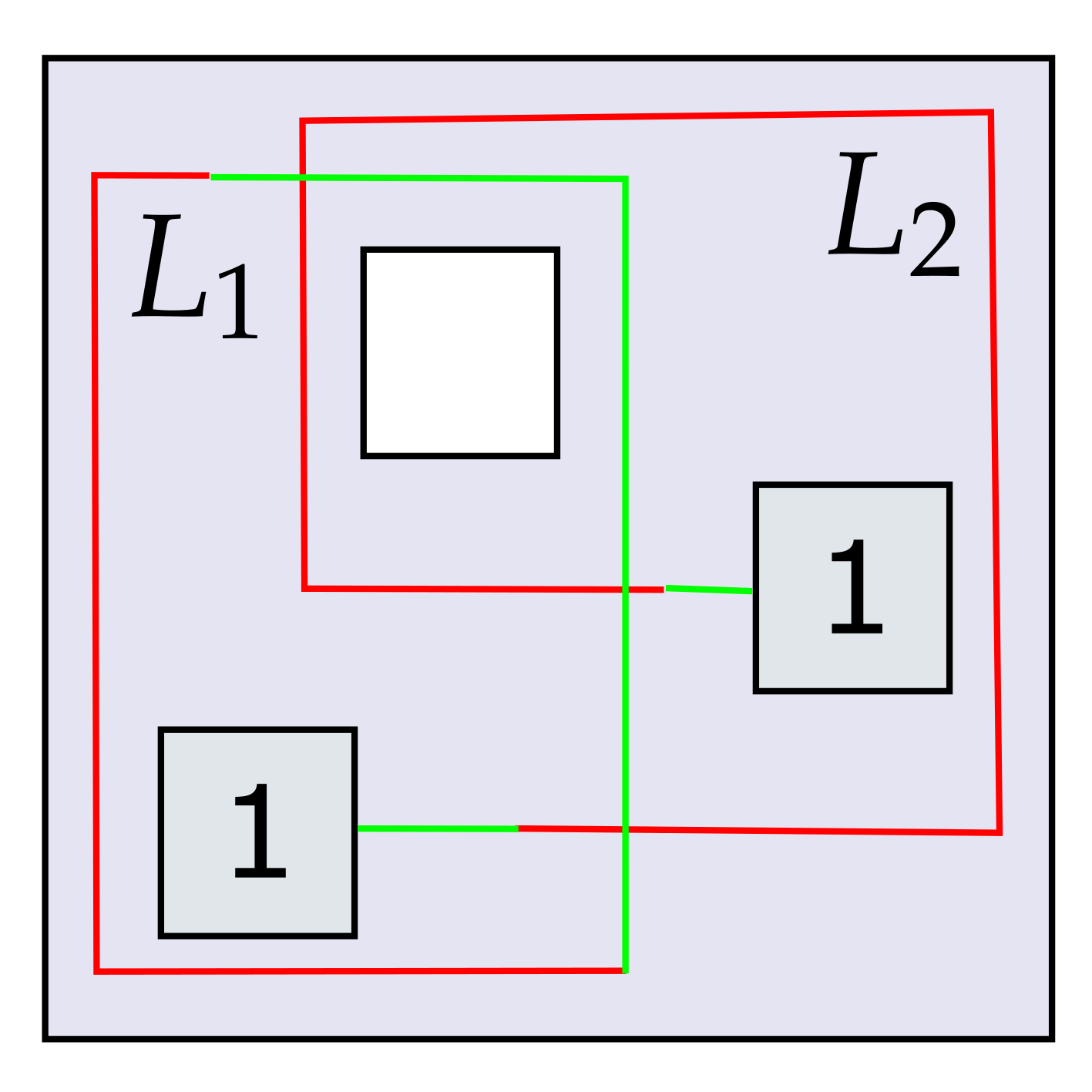}
\end{center}

{\bf Step 3.} For each curve $L_i$, choose a neighborhood $U_i$ of the two points at the intersection  of the red and green parts, of the following form:

\vspace{5pt}

\begin{center}
\includegraphics[width=.4\textwidth]{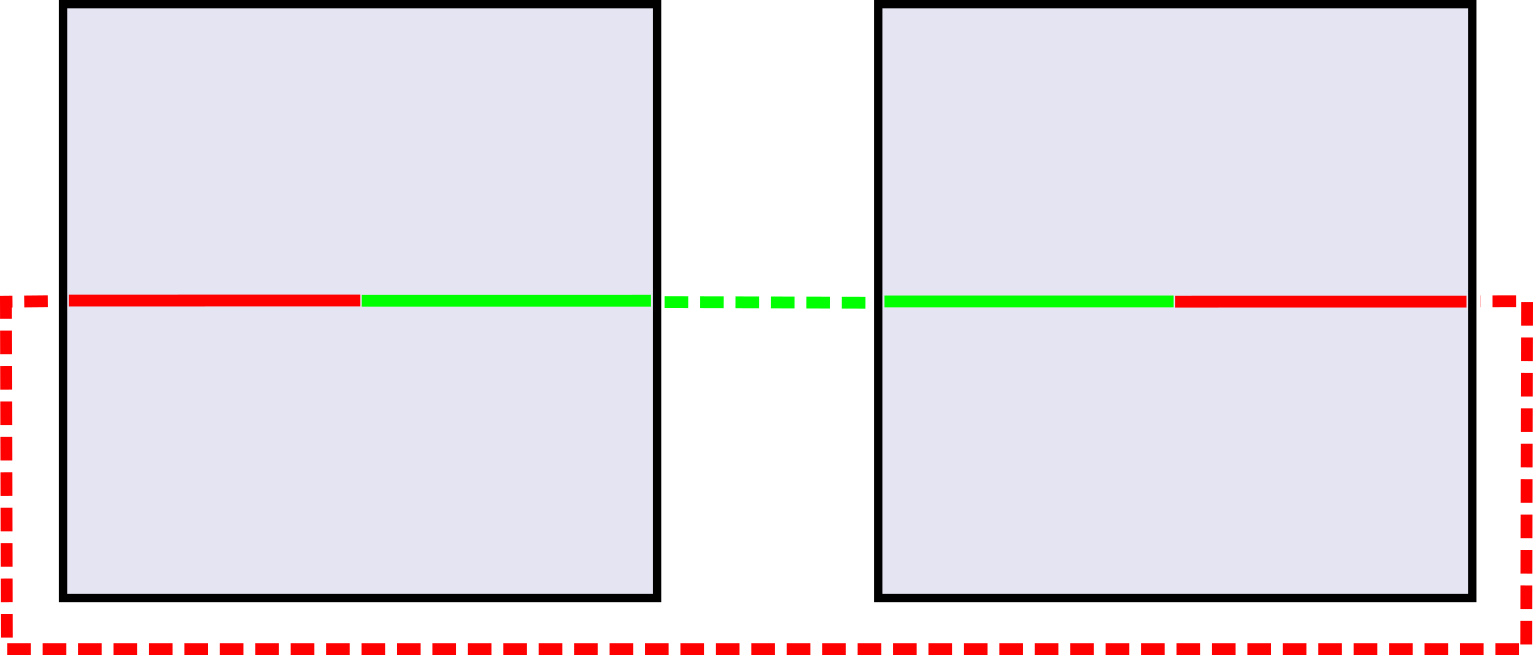}
\end{center}

\vspace{5pt}

{\bf Step 4.} For each curve $L_i$, replace the chosen neighborhood $U_i$ with the picture on the left (if the sign $s_i$ is positive) or the pucture on the right (if the sign $s_i$ is negative):

\vspace{5pt}

\begin{center}
\includegraphics[width=.8\textwidth]{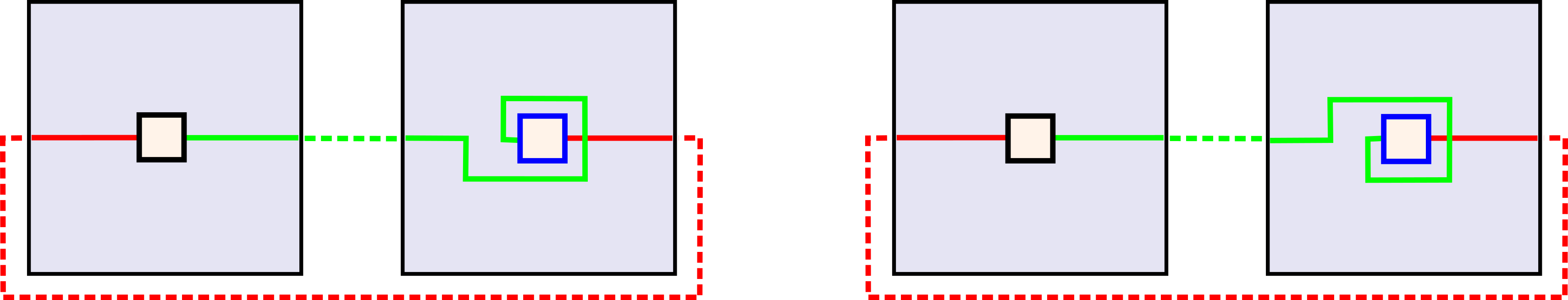}
\end{center}

\noindent Applying Steps 3 and 4 to the Lefschetz diagram in Figure \ref{fig:simple_lefschetz} yields the following picture:

\vspace{5pt}

\begin{center}
\includegraphics[width=.3\textwidth]{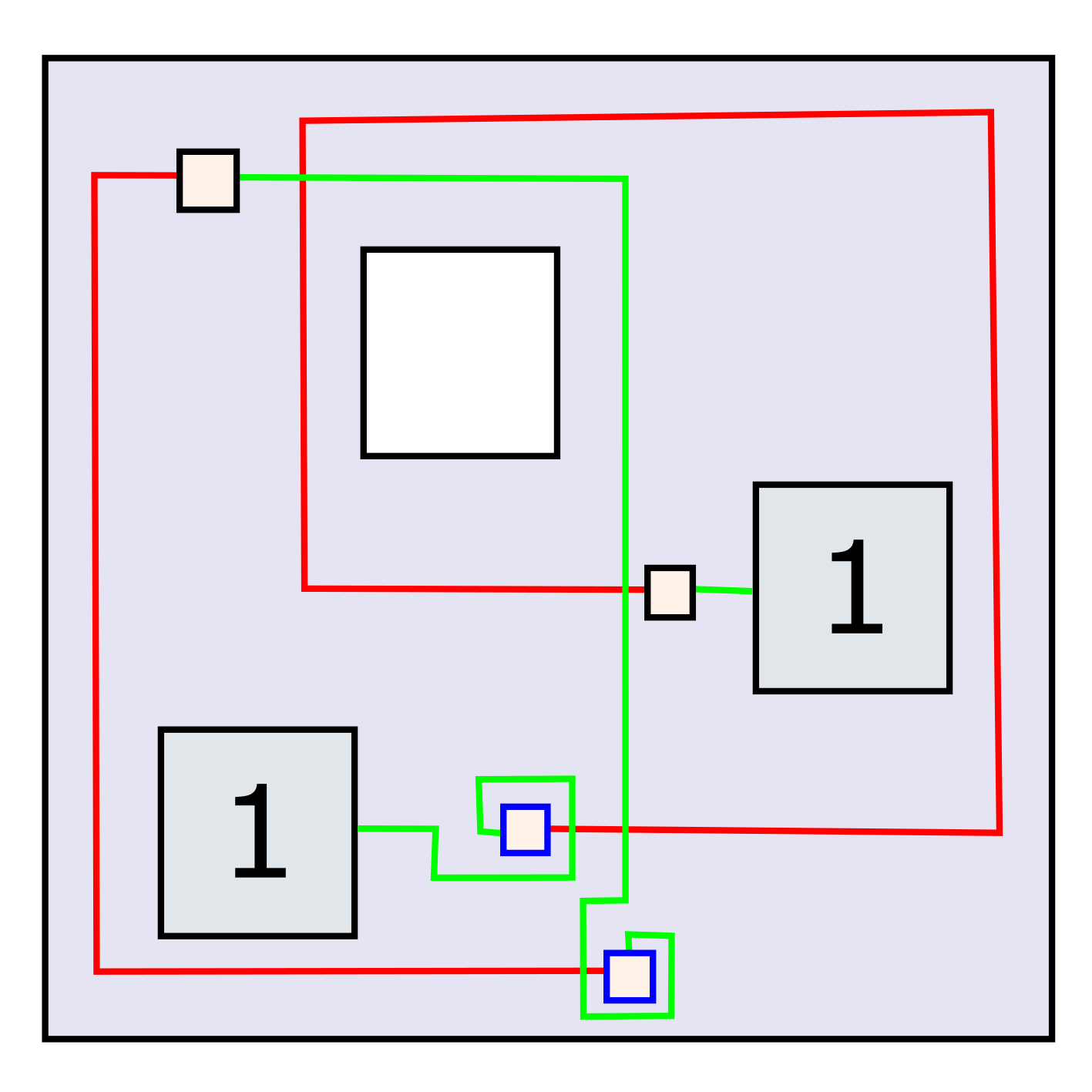}
\end{center}

\end{algorithm}

We will apply this algorithm to some more sophisticated examples in Section \ref{sec:stein_nuclei}.

\section{Trisection Invariant} \label{sec:main_results} We now proceed with the proof of the main results of this paper. In particular, we construct the (non-semisimple, relative) trisection brackets and trisection invariants. We also explain the basic properties of these invariants.

\begin{setup} \label{set:trisection_invariant} We fix the following setup to formulate the trisection invariants in this section.
\begin{itemize}
    \item A balanced Hopf triplet $\mathcal{H}$ with Hopf algebras and pairings denoted by
\[
H_\mu \quad\text{for}\quad \mu \in \{\alpha,\beta,\kappa\} \qquad\text{and}\qquad \langle -\rangle_{\mu\nu}\quad\text{for}\quad \mu\nu \in \{\alpha\beta,\beta\kappa,\kappa\alpha\}\]
\item A right cointegral on each Hopf algebra in the Hopf triplet, denoted by \[e^\mu \qquad\text{or}\qquad e^{H_\mu} \qquad\text{for}\qquad \mu \in \{\alpha,\beta,\kappa\}  \]
\item A marked trisection diagram with oriented curves, denoted by
\[T = (\Sigma;\alpha,\beta,\kappa; C) \qquad\text{of type}\qquad (g,k_1,k_2,k_3;p,b)\]
\item A basepoint on each closed curve $\mu_i$ in the trisection $T$, denoted by
\[p_{\mu_i} \in \mu_i\]
\item A singular combing $v$ on the marked trisection $T$ with singularities at the basepoints $p_{\mu_i}$.
\end{itemize}

 \noindent There is a generalized integral corresponding the right integral $e^\mu$ and any proper half-integer $\theta$ (see Definition \ref{def:generalized_cointegral}). We denote this tensor by
\[e^\mu_\theta \qquad\text{or}\qquad e^{H_\mu}_\theta \qquad\text{for any }\mu \in \{\alpha, \beta, \kappa\}\]We will drop the superscript $\mu$ when the Hopf algebra is clear from context. We note that
\[e_{-1/2}^\mu = e^\mu \qquad\text{and}\qquad e_{1/2}^\mu \text{ is a left integral}\]

\end{setup}

\subsection{Rotation numbers} We start by introducing various rotation numbers associated to a singular combing. To start, consider an embedded arc and a trivialization
\[
\gamma \subset \Sigma \qquad\text{and}\qquad \tau:T\Sigma|_\gamma \simeq \R^2
\]
Assume that the combing $v$ is non-vanishing along $\gamma$. Then we can associate a rotation number
\[
\text{rot}(\gamma,v,\tau) \in \R
\]
as follows. Choose a parametrization of the arc $f:[0,1] \simeq \gamma$. We can define a map
\[\phi:[0,1] \to S^1 \simeq \R/\Z \qquad\text{given by}\qquad \phi(t) = \frac{\tau \circ v(f(t))}{|\tau \circ v(f(t))|}\]
This map lifts to a map $\tilde{\phi}:[0,1] \to \R$, and we define $\text{rot}(\gamma,v,\tau) := \tilde{\phi}(1) - \tilde{\phi}(0)$. The following lemma is simple and left as an exercise.

\begin{lemma} \label{lem:continuity_of_rot} The rotation number $\text{\rm rot}(\gamma,v,\tau)$ is well-defined (i.e.~independent of $f$) and continuous in the arguments $\gamma,v$ and $\tau$.
\end{lemma}

We can now associate (half) integer rotation numbers to the curves and intersection points in the trisection diagram, as follows.

\begin{definition}[Rotation of curve] \label{def:rot_of_curve} The rotation number of a curve $\mu_i$ in the trisection $T$ with respect to the combing $v$ is the proper half-integer
\[
\theta(\mu_i) \in \Z + \frac{1}{2}
\]
defined as follows. Recall that the combing $v$ is non-vanishing along $\mu_i$ except in a neighborhood $U$ of $p_{\mu_i}$, where it has a singularity at $p_{\mu_i}$, is tangent to $\mu_i$ and is pointing towards $p_{\mu_i}$ (see Definition \ref{def:combing_on_diagram}). Choose an arc $\gamma \subset \mu_i$ with endpoints in $U$ and choose a trivialization $\tau:T\Sigma|_{\mu_i} \simeq \R^2$ along $\mu_i$ identifying $(1,0) \in \R^2$ with an oriented tangent vector-field to $\mu_i$. Then we let
\[\theta(\mu_i) = \text{rot}(\gamma,v,\tau)\]
This is a half-integer since $v$ is tangent to $\gamma$ and points in a constant direction in $U$. \end{definition}

\begin{definition}[Rotation of intersection] \label{def:rot_of_intersection} The rotation numbers of an intersection point $p$ of two curves $\mu_i$ and $\nu_j$ in the trisection $T$ with respect to $v$, where $\mu\nu \in \{\alpha\beta, \beta\kappa, \kappa\alpha\}$, are the integers
\[
\theta_{\mu_i}(p) \in \Z \qquad\text{and}\qquad \theta_{\nu_i}(p) = -\theta_{\mu_i}(p) \in \Z
\]
defined as follows. Let $\gamma_\mu \subset \mu_i$ and $\gamma_\nu \subset \nu_j$ denote the oriented arcs connecting $p_{\mu_i}$ to the intersection point $p$ in $\mu_i$ and $\nu_j$, respectively. Choose trivializations
\[
\tau_\mu:T\Sigma|_{\mu_i} \simeq \R^2 \qquad\text{and}\qquad \tau_\nu:T\Sigma_{\nu_j} \simeq \R^2
\]
sending the oriented tangent vector $\mu_i$ and $\nu_j$ to $(1,0)$, respectively. Also assume that these tangent line to $\mu_i$ is sent by $\tau_\nu$ to a perpendicular line to the tangent line to $\nu_j$. Then we set
\[\theta_{\mu_i}(p) = 2(\text{rot}(\gamma_\mu,v,\tau_\mu) - \text{rot}(\gamma_\nu,v,\tau_\nu)) + \frac{1}{2}\]
\[\theta_{\nu_j}(p) =  -\theta_{\mu_i}(p) = 2(\text{rot}(\gamma_\nu,v,\tau_\nu) - \text{rot}(\gamma_\mu,v,\tau_\mu)) - \frac{1}{2}\]
It can easily be checked that $\theta_{\mu_i}(p)$ and $\theta_{\nu_i}(p)$ are even integers if $\mu_i$ and $\nu_j$ intersect positively, and odd integers otherwise. \end{definition}

\begin{lemma} The rotation numbers $\theta(\mu_i)$ and $\theta_{\mu_i}(p)$ are independent of choices of trivialization, and invariant under isotopy of the trisection and the combing.
\end{lemma}

\begin{proof} By Lemma \ref{lem:continuity_of_rot}, the rotation number is continuous with respect to the trivialization, arc, and non-vanishing vector-field. This implies that the rotation numbers in Definitions \ref{def:rot_of_curve} and \ref{def:rot_of_intersection} are constant in these terms under isotopy, since they are half-integer valued.  The independence of trivialization follows since the space of trivializations in both cases is contractible.\end{proof}

\subsection{Trisection bracket} We are now ready to formally define the trisection bracket and state its key properties. We defer several proofs to their own sections for the sake of readability.

\begin{definition} The \emph{trisection bracket} associated to a trisection $T$, combing $v$, Hopf triplet $\mathcal{H}$ and cointegrals $e$ as in Setup \ref{set:trisection_invariant} is the scalar
\[\langle (T,v) \rangle_{\mathcal{H}} \in k\]
defined as follows. First, for each curve $\mu_i$ in $T$, associate the tensor $\Gamma(\mu_i)$ as shown in Figure \ref{fig:circle_tensor}.

\begin{figure}[h!]
    \centering
    \includegraphics[scale=.4, valign = c]{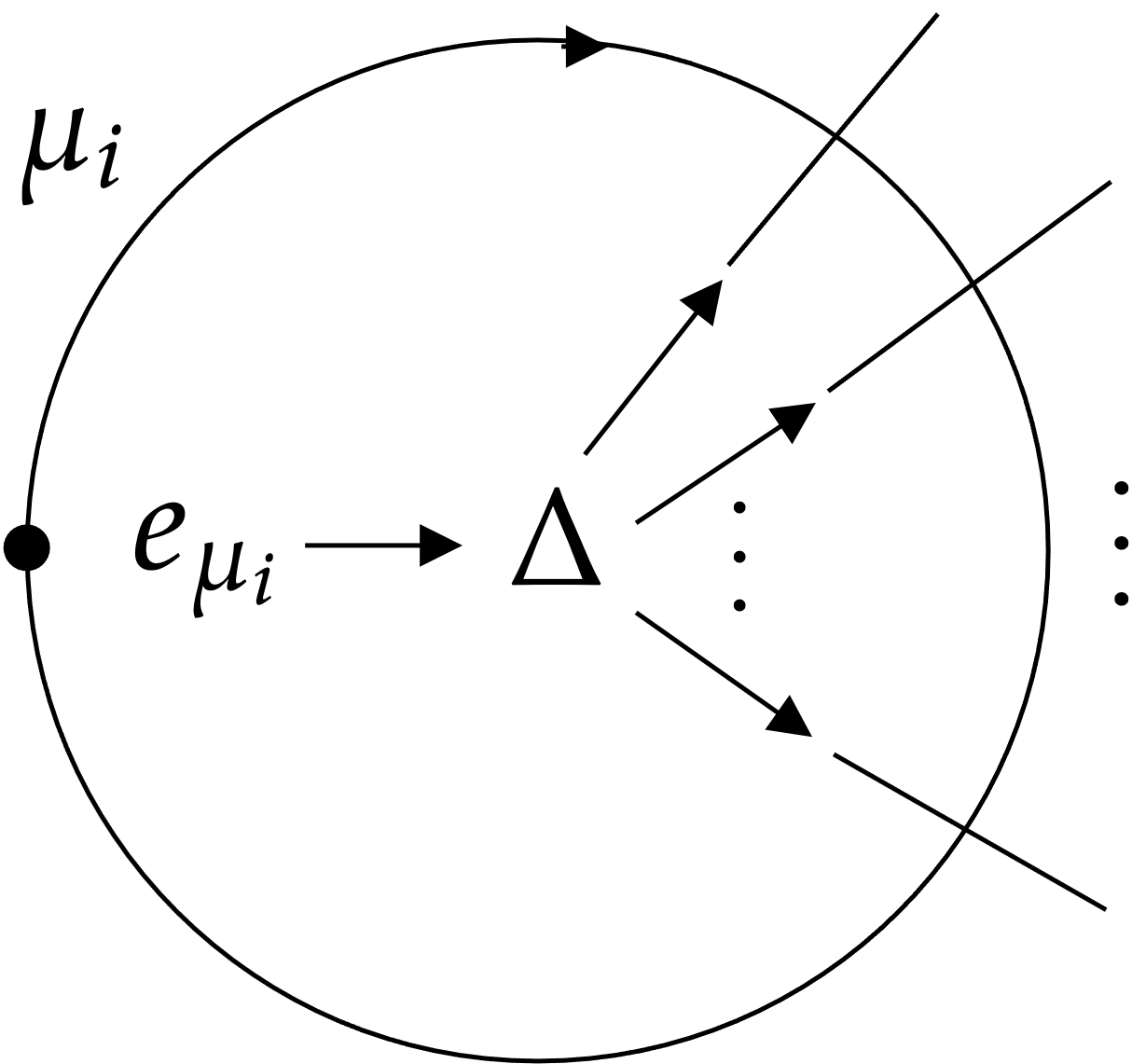}
    \caption{Tensors assigned to trisection circle $\mu_i$, where $\mu \in \{\alpha, \beta, \kappa\}$.}\label{fig:circle_tensor}
\end{figure}

\noindent Here the generalized $\Delta$ tensor has an outgoing edge for each intersection point of the $\mu_i$ curve with another curve, and the outgoing edges are labelled by these intersection points in order, according to the base point (depicted by the black dot on $\mu_i$) and the chosen orientation of $\mu_i$. $\Gamma(\mu_i)$ is acquired by contracting the tensor $\Delta$ with the generalized cointegral $e_{\mu_i}$ along the single input of $\Delta$. Note that both tensors are defined using the Hopf algebra $H_{\mu}$. 

\vspace{3pt}

Next, for each intersection $p$ between a $\mu_i$ curve and a $\nu_j$ curve for $\mu\nu \in \{\alpha\beta, \beta\kappa, \kappa\alpha\}$, $p$ is associated with both an outgoing edge of the tensor assigned to $\mu_i$ and an outgoing edge of the tensor assigned to $\nu_j$. We assign to $p$ the tensor depicted in the dashed box on either side of the identity in Figure \ref{fig:intersection_tensor}. 

\begin{figure}[h!]
    \centering
    \includegraphics[scale=.4, valign = c]{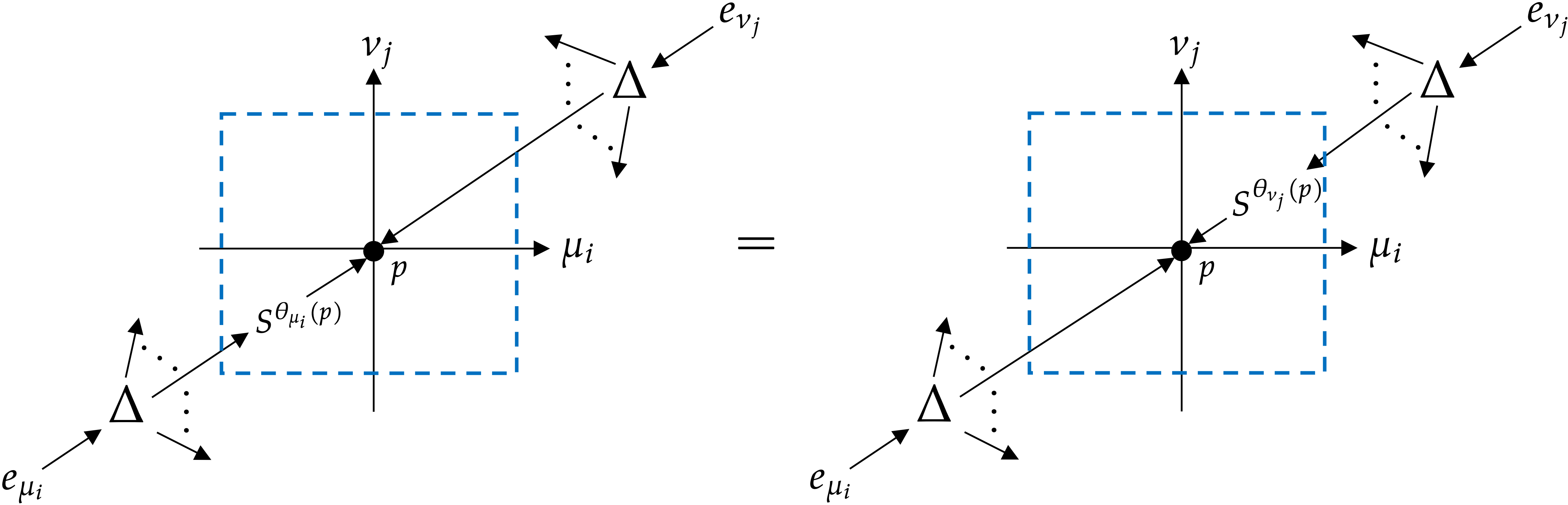}
    \caption{Tensors assigned to an intersection between a $\mu_i$ curve and a $\nu_j$ curve.}\label{fig:intersection_tensor}
\end{figure}

\noindent The bracket $\langle (T,v) \rangle_{\mathcal{H}}$ is obtained by contracting all the tensors assigned to all curves and all intersections on $T$ as described in the above paragraph.
\end{definition}
\begin{remark}\label{rem:choose_integral_1}
    The bracket $\langle (T,v) \rangle_{\mathcal{H}}$ obviously depends on the choice of a right integral in each Hopf algebra, albeit in a simple way. Explicitly, if we choose a new set of right integrals,
    \begin{align*}
        \tilde{e}_R^{H_{\mu}} := c_{\mu}e_R^{H_{\mu}}, \quad \mu \in \{\alpha, \beta, \kappa\},
    \end{align*}
    then for a diagram $(T,v)$ of type $(g,k_1, k_2, k_3; p,b)$, the bracket of $(T,v)$ with the new integrals picks up a multiplicative factor $(c_{\alpha}c_{\beta}c_{\kappa})^{g-p}$.
\end{remark}

\begin{remark}
From Definition \ref{def:triplet_phase_balanced}, the phase of $\mathcal{H}$ is $q_{\mathcal{H}} := q_{\alpha} = q_{\beta} = q_{\kappa} = \pm 1$. Define the equivalence relation on the ground field $k$ such that $x \sim y$ if $x = y$ or $x = q_{\mathcal{H}}y$. Denote by $k/\sim$ the set of equivalence classes. Clearly, the multiplication in $k$ induces a well-defined multiplication in $k/\sim$. For a reason that will be clear later (see Proposition \ref{prop:inv_combing_move}), we think of $\langle (T,v) \rangle_{\mathcal{H}}$ as taking values in $k/\sim$,
    \begin{equation*}
        \langle (T,v) \rangle_{\mathcal{H}} \in k/\sim.
    \end{equation*}
\end{remark}
\begin{remark}
The equality in Figure \ref{fig:intersection_tensor} shows that one can slide the antipode from one side of the pairing to the other as long as its exponent is changed accordingly. For this reason, when we need to focus on the tensors associated with one specific curve, say a $\mu_i$ curve, we can assume that for all the intersections $p$ on this curve, the corresponding antipode $S^{\theta_{\mu_i}(p)}$ has been slid over to the $\mu_i$ side, with 
\begin{align*}
    \theta_{\mu_i}(p) := 2(\theta(\mu_i;p) - \theta(\nu_j;p)) \pm \frac{1}{2}\,,
\end{align*}
with the $\pm$ sign depending on whether $\mu\nu \in \{\alpha\beta, \beta\kappa, \kappa\alpha\}$ or $\nu\mu \in \{\alpha\beta, \beta\kappa, \kappa\alpha\}$.
\end{remark}

\subsection{Properties of bracket} The trisection bracket satisfies several key invariance properties. To avoid being distracted by technicalities, we defer the complicated proofs to Section \ref{subsec:proof1}.

\vspace{3pt}

First, the trisection bracket is independent of the choice of orientation in Setup \ref{set:trisection_invariant}. 
\begin{prop}\label{prop:bracket_orientation}
The bracket $\langle (T,v) \rangle_{\mathcal{H}}$ does not depend on the choice of orientations on the curves of $T$.
\end{prop}

Second, recall that we are interested in equivalence classes of combings modulo combing moves (see Definition~\ref{def:combing_move}). The trisection bracket is invariant under these moves.

\begin{prop}\label{prop:inv_combing_move}
The bracket $\langle (T,v) \rangle_{\mathcal{H}} \in k/\sim$ is invariant under the combing moves. More precisely:
\begin{itemize}
    \item $\langle (T,v) \rangle_{\mathcal{H}} \in k$ is invariant under basepoint isotopy on $v$.
    \vspace{3pt}
    \item $\langle (T,v) \rangle_{\mathcal{H}} \in k$ changes by a multiplicative factor of $q_{\mathcal{H}}$ after a basepoint spiral.
\end{itemize}
The trisection invariant can therefore be viewed as an invariant
\[\langle (T,[v]) \rangle_{\mathcal{H}} := \langle (T,v) \rangle_{\mathcal{H}} \in k/\sim \qquad\text{of the combing class }[v] \in \text{Comb}(T)\]
\end{prop}

\begin{remark} When the corresponding 4-manifold $X$ has $H_1(X) = 0$), we write
\[\langle T \rangle_{\mathcal{H}} = \langle (T,v) \rangle_{\mathcal{H}} \qquad\text{for any singular combing $v$ on $T$}\] \end{remark}

Third, we consider the invariance of the trisection bracket under trisection moves (see Section \ref{subsec:trisections}). Recall from Section \ref{subsec:trisections} that if $T$ and $T'$ are two trisection diagrams that differ by a single trisection move, then there is a canonical identification
\[\text{Comb}(T) \simeq \text{Comb}(T')\]
equivariant with respect to the action of the first homology group of the 4-manifold. For a given combing class $[v]$ on $T$, we use $[v]$ to also denote the corresponding class on $T'$.

\begin{prop}\label{prop:inv_trisection_move}
    The bracket $\langle (T,[v]) \rangle_{\mathcal{H}}$ is invariant under handle slides, two-point moves and three-point moves performed on the diagram $T$. That is, if $T$ and $T'$ differ by one of these moves then
    \[\langle (T,[v])\rangle_{\mathcal{H}} = \langle (T',[v])\rangle_{\mathcal{H}} \]
\end{prop}

Finally, we recall that there is a connect sum operation on combings that descends to combing equivalence classes. That is, there is a canonical map
\[\text{Comb}(S) \times \text{Comb}(T) \to \text{Comb}(S \natural T) \qquad\text{given by}\qquad [u] \natural [v]  = [u \natural v]\]
The final property that we discuss is the additivity of the trisection bracket under connect sum. This property is essentially immediate from the construction.

\begin{prop} \label{prop:boundary_sum_property} The trisection bracket is multiplicative with respect to boundary sum. That is, if $(S,u)$ and $(T,v)$ are combed, marked trisections then
\[
\langle (S \natural T,[u \natural v])\rangle_{\mathcal{H}} = \langle (S,[u])\rangle_{\mathcal{H}} \; \cdot \; \langle (T,[v])\rangle_{\mathcal{H}}
\]
\end{prop}

We next define the normalized version of the bracket that is invariant under stabilization. Recall from Definition \ref{def:stabilization} that the standard, stabilization trisections
\[T^{\op{st}}_j \qquad\text{for}\qquad j \in \{1,2,3\} \qquad\text{of type}\qquad  (1,\delta_{1j},\delta_{2j},\delta_{3j};0,1)\]
are three trisections for the $4$-ball $D^4$ with the standard open book $\pi_{\op{std}}$ with disk-like pages.
In particular, given any trisection $T$ of type $(g,k_1, k_2, k_3; p,b)$, the stabilization $T \natural T^{\op{st}}_j$ has type
\[(g+1,k_1+\delta_{1j},k_2 + \delta_{2j},k_3 + \delta_{3j};p,b)\]
Note also that there is a single equivalence class of combing on $T^{\op{st}}_j$ by Corollary \ref{cor:unique_combings}. We now define the trisection invariant of the data $\mathcal{H},e,T,[v]$ as in Setup \ref{set:trisection_invariant} as follows:
\[
\tau_{\mathcal{H}}(T,[v]) := \langle T^{st}_1\rangle_{\mathcal{H}}^{-k_1} \cdot \langle T^{st}_2\rangle_{\mathcal{H}}^{-k_2}\ \cdot \langle T^{st}_3\rangle_{\mathcal{H}}^{-k_3} \cdot \langle (T,[v]) \rangle_{\mathcal{H}}.
\]

\begin{remark}\label{rem:choose_integral_2} Note that $\tau_{\mathcal{H}}$ depends on the choice of right integerals $e$, which we are suppressing in the notation. However, the dependence on this choice is very mild. Let $(X, \pi)$ be the 4-manifold represented by the trisection $T$ of type $(g,k_1, k_2, k_3; p,b)$, then the Euler characteristic of $X$ is
\begin{equation}\label{eqn:euler of X}
    \chi(X) = g - k_1 -k_2 - k_3 + 3p + 2b -1.
\end{equation}
Note that both $p$ and $b$ are invariants of $X$ independent of $T$. For a new set of cointegrals 
\begin{align*}
        \tilde{e}_R^{H_{\mu}} := c_{\mu}e_R^{H_{\mu}}, \quad \mu \in \{\alpha, \beta, \kappa\},
    \end{align*}
if we denote by $\tilde{\tau}_{\mathcal{H}}(T, [v])$ the invariant of $(T,v)$ computed using  $\tilde{e}_R^{H_{\mu}}$, then by Remark~\ref{rem:choose_integral_1},
\begin{align*}
    \tilde{\tau}_{\mathcal{H}}(T, [v]) = (c_{\alpha}c_{\beta}c_{\kappa})^{\chi(X)-4p-2b+1}\ \  \tau_{\mathcal{H}}(T, [v])\,.
\end{align*}
\end{remark}

\begin{prop}\label{prop:inv_stabilization} The quantity
    $\tau_{\mathcal{H}}(T, [v])$ is invariant under all trisection moves, i.e.~isotopy, handleslides, stabilization, and destabilization.
\end{prop}

\begin{proof} For handleslides and isotopies, this is simply Proposition \ref{prop:inv_trisection_move}. For (de)stabilizations, we note that if we stabilize by $T^{\op{st}}_i$, then by Proposition \ref{prop:boundary_sum_property} we have
\begin{align*}
    \tau_{\mathcal{H}}(T \natural T^{\op{st}}_1,[v]) &:= \langle T^{st}_1\rangle_{\mathcal{H}}^{-k_1-1} \cdot \langle T^{st}_2\rangle_{\mathcal{H}}^{-k_2}\ \cdot \langle T^{st}_3\rangle_{\mathcal{H}}^{-k_3} \cdot \langle (T \natural T^{\op{st}}_1,[v]) \rangle_{\mathcal{H}} \\
    &= \langle T^{st}_1\rangle_{\mathcal{H}}^{-k_1-1} \cdot \langle T^{st}_2\rangle_{\mathcal{H}}^{-k_2}\ \cdot \langle T^{st}_3\rangle_{\mathcal{H}}^{-k_3} \cdot \langle (T,[v]) \rangle_{\mathcal{H}} \cdot \langle T^{\op{st}}_i\rangle_{\mathcal{H}} \\
    &= \tau_{\mathcal{H}}(T,[v])
\end{align*}
The same proof holds for stabilization by the other stabilization trisections. \end{proof}

Finally, we can formulate the main invariant of the paper and prove the main results. Fix a Hopf triplet $\mathcal{H}$ and cointegrals $e$ as in Setup \ref{set:trisection_invariant} that satisfies
\[\langle T^{\op{st}}_i\rangle_{\mathcal{H}} \neq 0 \qquad\text{for}\qquad i \in \{1,2,3\}\]
Fix a pair $(X,\pi)$ of a smooth 4-manifold $X$ with non-empty, connected boundary and a marked open book $\pi$ on $\partial X$.

\begin{definition}\label{def:trisection_inv_pair} 
The \emph{trisection invariant} $\tau_{\mathcal{H}}(X,\pi)$ of the pair $(X,\pi)$ with respect to $(\mathcal{H},e)$ is the set
    \begin{align}
        \tau_{\mathcal{H}}(X, \pi) := \Big\{\tau_{\mathcal{H}}(T, [v]) \ \Big|\ [v] \in \text{Comb}(T)\Big\} \subset k/\sim
    \end{align}
where $T$ is any trisection for $(X,\pi)$ and $\text{Comb}(T)$ is the set of combing classes.
\end{definition}

As an immediate consequence of the properties of the bracket stated in Propositions \ref{prop:inv_combing_move}, \ref{prop:inv_trisection_move}, and \ref{prop:inv_stabilization}, we have the main theorem of the paper.
\begin{thm}\label{thm:trisection_inv_main_thm}
Let $\mathcal{H}$ be a balanced Hopf triplet such that  $\langle T^{st}_i \rangle_{\mathcal{H}} \neq 0$ for $i = 1,2,3$. 
 Then $\tau_{\mathcal{H}}(X, \pi)$ is an invariant of the pair $(X, \tau)$, called the trisection invariant.
\end{thm}
\begin{cor}
    Let $\mathcal{H}$  be as in Theorem \ref{thm:trisection_inv_main_thm}, and assume $H_1(X) = 0$. Then $\tau_{\mathcal{H}}(X, \pi)$ is a \emph{scalar} invariant of $(X, \pi)$.
\end{cor}

Let $Y$ be a closed 4-manifold, and $X = Y - D^4$ where $D^4$ is a smooth 4-ball. Endow $\partial X = S^3$ with the standard on open book $\pi_{st}$. Define
\begin{align*}
    \tau_{\mathcal{H}}(Y) := \tau_{\mathcal{H}}(X, \pi_{st})
\end{align*}
\begin{cor}
    Let $\mathcal{H}$  be as in Theorem \ref{thm:trisection_inv_main_thm}. Then $\tau_{\mathcal{H}}(Y)$ is an invariant of a closed 4-manifold $Y$.
\end{cor}
When $\mathcal{H}$ is semisimple, $\tau_{\mathcal{H}}(Y)$ reduces to the 4-manifold invariant introduced in Definition 3.6 of \cite{chaidez20194manifold}. More explicitly, denote by $\tau_{\mathcal{H}, \zeta}(Y)$ the semisimple trisection invariant of \cite{chaidez20194manifold}, where $\zeta$ is a third root of $\langle T_1^{st}\rangle \langle T_3^{st}\rangle \langle T_3^{st}\rangle$. Given a $(g,k)$ trisection diagram $T'$ of $Y$, denote by $T$ the trisection diagram of $X:= Y - D^4$ obtained by removing a disk from the surface of $T'$ (see Lemma \ref{lem:rel_trisection_to_trisection}). Then $T$ has type $(g,k,k,k;0,1)$. Note that $\chi(X) = g - 3k +1$ by Equation \ref{eqn:euler of X} and $\chi(Y) = g - 3k + 2$. Using this trisection $T$ to compute the $\tau_{\mathcal{H}}(Y)$ and comparing with Definition  3.6 of \cite{chaidez20194manifold}, we have,
\begin{equation*}
    \tau_{\mathcal{H}}(Y) =\zeta^{\chi(Y)-2} \ \ \tau_{\mathcal{H}, \zeta}(Y).
\end{equation*}

\subsection{Bracket of stabilization trisections} To define the trisection invariant, we require the previously stated non-degeneracy condition that the stabilization trisections have non-zero bracket.
\[\langle T^{st}_i\rangle_{\mathcal{H}} \neq 0 \qquad\text{for }i = 1,2,3\]
After extensive investigation with computer aid of Hopf algebras up to dimension 11, we have not discovered a non-semisimple Hopf triplet that satisfies this condition. It is plausible that such Hopf triplets exist outside the scope of our investigations, but this will be the subject of future work. In order to make the non-degeneracy condition more explicit, we now give an explicit tensor diagram expression for the brackets $\langle T^{st}_i \rangle_{\mathcal{H}}$. 

\vspace{3pt}

Since $H_1(D^4) = 0$, there is a unique class of combings on each $T^{st}_i$. Moreover, the three trisection diagrams have the same underlying surface $\Sigma_{1,1}$ (the genus-1 surface with one boundary component) and can be obtained from each other by cyclically permuting the labels of their curves in the order $\alpha \mapsto \beta \mapsto \kappa \mapsto \alpha$. In Figure \ref{fig:tri_inv_56} we describe a combing on $\Sigma_{1,1}$ corresponding to $T^{st}_2$, but the same combing also applies to $T^{st}_1$ and $T^{st}_3$. 

\begin{figure}[h!]
    \centering
    \includegraphics[scale=.44, valign = c]{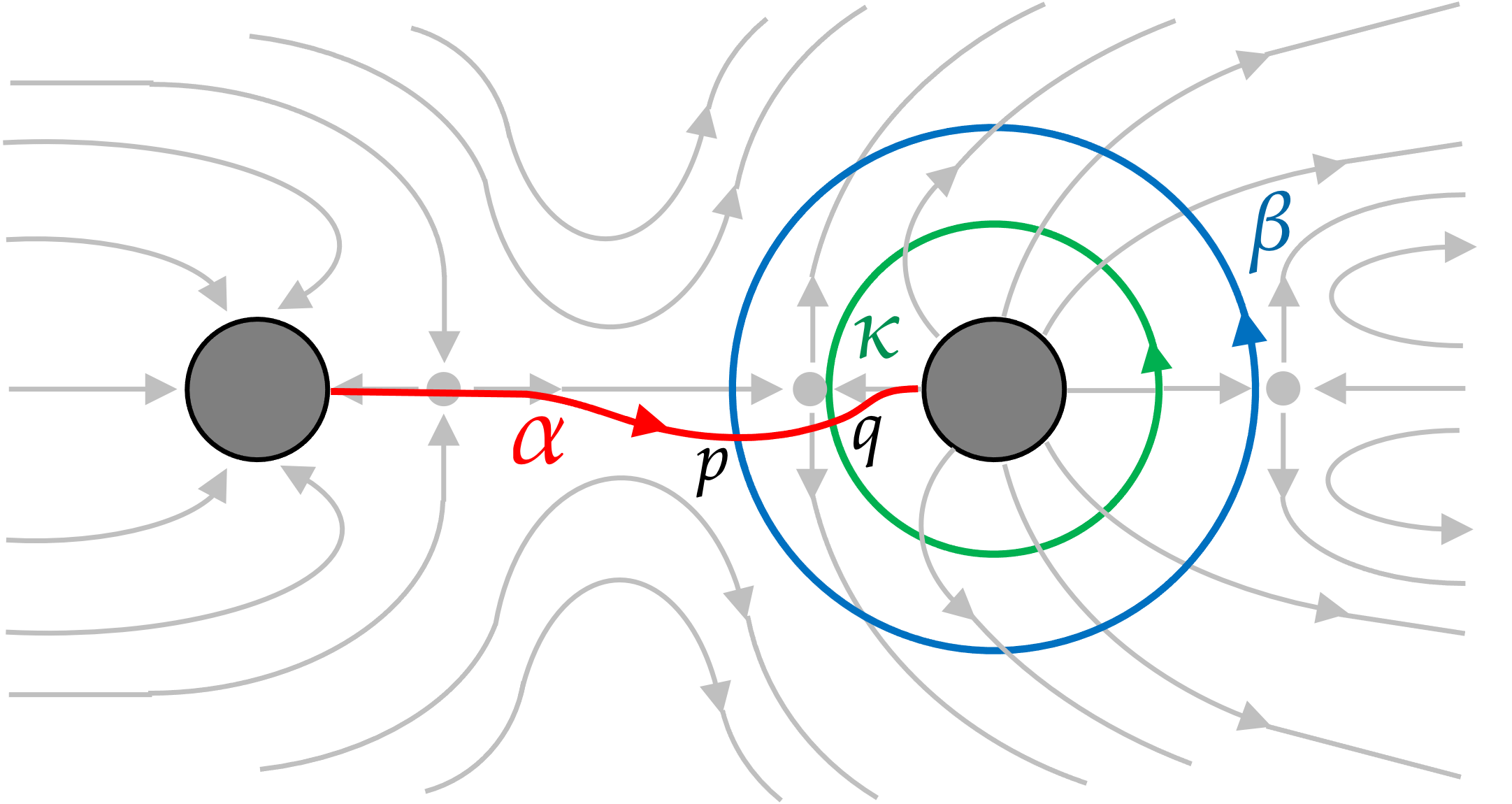}
    \caption{A combing $v_0$ on $\Sigma_{1,1}$, where there is a singularity of index $-1$ on each curve.}
    \label{fig:tri_inv_56}
\end{figure}

In Figure \ref{fig:tri_inv_56}, we think of the plane as $\mathbb{R}^2 \sqcup \{\infty\}$. The surface $\Sigma_{1,1}$ is obtained by removing a neighborhood of $\infty$ and the two grey disks and then identifying the boundary circles of the two disks. The grey lines represent the flow lines of a combing $v_0$ on $\Sigma_{1,1}$, which has one singularity of index $-1$ on each of the three curves. It follows from  Figure \ref{fig:tri_inv_56} that $\theta(\alpha) = \frac{1}{2}$, $\theta(\beta) = \frac{3}{2}$, $\theta(\kappa) = \frac{1}{2}$. Denote by $p$ the intersection of the $\alpha$-curve with the $\beta$-curve, and by $q$ the intersection of the $\alpha$-curve with the $\kappa$-curve. Then
\[\theta(\alpha;p) = 0\,, \qquad \theta(\beta;p) = \frac{3}{4}\,,
    \qquad \theta(\alpha;q) = \frac{1}{4}\,, \quad\text{and}\quad \theta(\kappa;q) = 0\,.\]
Hence, $\langle T^{st}_2 \rangle_{\mathcal{H}}$ equals the tensor diagram below, where the $\beta$ (resp.~$\kappa$) symbol by the arrow indicates that the leg represents the Hopf algebra $H_{\beta}$ (resp.~$H_{\kappa}$):
\begin{align}\label{eqn:stabilization_tensor_2}
    \includegraphics[scale=.4, valign = c]{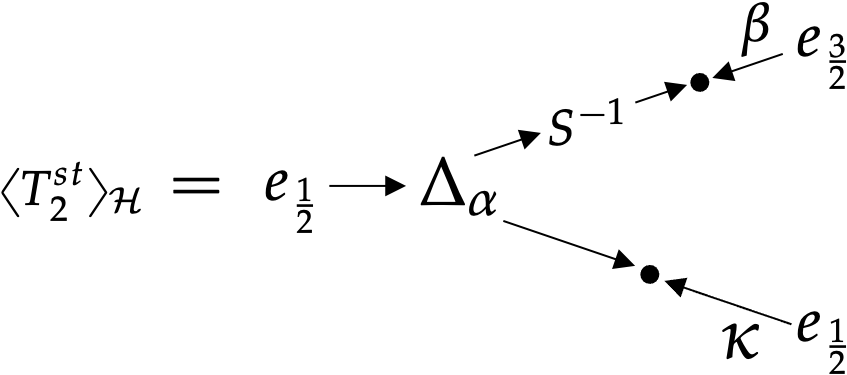}
\end{align}
The bracket for the other two diagrams are obtained from $\langle T^{st}_2 \rangle_{\mathcal{H}}$ by cyclically permuting the labels in the tensor diagram:
\begin{align}\label{eqn:stabilization_tensor_1_3}
    \includegraphics[scale=.4, valign = c]{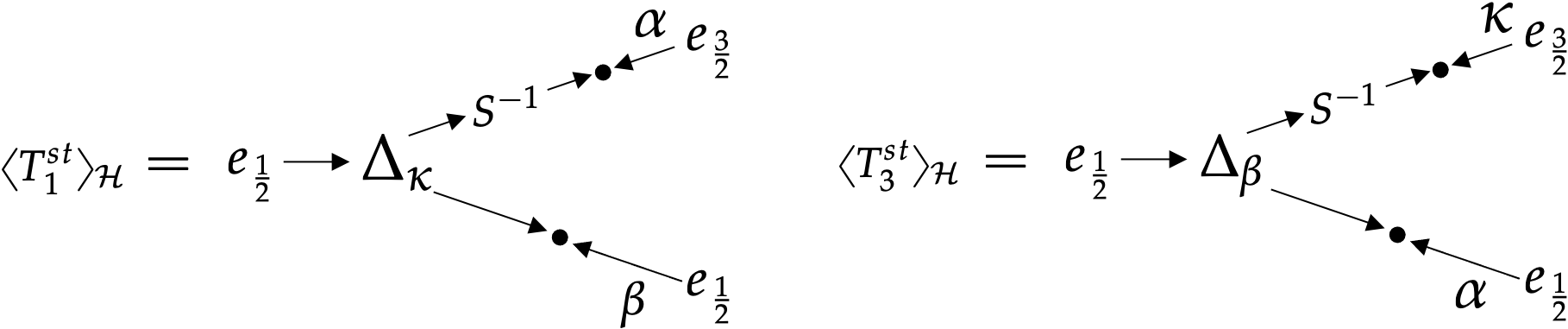}
\end{align}

\subsection{Proofs of Propositions~\ref{prop:bracket_orientation},~\ref{prop:inv_combing_move}, and~\ref{prop:inv_trisection_move}}\label{subsec:proof1} In this section, we give the deferred proofs of Propositions \ref{prop:bracket_orientation}, \ref{prop:inv_combing_move} and \ref{prop:inv_trisection_move} from earlier in this section.

\begin{proof} (Proposition \ref{prop:bracket_orientation})
    Without loss of generality, assume the orientation of an $\alpha_i$ curve is reversed as shown below:
\begin{align*}
    \includegraphics[scale=.4, valign = c]{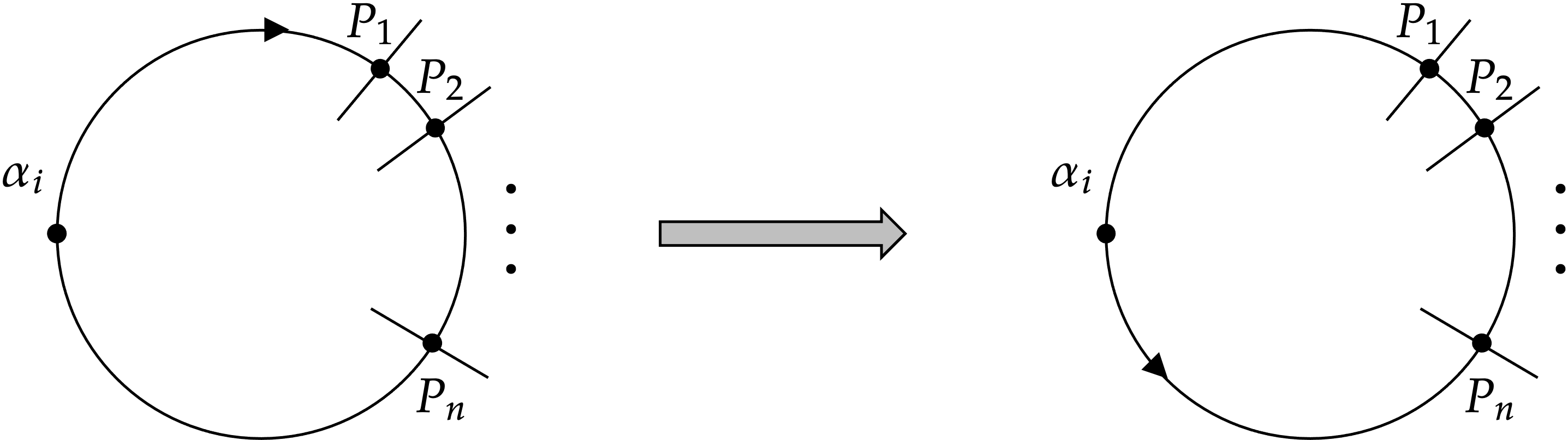}
\end{align*}
For ease of notation, we denote $\theta = \theta(\alpha_i)$ and $\theta(p) = \theta_{\alpha_i}(p)$ where the $\alpha_i$ curve has its initial orientation.  When the orientation of the $\alpha_i$ curve is reversed, we use similar notation with $\alpha_i$ replaced by $\alpha_i'$.  We directly see that
\begin{align*}
    \theta' = - \theta,
\end{align*}
and for each intersection $p$ on $\alpha_i$,
\begin{align*}
    \theta'(p) = -2 \theta + \theta(p)
\end{align*}
The relevant tensor diagram associated with $\alpha_i$ before orientation reversal is
\begin{align}\label{eqn:ori_reversal_before}
    \includegraphics[scale=.4, valign = c]{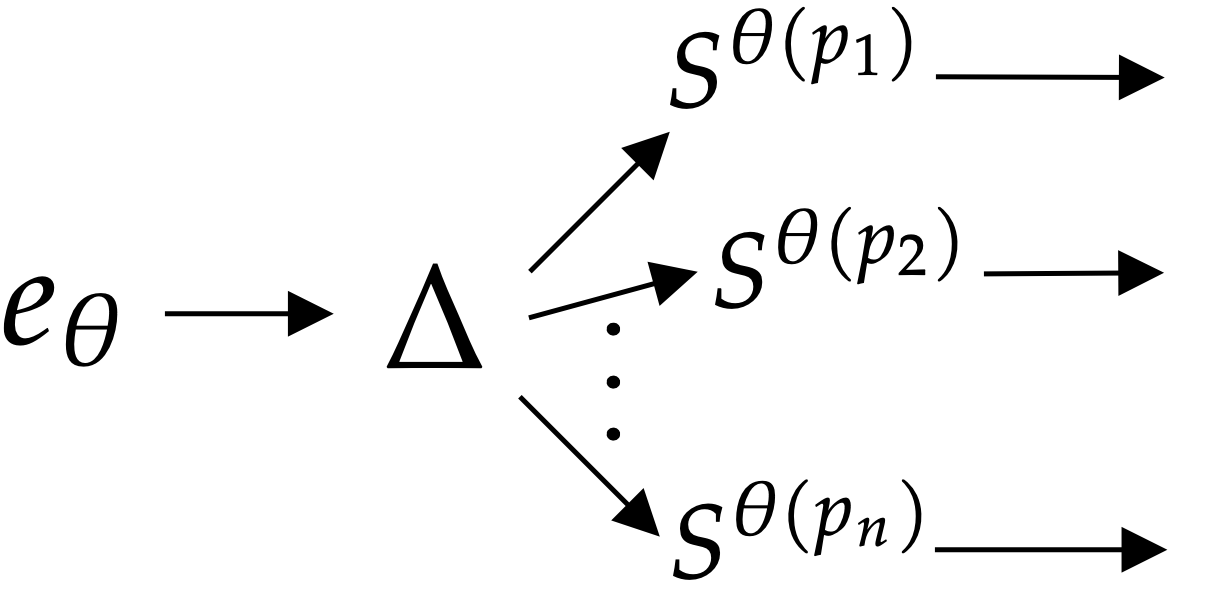}
\end{align}
Then the relevant tensor diagram associated with $\alpha_i$ after orientation reversal and some manipulations are as follows:
\begin{align*}
    \includegraphics[scale=.4, valign = c]{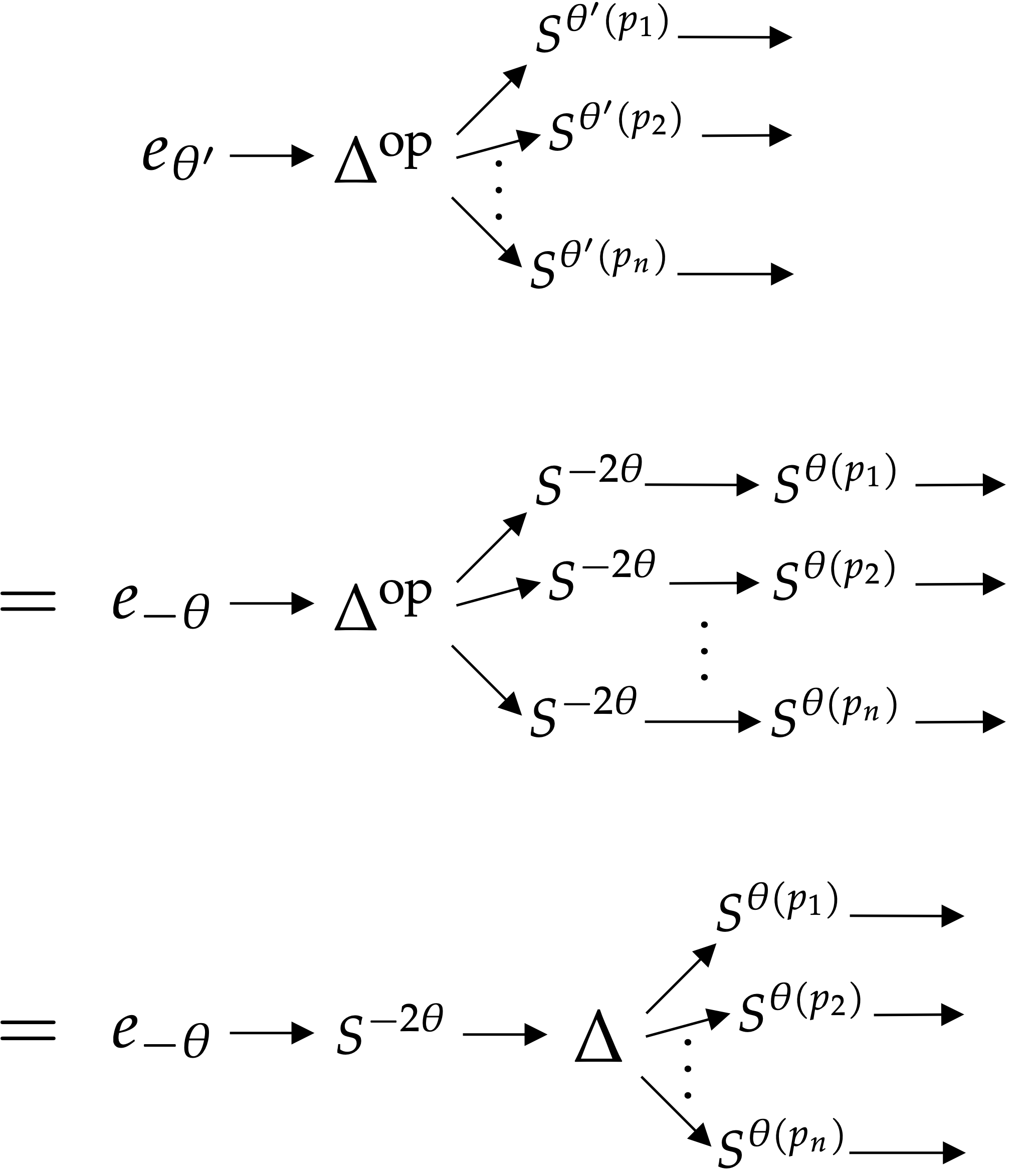}
\end{align*}
The second equality above is due to the fact that $-2\theta$ is an odd integer and $S^{-2\theta}$ is hence an anti-coalgebra morphism. That the last tensor diagram above equals Equation~\eqref{eqn:ori_reversal_before} follows from the identity (see Lemma \ref{lem:S_fli_e_theta}),
\begin{align*}
    \includegraphics[scale=.4, valign = c]{Figures/tri_inv_24.png}
\end{align*}
This concludes the proof. \end{proof}

\begin{proof} (Proposition \ref{prop:inv_combing_move})
On account of Proposition~\ref{prop:bracket_orientation} we can choose the orientation of each curve as we please, and then examine basepoint isotopy and basepoint spiral moves accordingly.

\vspace{3pt}

 \textbf{Basepoint isotopy.} Without loss of generality, consider the basepoint on an $\alpha_i$ curve which we will slide through the intersection $x$ between the $\alpha_i$ curve and a $\beta_j$ curve.  We further assume $\alpha_i$ and $\beta_j$ intersect positively. This situation is depicted below:
 \begin{align*}
    \includegraphics[scale=.4, valign = c]{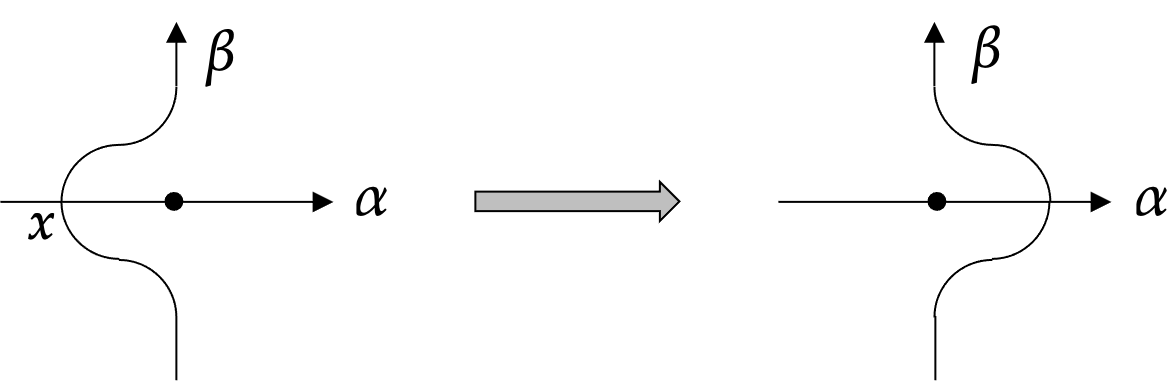}
\end{align*}
For ease of notation, we will call $\alpha_i$ the \emph{$\alpha$ curve} and similarly call $\beta_j$ the \emph{$\beta$ curve}. Now let us visualize the base point with the combing present, both before and after the basepoint isotopy:
\begin{align*}
    \includegraphics[scale=.4, valign = c]{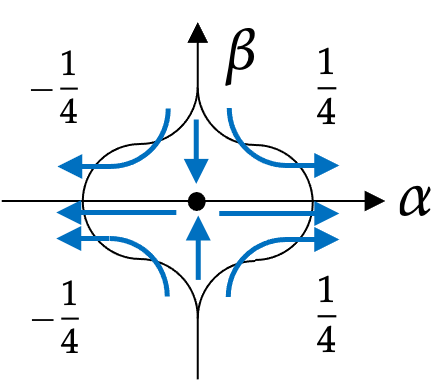}
\end{align*}
The numbers indicate the degrees of rotation, in units of $2\pi$, of the tangent field of the $\beta$ curve relative to the combing restricted to a specific segment. For instance,  as one travels along the old $\beta$ curve (i.e.~the left route), the relative degree of rotation from the ``bottom'' point to the point of intersection with the $\alpha$ curve is $-\frac{1}{4}$.

\vspace{3pt}

Note that before the move $x$ is the last intersection on the $\alpha$ curve with $\theta(\alpha;x) = \theta(\alpha)$, whereas after the move $x$ becomes the first intersection with $\theta(\alpha;x) = 0$.  The following shows the relevant tensor diagram before the move:
\begin{align}\label{eqn:base_point_isotopy_before}
    \includegraphics[scale=.4, valign = c]{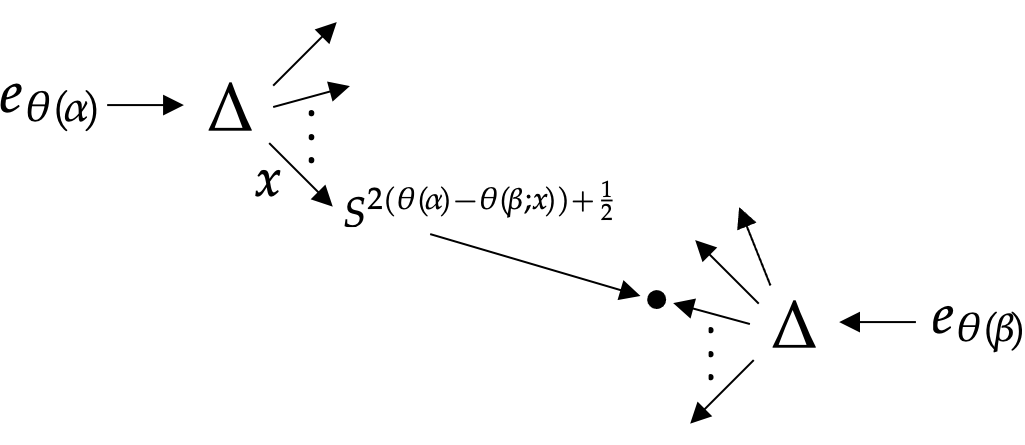}
\end{align}
After the move, $\theta(\beta;y)$ increases by $1$ for all intersections $y$ on $\beta$ occurring after $x$.  Moreover $\theta(\beta;y)$ increases by $\frac{1}{2}$ for $y = x$, and $\theta(\beta;y)$ remains unchanged for intersections $y$ on $\beta$ occuring before $x$. Furthermore $\theta(\beta)$ increases by $1$. Accordingly the tensor diagram after the move is given by
\begingroup
\allowdisplaybreaks
\begin{align*}
    &\includegraphics[scale=.4, valign = c]{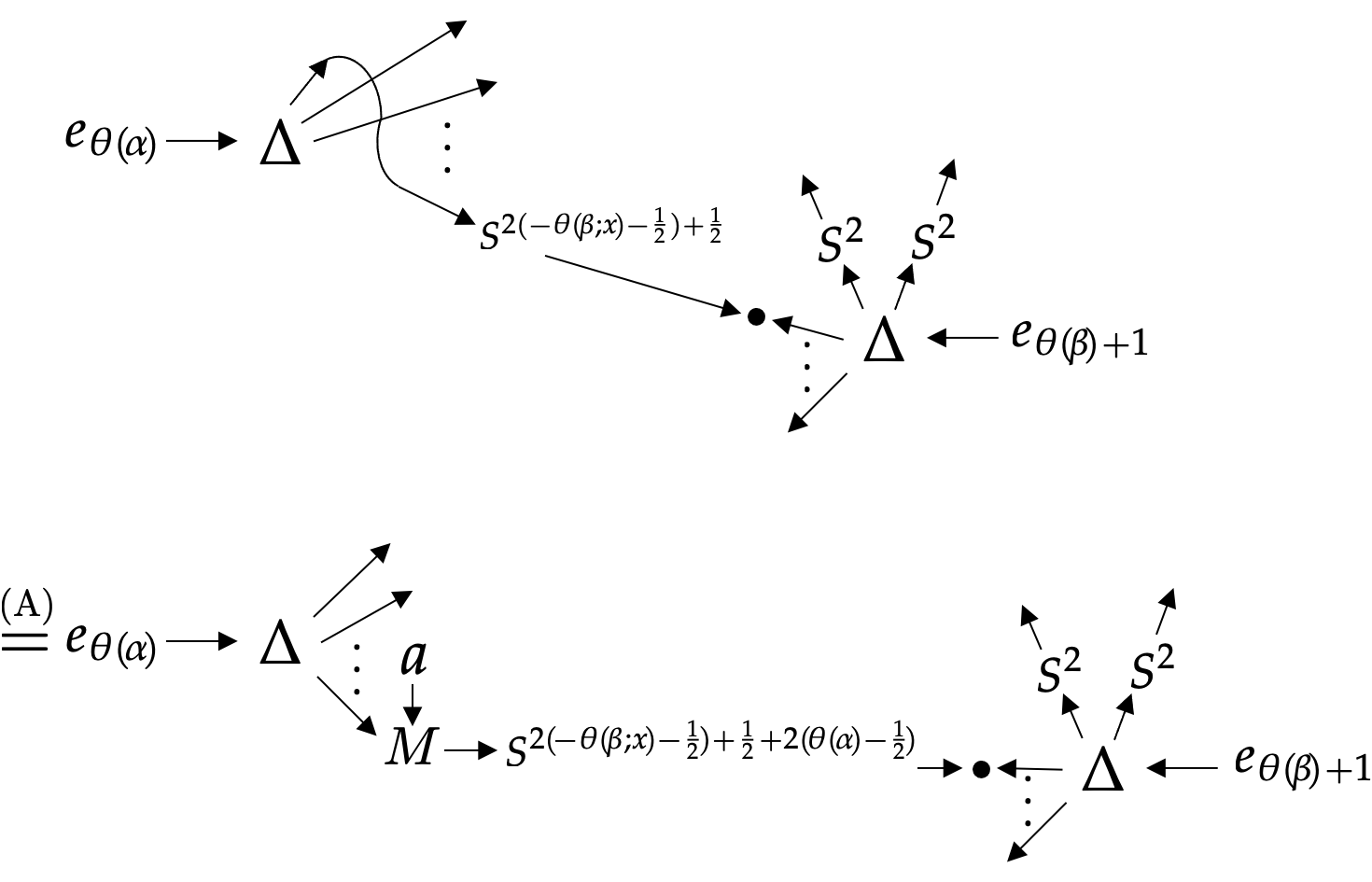}\\
    &\includegraphics[scale=.4, valign = c]{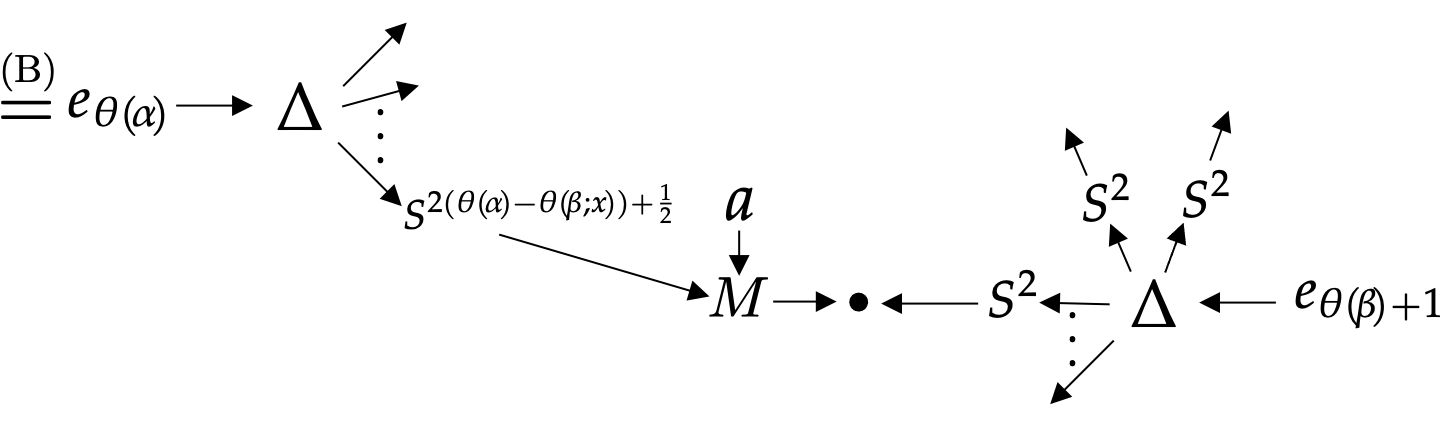}\\
    &\includegraphics[scale=.4, valign = c]{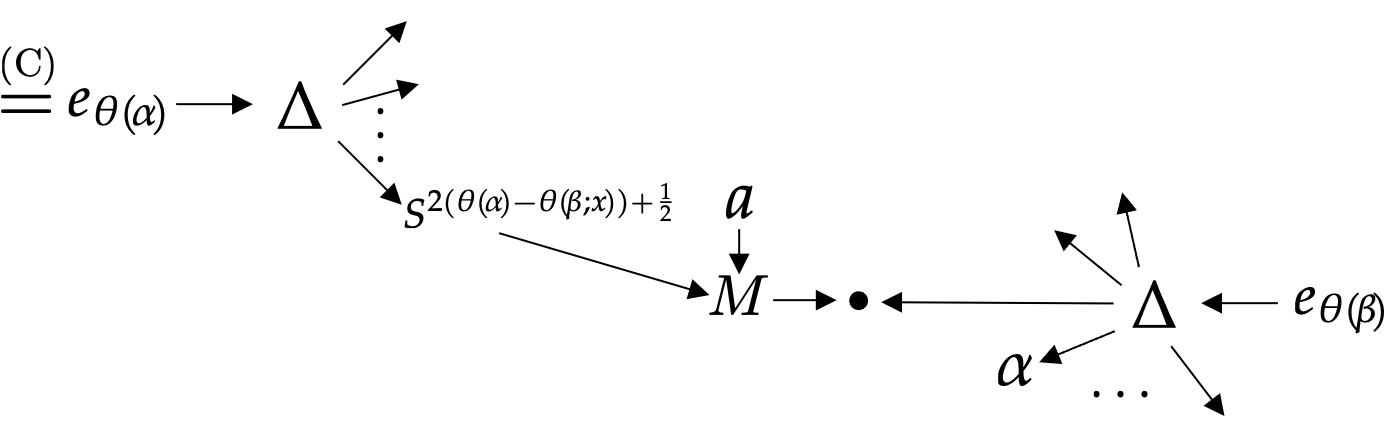}\\
    &\includegraphics[scale=.4, valign = c]{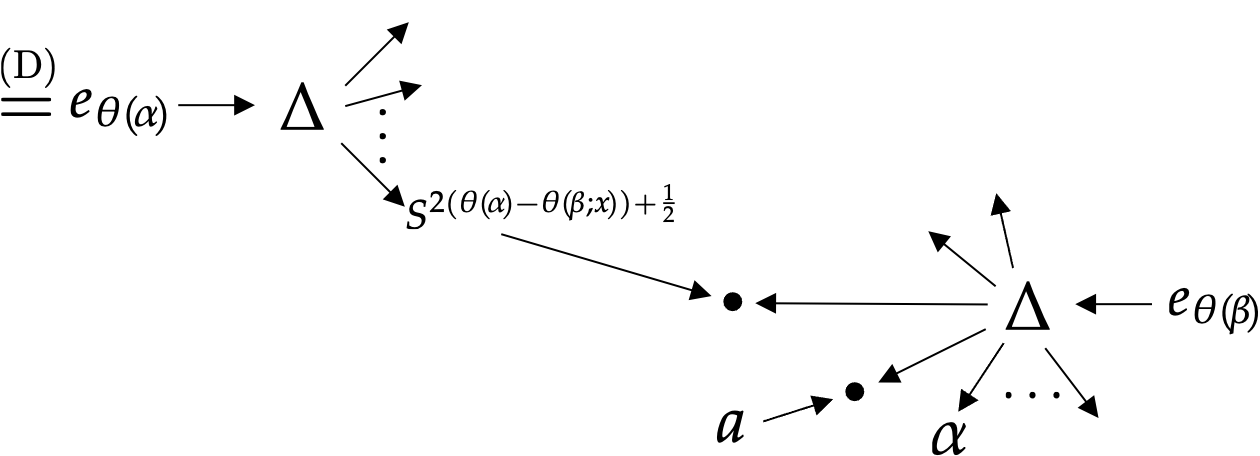}
\end{align*}
\endgroup
In the above diagrams, identity (A) comes from Lemma~\ref{lem:e_theta_cocommutativity}. Identity (B) utilizes the fact that $2(\theta(\alpha) - \theta(\beta;x)) + \frac{1}{2}$ is an even integer (since the $\alpha$ and $\beta$ curves intersect positively) and that the antipode on one side of the pairing is changed to the inverse antipode when it is slid to the other side (since the pairing on $H_{\alpha} \otimes H_{\beta}$ defines a Hopf algebra morphism $H_{\alpha} \to H_{\beta}^{*, \text{cop}}$).  Identity (C) is due to the definition of the generalized integral and the assumption that the Hopf algebras involved are balanced:
\begin{align*}
    \includegraphics[scale=.4, valign = c]{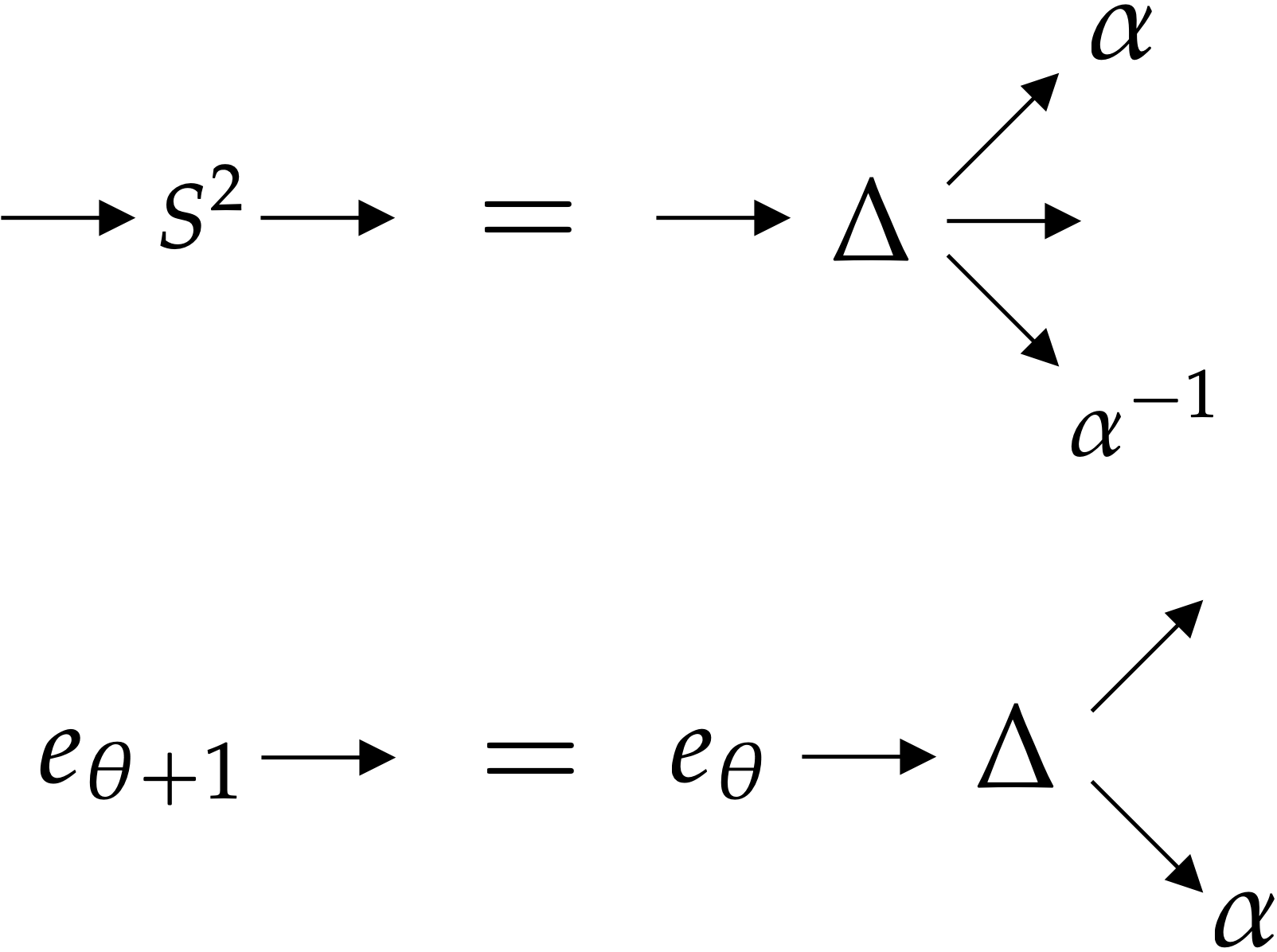}
\end{align*}
It should be clarified that $\alpha$ carries multiple meanings in identity (C) and the identities thereafter.  Specifically, $\alpha$ both denotes the $\alpha$-type curve (as in $\theta(\alpha)$) as well as the distinguished group-like element $\alpha$ in the dual Hopf algebra.
Identity (D) again utilizes the fact that the pairing on $H_{\alpha} \otimes H_{\beta}$ defines a Hopf algebra morphism $H_{\alpha} \to H_{\beta}^{*, \text{cop}}$. Lastly, since the pairing preserves the distinguished group-like element, we have
\begin{align*}
    \includegraphics[scale=.4, valign = c]{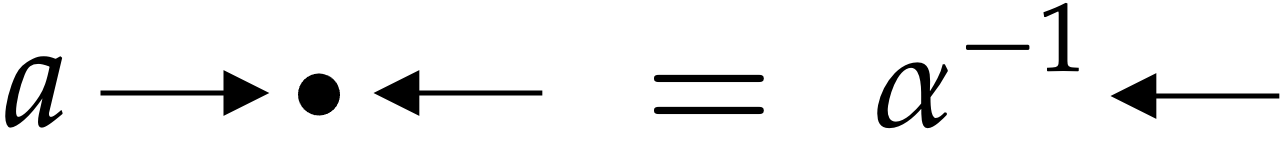}
\end{align*}
and so the right side of Identity (D) equals the tensor diagram in Equation~\eqref{eqn:base_point_isotopy_before}.

\vspace{3pt}

 \textbf{Basepoint spiral.} Without loss of generality, assume a basepoint spiral is performed around the base point of an $\alpha$ curve:
 \begin{align*}
    \includegraphics[scale=.4, valign = c]{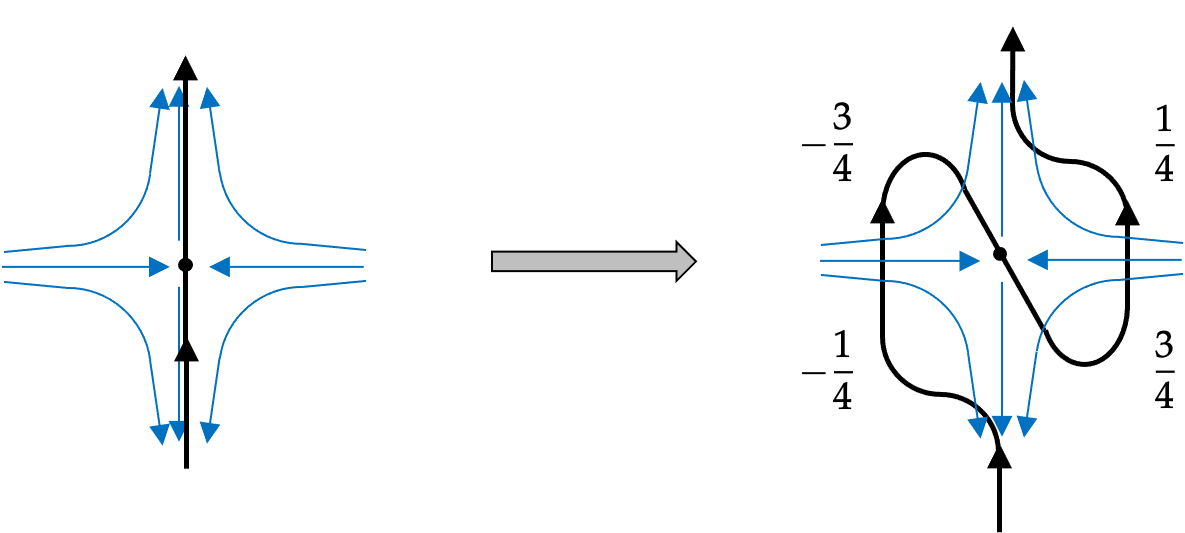}
\end{align*}
In the above figure, the black curve denotes the $\alpha$ curve, and the numbers in the right part of the figure again represent the degree of rotation of the $\alpha$ curve relative to the combing in each depicted segment. From the figure, it follows that $\theta(\alpha)$ remains unchanged, while for each intersection $x$ on the $\alpha$ curve, $\theta(\alpha; x)$ increases by 1. The relevant tensor diagram before the move is
\begin{align*}
    \includegraphics[scale=.4, valign = c]{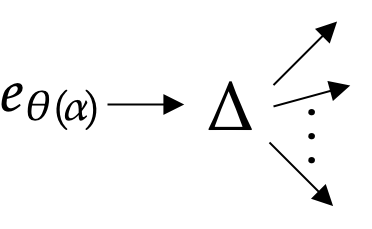}
\end{align*}
and the diagram after the move is
\begin{align*}
    \includegraphics[scale=.4, valign = c]{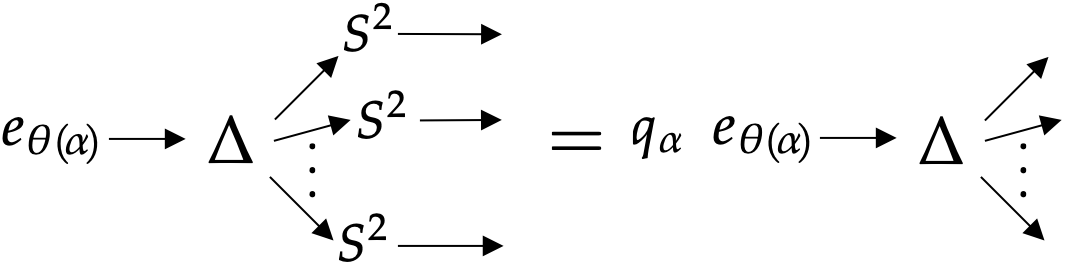}
\end{align*}
where the equality above is due to the facts that $S^2$ is a Hopf algebra morphism and that $e_{\theta(\alpha)}$ is an eigenvector of $S^2$ with eigenvalue $q_{\alpha}$. Since $q_{\mathcal{H}} = q_{\alpha}$, the two diagrams before and after the move evaluate to the same element in $k/\sim$.
\end{proof}

\begin{proof} (Proposition \ref{prop:inv_trisection_move})
We check each of the moves. Again, due to Proposition \ref{prop:bracket_orientation}, we can choose a specific orientation for each curve in the configurations discussed below.

\vspace{3pt}

\textbf{Handle-slides.} Consider the case of sliding an $\alpha_j$ curve over an $\alpha_i$ curve,
\begin{align*}
    \includegraphics[scale=.4, valign = c]{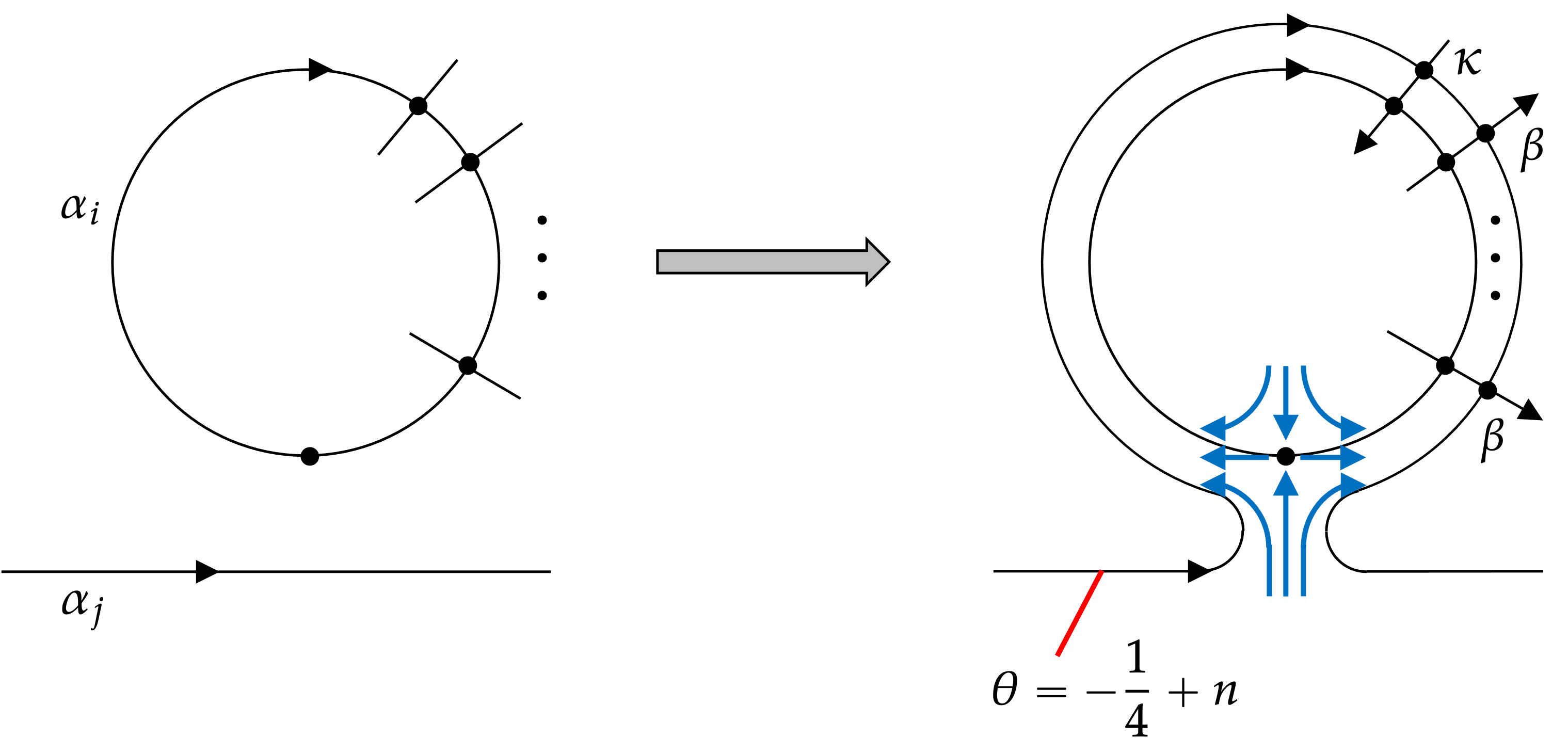}
\end{align*}
The $\alpha_i$ curve may have intersections with $\beta$-type and $\kappa$-type curves. We simply label each of those intersections with the $\alpha_i$ curve with either $\beta$ or $\kappa$ accordingly.  Denote by $\alpha_j'$ the new $\alpha_j$ after the move. Let $p$ be a point located  on $\alpha_j'$ right before $\alpha_j'$ makes a `U' turn to go around $\alpha_i$. ($p$ is the intersection of the red segment with $\alpha_j'$ in the above diagram.) Then 
$$\theta(\alpha_j';p) = -\frac{1}{4} + n$$
for some integer $n$. The new intersections $\alpha_j'$ obtained are in one-to-one correspondence with intersections on $\alpha_i$, and for each such intersection $q$ we have
$$\theta(\alpha_j';q) = n + \theta(\alpha_i; q),$$
where we have used the same symbol $q$ to denote an intersection on $\alpha_i$ and the corresponding intersection on $\alpha_j'$. Set $\theta_i := \theta(\alpha_i)$ and $\theta_j := \theta(\alpha_j)$. Then we have
$$\theta(\alpha_j') = \theta_i + \theta_j + \frac{1}{2}.$$

The tensor diagram after the move is given by
\begin{align*}
    \includegraphics[scale=.4, valign = c]{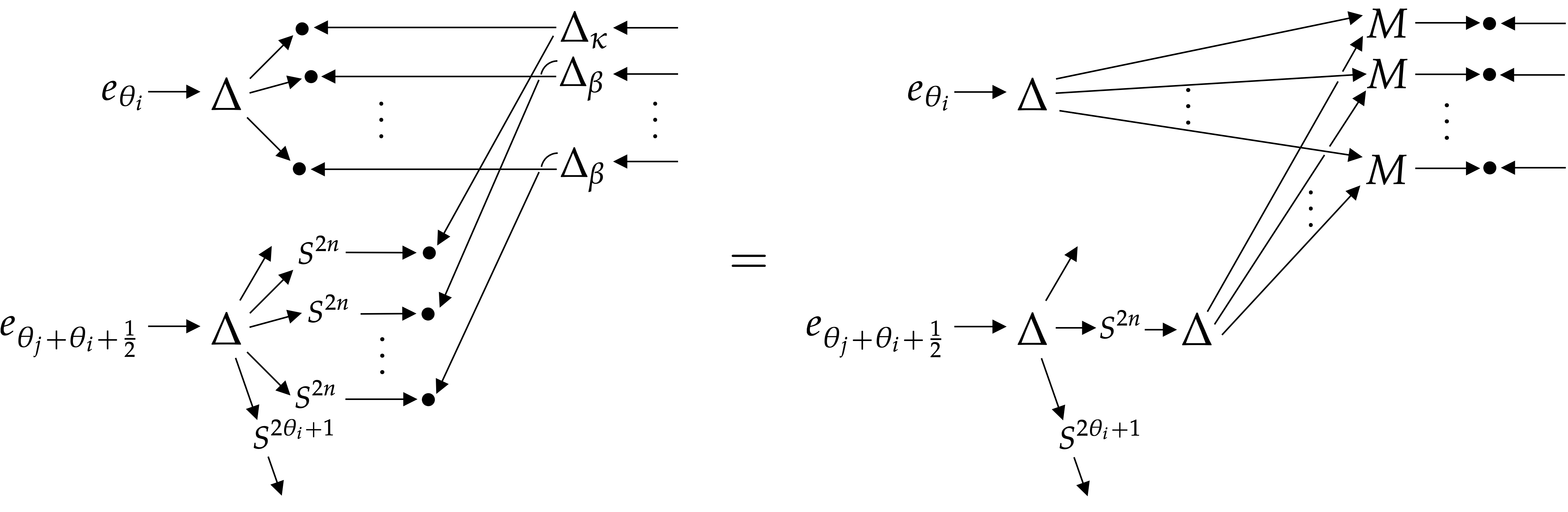}
\end{align*}
\begin{align*}
    \includegraphics[scale=.4, valign = c]{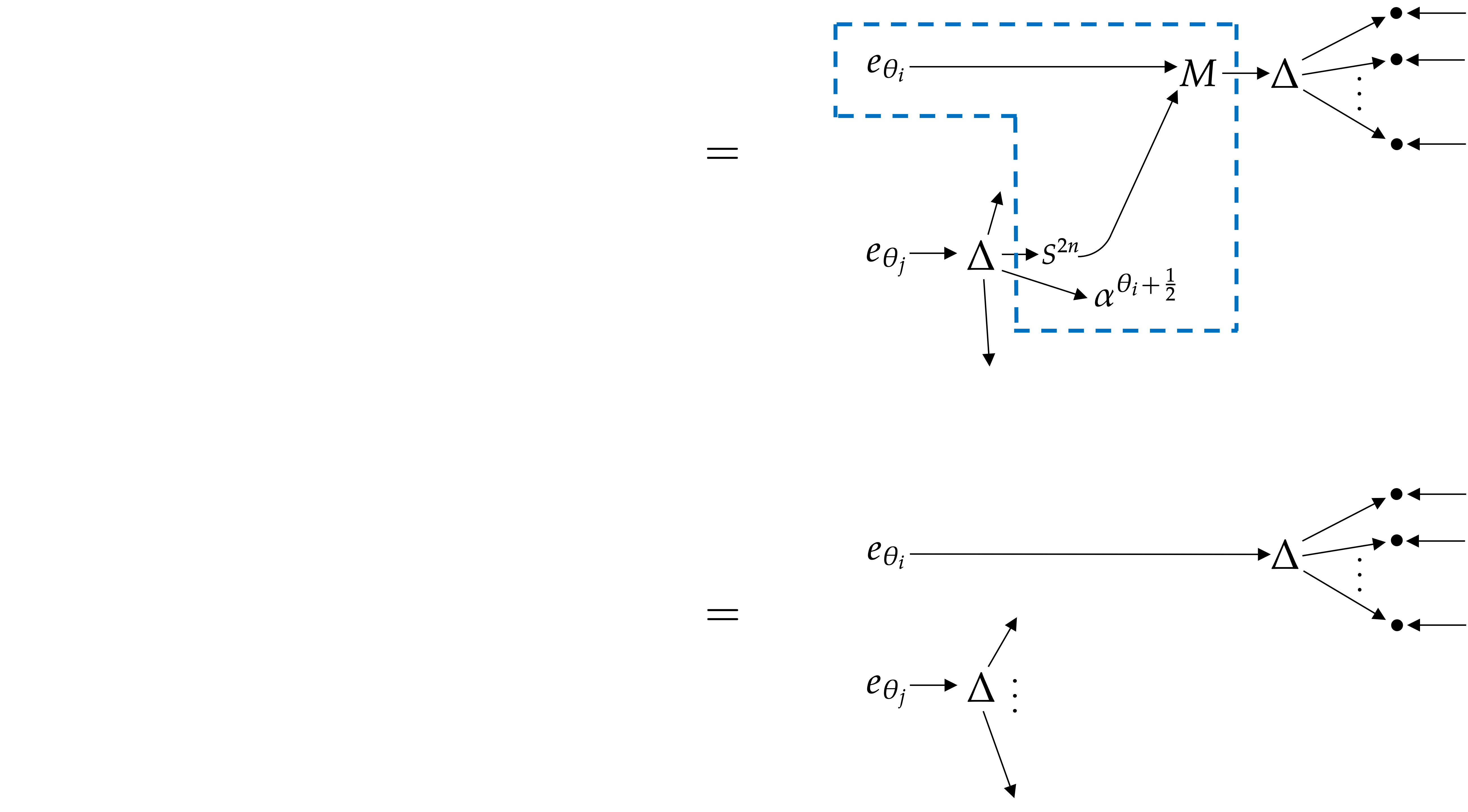}
\end{align*}
In the first equality above, we used the fact that the pairs induce Hopf algebra morphisms $H_{\alpha} \to H_{\beta}^{*, \text{cop}}$ and $H_{\alpha} \to H_{\kappa}^{*, \text{op}}$. The second equality relies on the compatibility condition between multiplication and comultiplication in a Hopf algebra. The third equality follows directly from the following identity, with the portion contained in the dashed box above corresponding to the left hand-side below (see Lemma~\ref{lemma:etheta1}),
\begin{align*}
    \includegraphics[scale=.35, valign = c]{Figures/tri_inv_38.png}
\end{align*}
The expression in the third equality is exactly the corresponding tensor diagram before the move.

\vspace{3pt}

\textbf{Two-point move.} Without loss of generality, consider the following two-point move between an $\alpha_i$ and a $\beta_j$ curve:
\begin{align*}
    \includegraphics[scale=.4, valign = c]{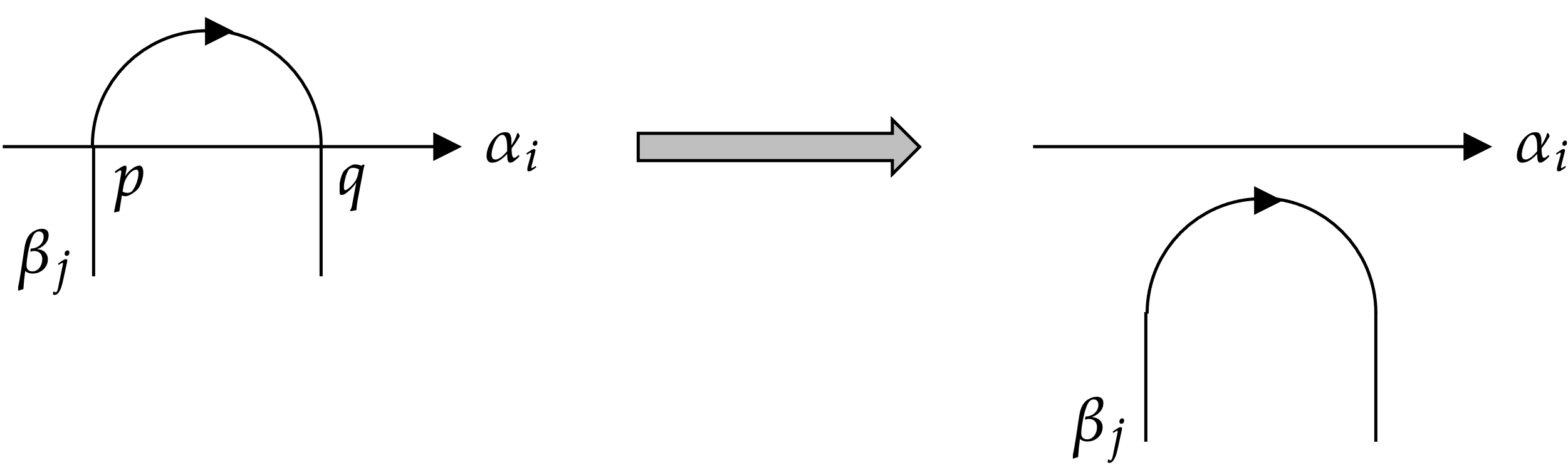}
\end{align*}
For the curves before the move, we have $\theta(\alpha_i;p) = \theta(\alpha_i;q)$ and $\theta(\beta_j;p) = \theta(\beta_j;q) + \frac{1}{2}$. Hence $\theta_{\alpha_i}(q) = \theta_{\alpha_i}(p) + 1$. Set $\theta(p) = \theta_{\alpha_i}(p)$, which is an even integer since $\alpha_i$ and $\beta_j$ intersect positively. Then the relevant tensor diagram before the move is
\begin{align*}
    \includegraphics[scale=.4, valign = c]{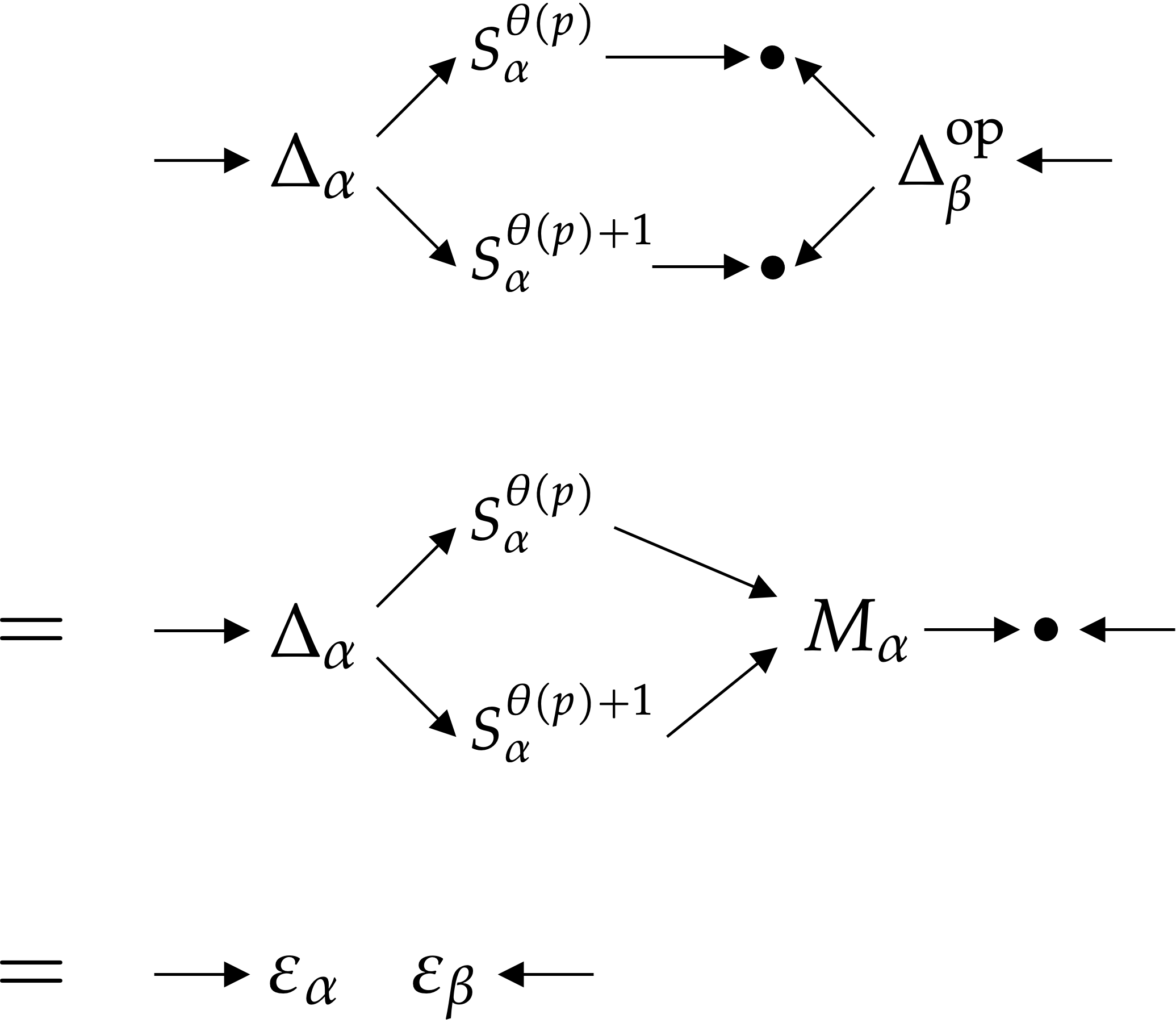}
\end{align*}
The first equality above holds due to the condition that the pairing induces a Hopf algebra morphism $H_{\alpha} \to H_{\beta}^{*, \text{cop}}$, and the second equality follows from basic properties of the antipode. The expression following the second equality isolates the relevant part of the tensor diagram after the 2-point move.

\textbf{Three-point move I.} The three-point move involves curves of each of the three types. We label each curve  by its type. There are two cases to consider depending on the relative ordering of the involved curves. Specifically, the following shows a configuration where the $\alpha$-, $\beta$-, and $\kappa$-type curves are arranged \emph{counterclockwise} on the boundary of the triangle formed by the curves:

\begin{align*}
    \includegraphics[scale=.4, valign = c]{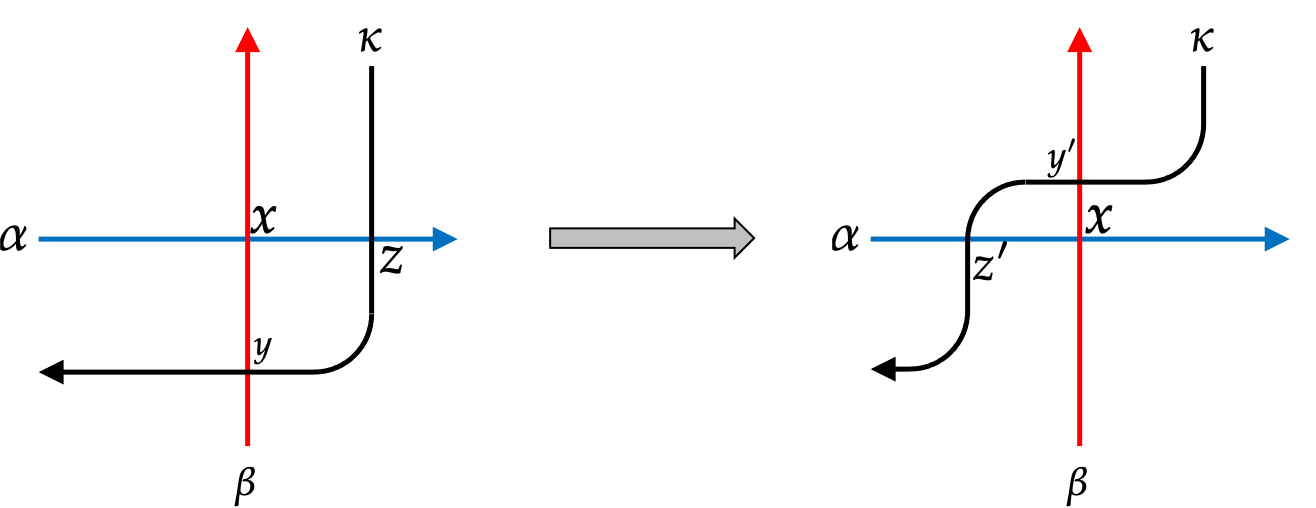}
\end{align*}
For the curves before the move, let
\begin{alignat*}{3}
    C_{\alpha} &:= \theta(\alpha;x) &&= \theta(\alpha;z)\,,\\
    C_{\beta}  &:= \theta(\beta;x) &&= \theta(\beta;y)\,,\\
    C_{\kappa} &:= \theta(\kappa;z) &&= \theta(\kappa;y) + \frac{1}{4}\,.
\end{alignat*}
Then after the move, we have
\begin{alignat*}{3}
    C_{\alpha} &= \theta(\alpha;x) &&= \theta(\alpha;z')\,,\\
    C_{\beta}  &= \theta(\beta;x) &&= \theta(\beta;y')\,,\\
    C_{\kappa} &= \theta(\kappa;z') &&= \theta(\kappa;y') + \frac{1}{4}\,.
\end{alignat*}
We also mark these numbers on the curves below for convenience:
\begin{align*}
    \includegraphics[scale=.4, valign = c]{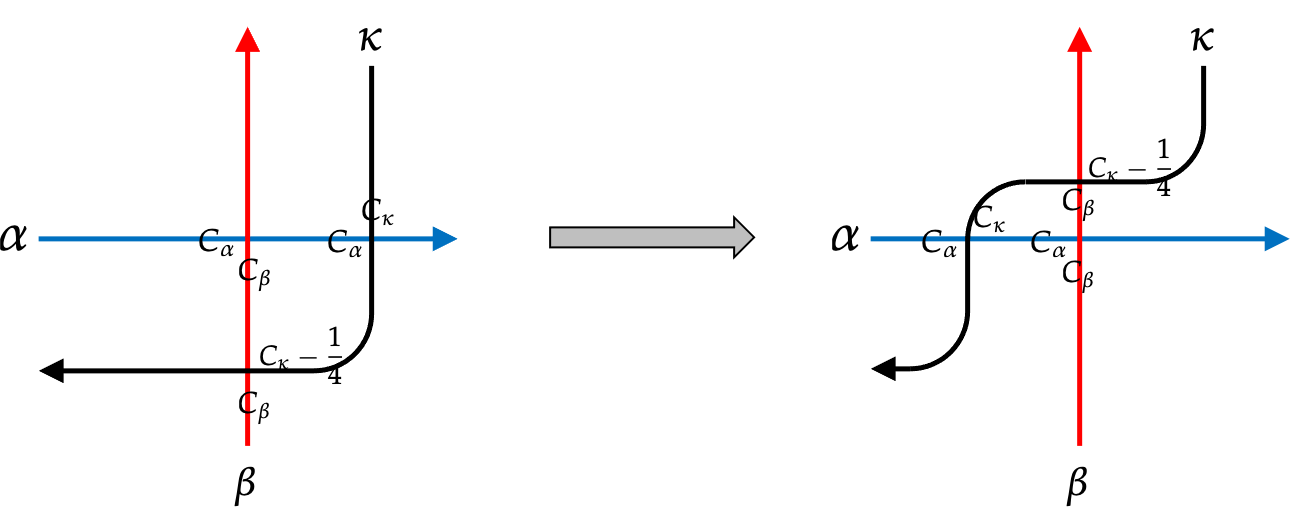}
\end{align*}

\noindent Invariance under the move requires the equality of the tensor diagrams
\begin{align*}
    \includegraphics[scale=.4, valign = c]{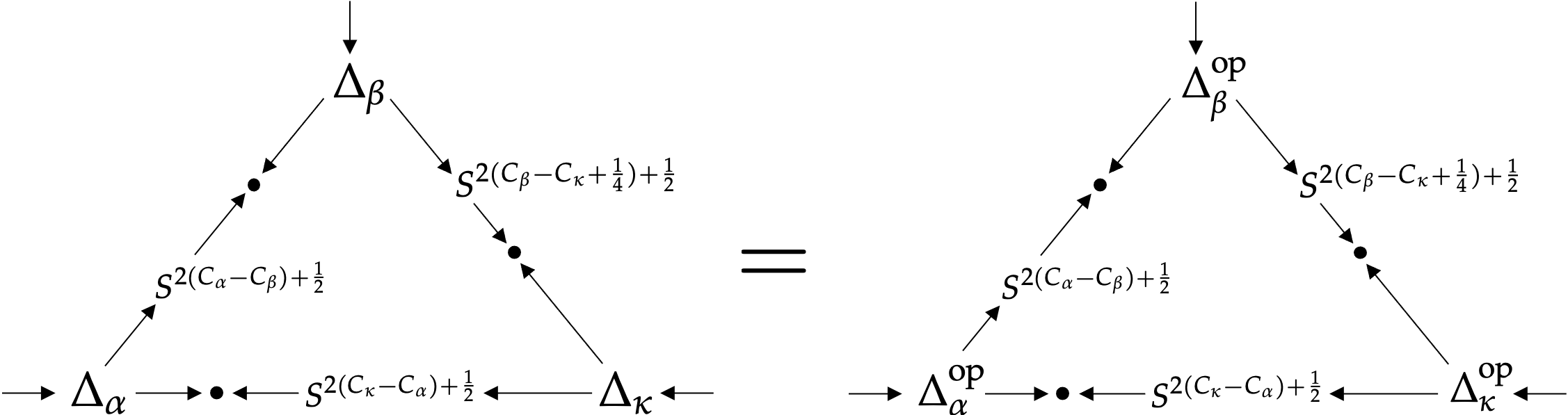}
\end{align*}
Note that all the exponents of the antipode above are even integers since the curves are oriented so that their intersections are positive. The above equality is equivalent to the following identity  which is a part of the definition of Hopf triplets:
\begin{align*}
    \includegraphics[scale=.4, valign = c]{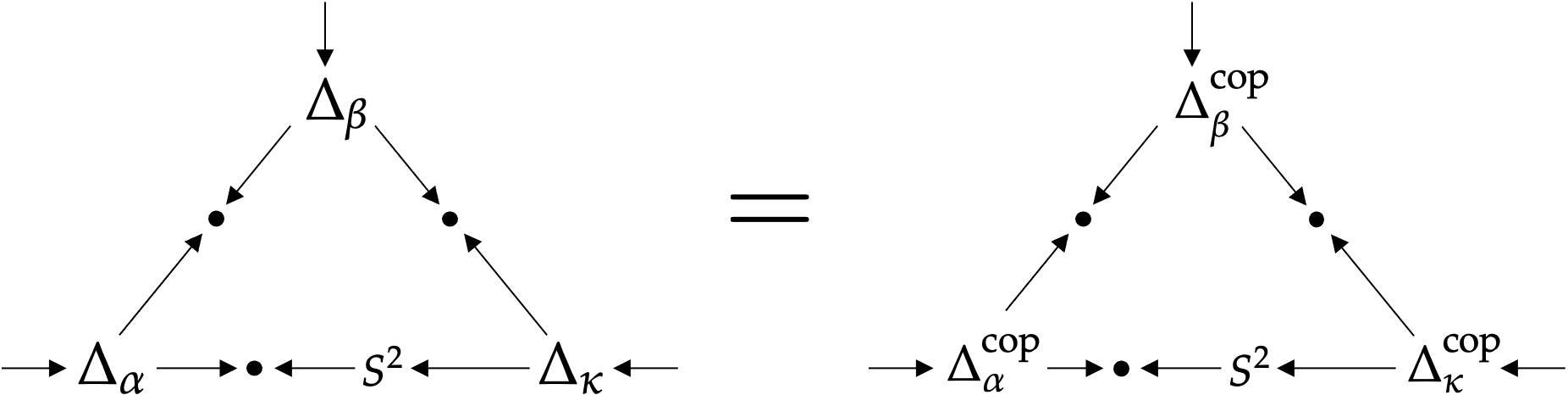}
\end{align*}

\textbf{Three-point move II.} Now we consider the other case where the $\alpha$-, $\beta$-, and $\kappa$-type curves are arranged \emph{clockwise} on the boundary of the triangle. This time the curves are oriented so that all intersections are negative:
\begin{align*}
    \includegraphics[scale=.4, valign = c]{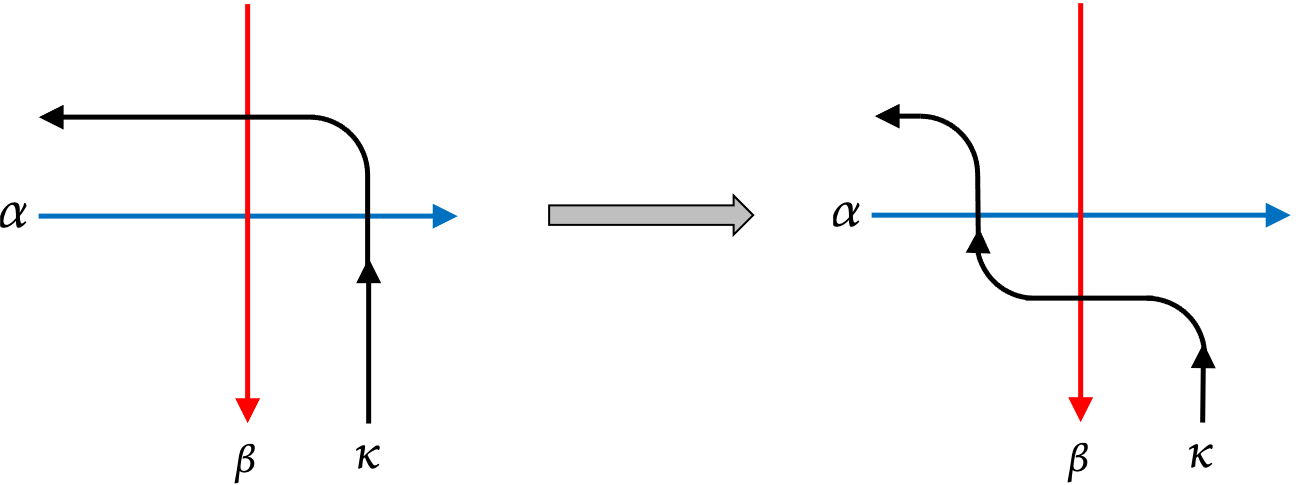}
\end{align*}
As in the first case, we mark the rotation number of each crossing on the diagram.
\begin{align*}
    \includegraphics[scale=.4, valign = c]{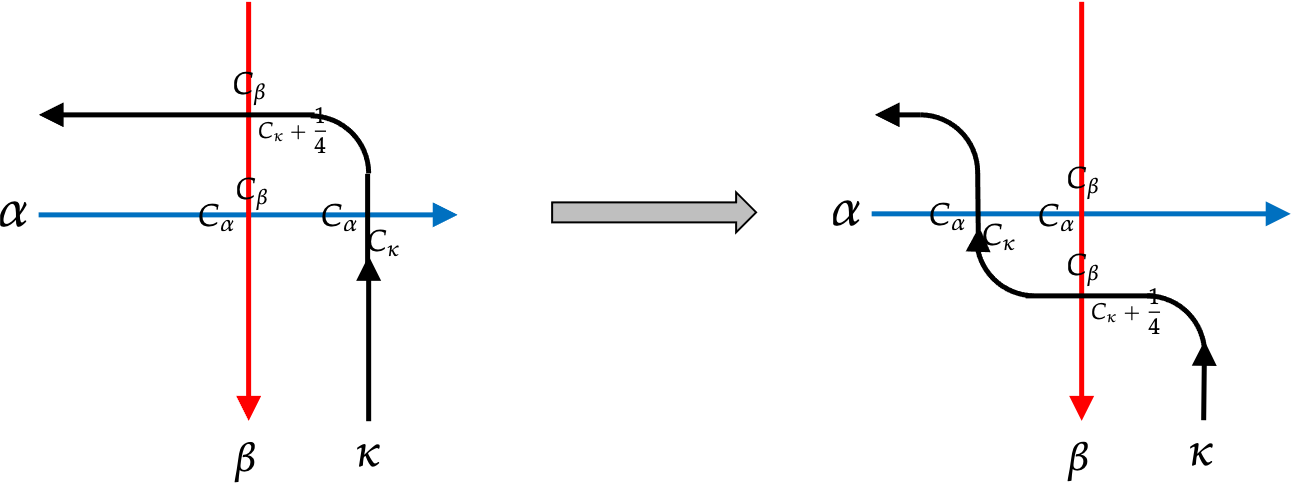}
\end{align*}

\noindent Then, invariance under the move requires the equality of the tensor diagrams
\begin{align*}
    \includegraphics[scale=.4, valign = c]{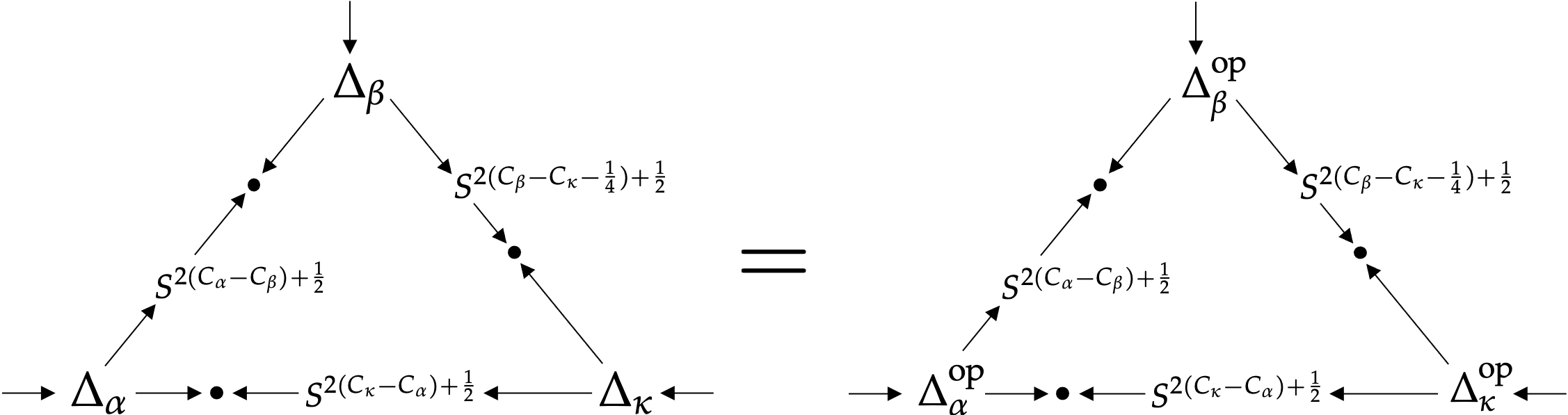}
\end{align*}
Note that all the exponents of the antipode above are odd integers. The equality of the above tensor diagrams is equivalent to the following identity which is a consequence of the definition of Hopf triplets (see Theorem \ref{thm:triangle_identity_versions}):
\begin{align*}
    \includegraphics[scale=.4, valign = c]{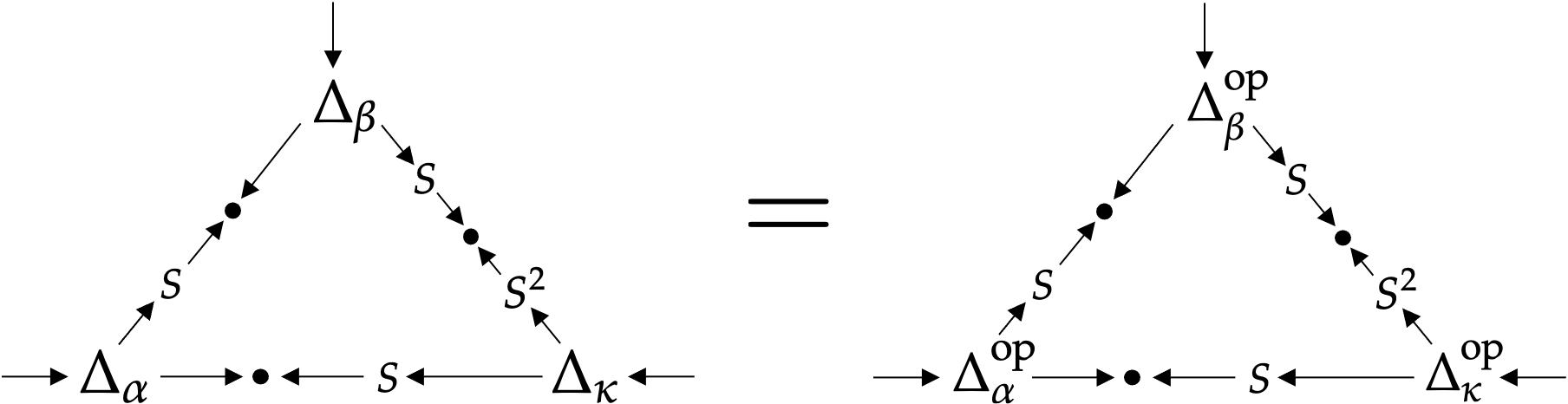}
\end{align*}
\end{proof}

\section{Trisection Invariants Of Small Exotica} \label{sec:stein_nuclei} A natural question, arising in the study of trisections, is: what is the simplest pair of trisections describing an exotic pair? This question is particularly relevant in this paper, since computing the trisection invariant is very computationally taxing. Thus, simple exotic pairs of trisection diagrams provide the most promising candidates for exotica that new invariants could feasibly be used to distinguish.

\vspace{3pt}

In this section, we discuss two families of 4-manifolds that admit exotic pairs and small trisections. First, we discuss \emph{Stein nuclei}, a family of Stein 4-manifolds introduced by Yasui in \cite{yasui2014partial}. Our trisections of the Stein corks is new and provides new examples of small exotica in the context of trisections. Second, we discuss some examples of Stein 4-manifolds with small trisection genus, introduced by Takahashi \cite{takahashi2023exotic,takahashi2022minimal}.  

\subsection{Stein Nuclei} \label{subsec:stein_nuclei} In this section, we trisect Yasui's Stein nuclei \cite{yasui2014partial} and discuss the properties of our invariants for these spaces. These manifolds admit relatively simple Lefschetz fibrations, and thus simple trisection diagrams via the algorithm in Section \ref{subsec:trisection_diagrams_of_examples}. We begin by recalling the definition of the Stein nuclei using Kirby diagrams.

\begin{definition} \cite{yasui2014partial} Let $m = (m_0,m_1,m_2)$ be a tuple of non-negative integers and let $i$ be an integer. The \emph{Stein nucleus} $N^m_i$ is the $4$-manifold with boundary given by the Kirby diagram
\begin{center}
\includegraphics[width=.3\textwidth]{Figures/stein_nucleus_kirby_diagram.png}
\end{center}
\end{definition}

The following results describe the topological and smooth properties of the Stein nuclei. These properties make the Stein nuclei a particularly fruitful testing ground for new invariants.

\begin{thm} \cite[Prop.~7.1]{yasui2014partial} For any tuple $m = (m_0,m_1,m_2)$, the Stein nuclei $N^m_i$ for $i \in \Z$ have the following topological properties.
\begin{enumerate}[label=(\alph*)]
\item (Boundary) $\partial N_i^m$ is a homology $3$-sphere.
\item (Homology) $N^m_i$ is simply connected, $H_2(N^m_i)$ is isomorphic to $\Z \oplus \Z$ and the intersection form of $N^m_i$ is unimodular and indefinite.
\item (Homeomorphism type) An orientation preserving homeomorphism $\partial N^m_i \to \partial N^m_j$ extends to a homeomorphism $N^m_i \to N^m_j$ if and only if $m_1i = m_1j \mod 2$.
\end{enumerate}
\end{thm}

The Stein nuclei yield a infinite sequences of pairwise exotic 4-manifolds. Indeed, we have the following result, proven in \cite[Prop.~7.3]{yasui2014partial}. Also see \cite{yasui2019nonexistence} for a simplified proof.

\begin{thm} \label{thm:stein_cork_top} \cite{yasui2014partial} \label{thm:stein_cork_smooth} For any tuple $m = (m_0,m_1,m_2)$ with $m_1 \ge 1$, the Stein nuclei $N^m_i$ for $i \in \Z$ have the following smooth properties. Then there are infinite subsets
\[
S_0 \subset \{N^m_{2i} \; : \; i \in \Z\} \qquad\text{and}\qquad S_1 \subset \{N^m_{2i+1} \; : \; i \in \Z\} \qquad
\]
such that no two manifolds in $S_0$ are diffeomorphic, and similarly no two manifolds in $S_1$ are diffeomorphic.
\end{thm}

In particular, the Stein nuclei furnish an infinite family of different exotic pairs. 

\subsection{Lefschetz diagrams of Stein nuclei} Each Stein nucleus is also equipped with a natural Lefschetz fibration, and thus a natural open book on the boundary. We now describe the diagrams for these fibrations in some detail. Let $F$ be the following surface of genus $0$ with 3 boundary components equipped with four curves $\alpha_1,\alpha_2,\alpha_3$ and $\beta$ as depicted below:

\vspace{8pt}

\begin{center}
\includegraphics[width=.4\textwidth]{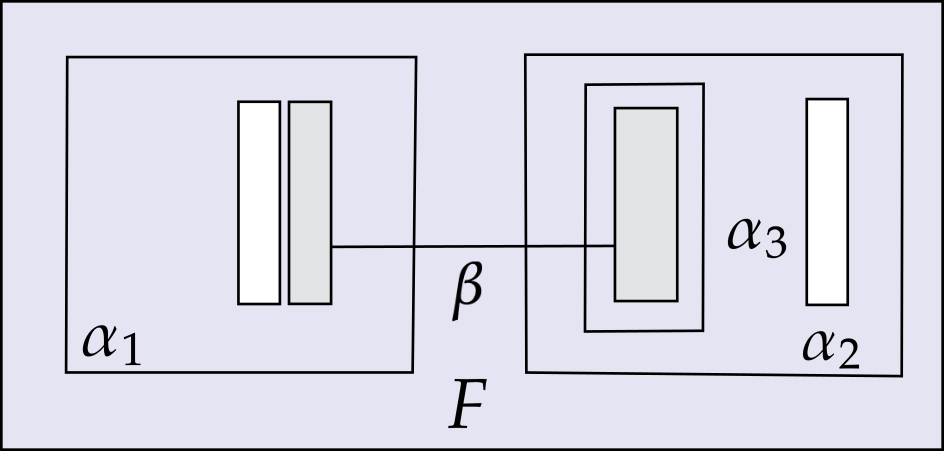}
\end{center}
Let $\gamma_i,\gamma_i^j$ and $\beta^j$ denote the curves in $F$ given by Dehn twisting $\beta$ as follows.
\[
\gamma_i = (\tau_{\alpha_1} \circ \tau_{\alpha_2} \circ \tau_{\alpha_3})^i(\beta) \qquad \gamma^j_i = \tau_{\alpha_1}^j(\gamma_i) \qquad \beta^j = \tau_{\alpha_1}^j(\beta)
\]
Starting with the simplest case, the Stein nucleus $N^0_0$ is given by the Lefschetz diagram corresponding to the zero tuple $m = (0,0,0)$ and $i = 0$ is given by
\[(F;\gamma_1,\beta,\gamma_{-1},\gamma_{-1},\alpha_1,\alpha_2)\]
This diagram may be drawn as:

\vspace{8pt}
\begin{center}
\includegraphics[width=.7\textwidth]{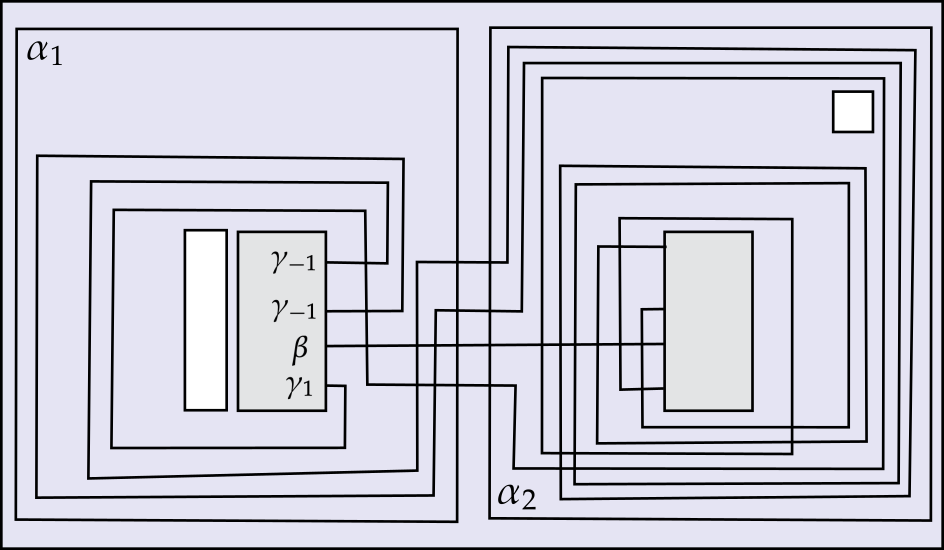}
\end{center}
The Lefshetz diagram for $N^0_i$ is acquired by Dehn twisting the curves $\gamma_1,\beta,\gamma_{-1}$ around $\alpha_1$ $i$-times. This yields the diagram
\[(F;\gamma_1^i,\beta^i,\gamma_{-1}^i,\gamma_{-1},\alpha_1,\alpha_2)\]
For purposes of depicting this diagram, it is easiest to apply the iterated Dehn twist $\tau_{\alpha_1}^{-i}$ to $F$. Since $\alpha_1$ is disjoint from $\alpha_2$, this yields the isomorphic diagram
\[(F;\gamma_1,\beta,\gamma_{-1},\gamma_{-1}^{-i},\alpha_1,\alpha_2)\]
This diagram may be drawn as follows:

\vspace{8pt}

\begin{center}
\includegraphics[width=.7\textwidth]{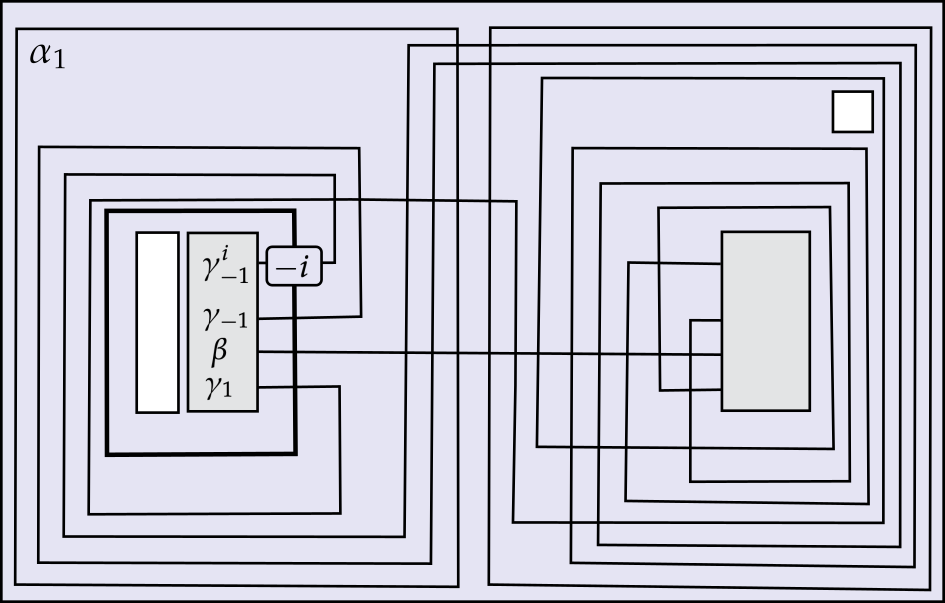}
\end{center}

\vspace{8pt}

\noindent Here, the bold circle and boxed $-i$ denotes the fact that the curve $\gamma^i_{-1}$ is acquired from $\gamma_{-1}$ by doing $-i$ Dehn twists around $\alpha_1$. This produces $|i|$-intersections with $\gamma_{-1},\beta$ and $\gamma_1$.

To extend create similar Lefschetz diagrams for $N^m_i$, we must perform \emph{$R$-modifications} on some curves in the diagram (in the language of \cite{yasui2014partial}). This involves attaching $1$-handles to the boundary of $F$ and sliding certain curves over a core circles of the $1$-handle.

More precisely, given a tuple $m = (m_1,m_2,m_3)$, let $F_m$ denote the surface acquired by attaching $m_1 + m_2 + m_3$ total $1$-handles to $F$ in an oriented way. We divide these handles into three groups as
\[
h_{0,1},\dots,h_{0,m_0} \qquad h_{1,1},\dots,h_{1,m_1} \quad\text{and}\quad h_{2,1}, \dots ,h_{2,m_2}
\]
We also let $E_{i,j}$ denote a core circle of $h_{i,j}$ (i.e.~a circle intersecting a belt arc of $h_{i,j}$). We let $\gamma_{-1}(m_0)$ be the curve acquired by handlesliding $\gamma_{-1}$ over all the core circles $E_{0,i}$, $\alpha_1(m_1)$ be the curve acquired by handlesliding $\alpha_1$ over all the core circles $E_{1,i}$, and $\alpha_2(m_2)$ be the curve acquired by handlesliding $\alpha_2$ over all the core circles $E_{2,i}$. The resulting picture of $F_m$ is given in Figure \ref{fig:stein_cork_Lef_N_m_i}.

\begin{figure} 
\includegraphics[width=\textwidth]{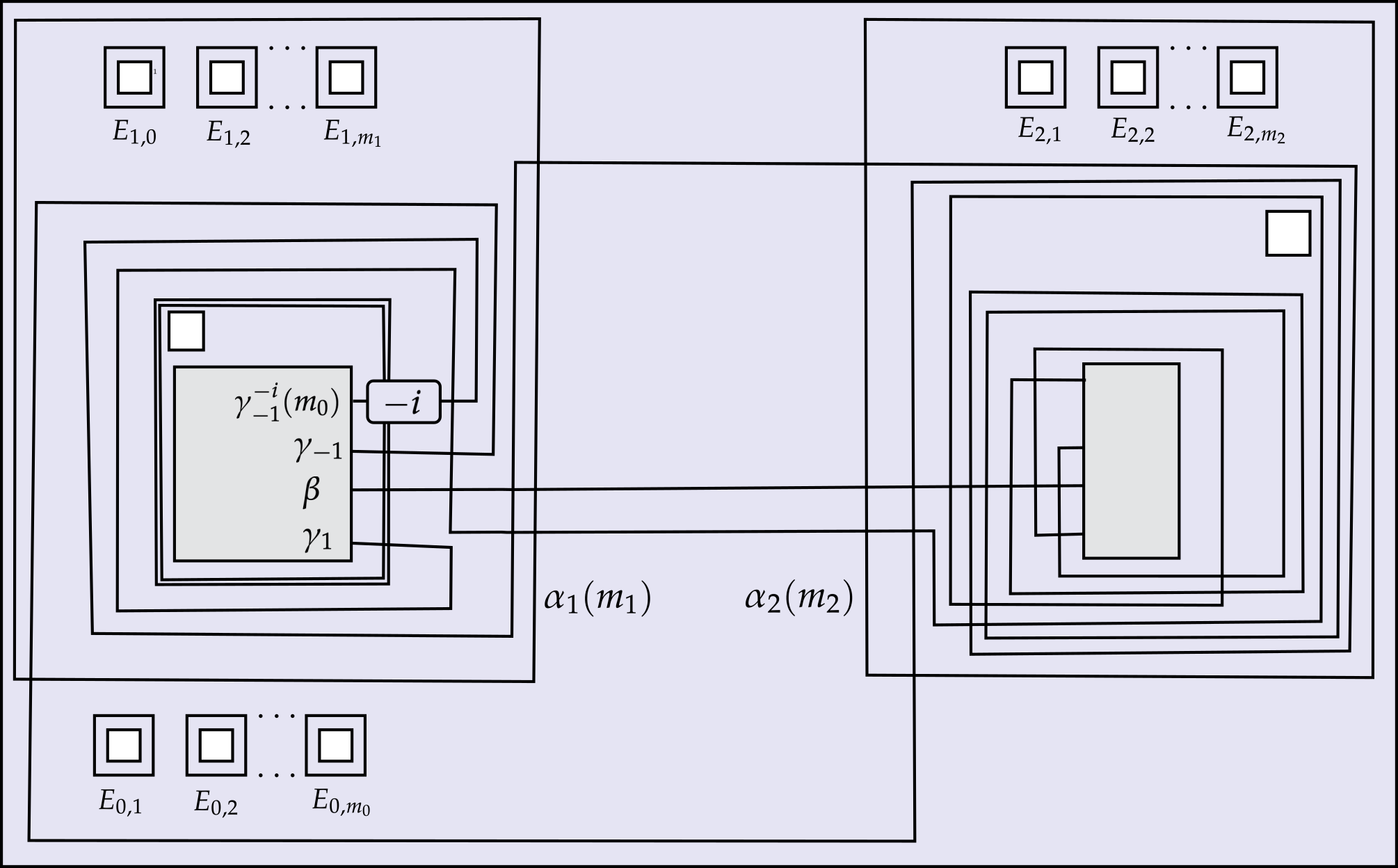}
\caption{A Lefschetz diagram for a general Stein nucleus $N^m_i$.}\label{fig:stein_cork_Lef_N_m_i}
\end{figure}

\begin{prop} \cite[\S 7]{yasui2014partial} The Stein nucleus $N^m_i$ is given by the following Lefschetz diagram $(F_m,L_m^i)$ where $L^i_m$ is the sequence of $6 + 6(m_1+m_2+m_3)$ curves
\[L_m^i = (\gamma_1,\beta,\gamma_{-1},E_{0,1},\dots,E_{0,m_0},\gamma_{-1}^{-i}(m_0),E_{1,i},\dots,E_{1,m_1},\alpha_1(m_1),E_{2,1},\dots,E_{2,m_2},\alpha_2(m_2))\]
\end{prop}

The Lefschetz diagram $(F_m,L^i_m)$ determined an open book $\pi^i_m$ on $\partial N^m_i$. This open book is independent of $i$, in the following sense. 

\begin{lemma} \cite[Prop.~3.6(4)]{yasui2014partial} There is a diffeomorphism $\partial N^m_i \simeq \partial N^m_j$ intertwining $\pi^i_m$ and $\pi_m^j$. \end{lemma}

\noindent In light of Theorem \ref{thm:stein_cork_top}(c), the above implies the following corollary. 

\begin{cor} For any $i$ and $j$, there is a homeomorphism $N^m_{2i} \simeq N^m_{2j}$ restricting to a diffeomorphism $\partial N^m_{2i} \simeq \partial N^m_{2j}$ intertwining the open books $\pi^m_{2i}$ and $\pi^m_{2j}$.  However, if $m_1 \ge 1$ then for infinitely many $i$ and $j$, there is no such diffeomorphism. \end{cor}

\noindent Thus the Stein nuclei provide infinitely many examples of the type of relative exotic phenomena that may be accessible with our trisection invariants.

\subsection{Trisecting Stein nuclei} We are now ready to trisect the Stein nuclei, using Algorithm \ref{alg:Lef_to_Tri} and some simplification. To start, we restrict to the nuclei $N^m_i$ where
\[m = (0,1,0)\]
This is the simplest family where exotic behavior is present. This is for simplicity, and adapting this calculation to the other Stein nuclei is straightforward.

We will start with the following version of the Lefschetz diagram $(F_m,L^i_m)$ for $N^{(0,1,0)}_i$: 

\vspace{8pt}

\begin{center}
\includegraphics[width=.6\textwidth]{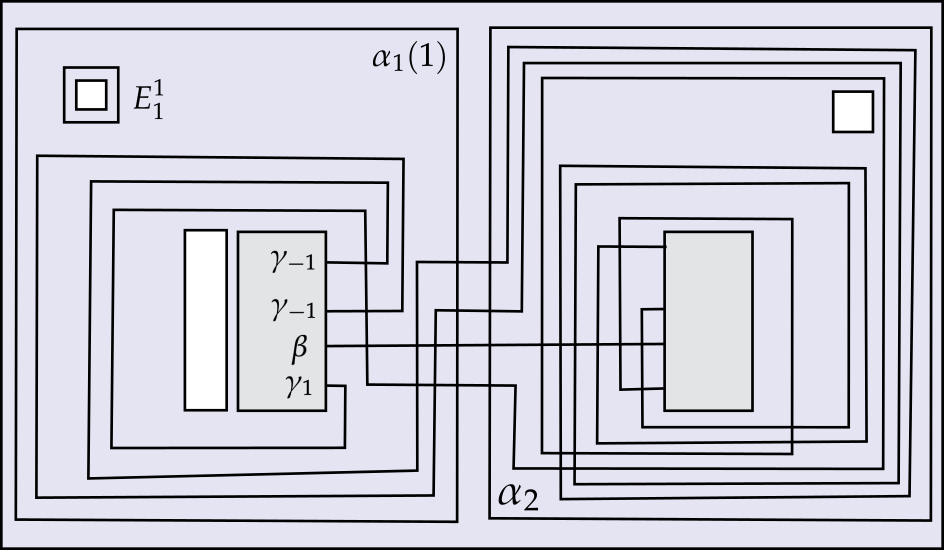}
\end{center}

\vspace{8pt}

By direct application of Algorithm \ref{alg:Lef_to_Tri}, we acquire the trisection diagram in Figure \ref{fig:trisection_of_Nmi_1} below. Note that we have omitted the label $-i$ in the twisted red curve corresponding to $\gamma^{-i}_{-1}$. After a series of handle slides and isotopies (included at the end of this paper) we arrive at the relatively simple diagram in Figure \ref{fig:trisection_of_Nmi_final} below.

\vspace{8pt}

\begin{figure} 
\begin{center}
\includegraphics[width=.8\textwidth]{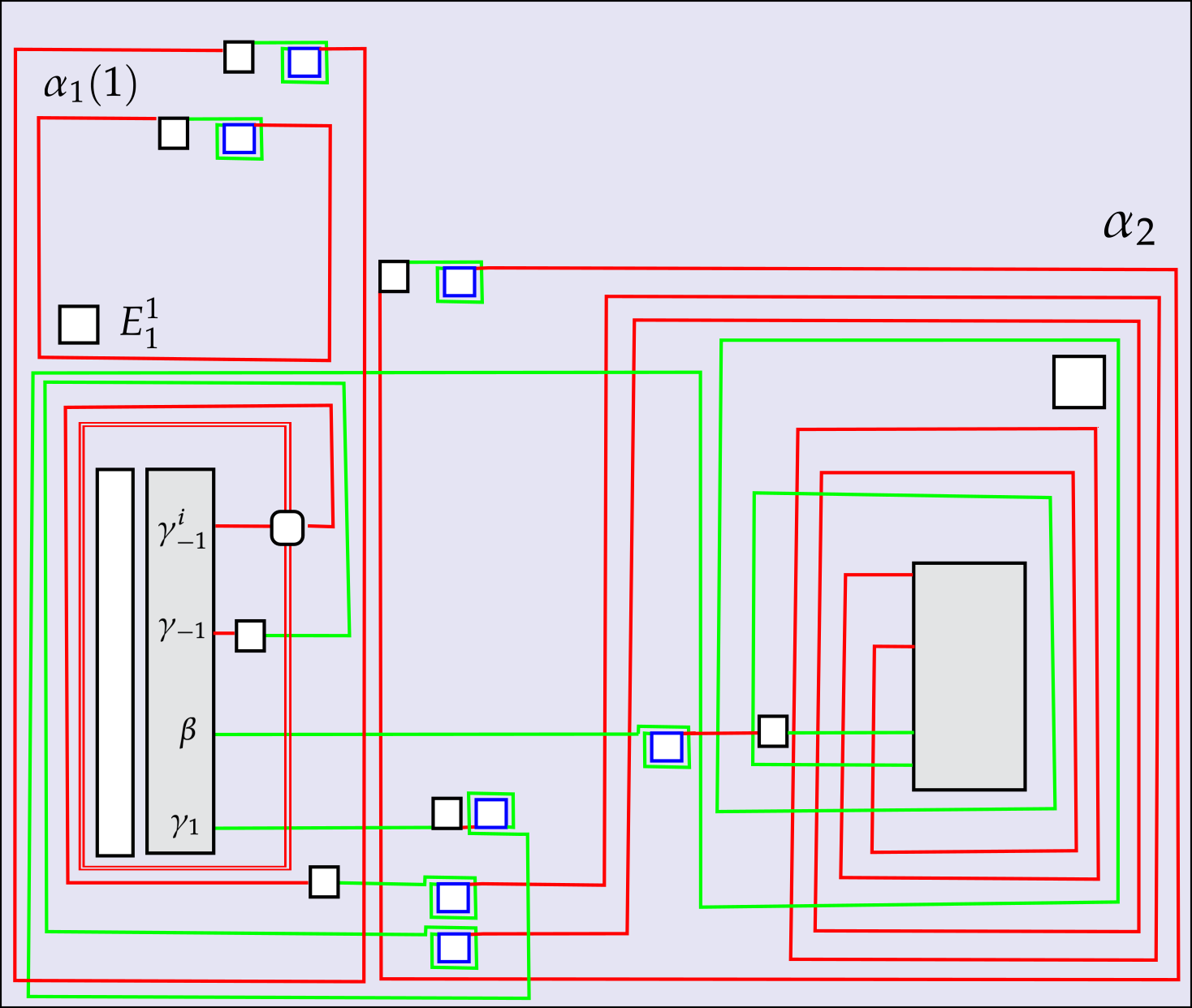}
\end{center}
\caption{The initial trisection of the Stein nucleus $N^{(0,1,0)}_i$ acquired via Algorithm \ref{alg:Lef_to_Tri}.}\label{fig:trisection_of_Nmi_1}
\end{figure}

\begin{figure} 
\begin{center}
\includegraphics[width=1\textwidth]{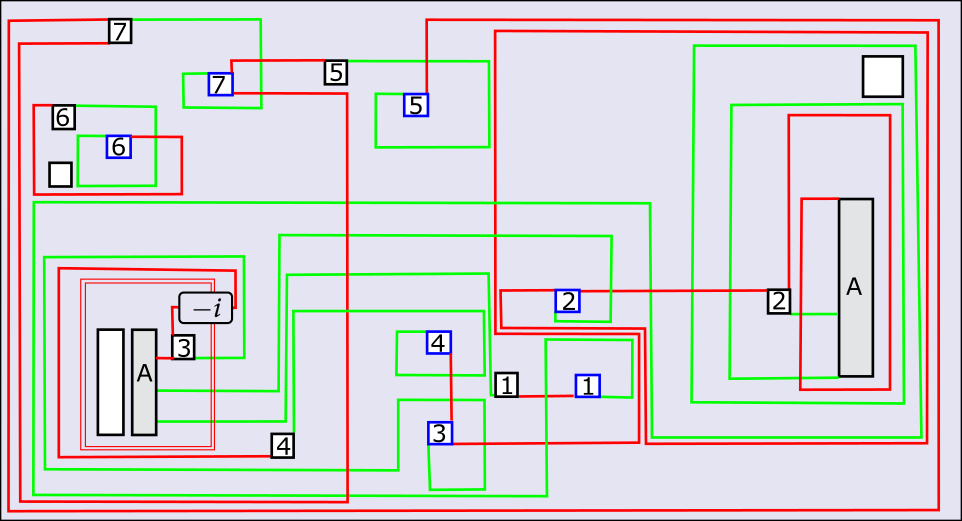}
\end{center}
\caption{The simplified trisection of the cork $N^{(0,1,0)}_i$.}\label{fig:trisection_of_Nmi_final}
\end{figure}

This trisection calculation generalizes in a straightforward way to all Stein nuclei. In particular, we arrive at Figure \ref{fig:trisection_of_Nmi_general_final} by applying Algorithm \ref{alg:Lef_to_Tri} to Figure \ref{fig:stein_cork_Lef_N_m_i}. Here we use $S(n)$ to denote the sub-diagram depicted in Figure \ref{fig:trisection_of_Nmi_component}.

\begin{figure} 
\begin{center}
\includegraphics[width=.8\textwidth]{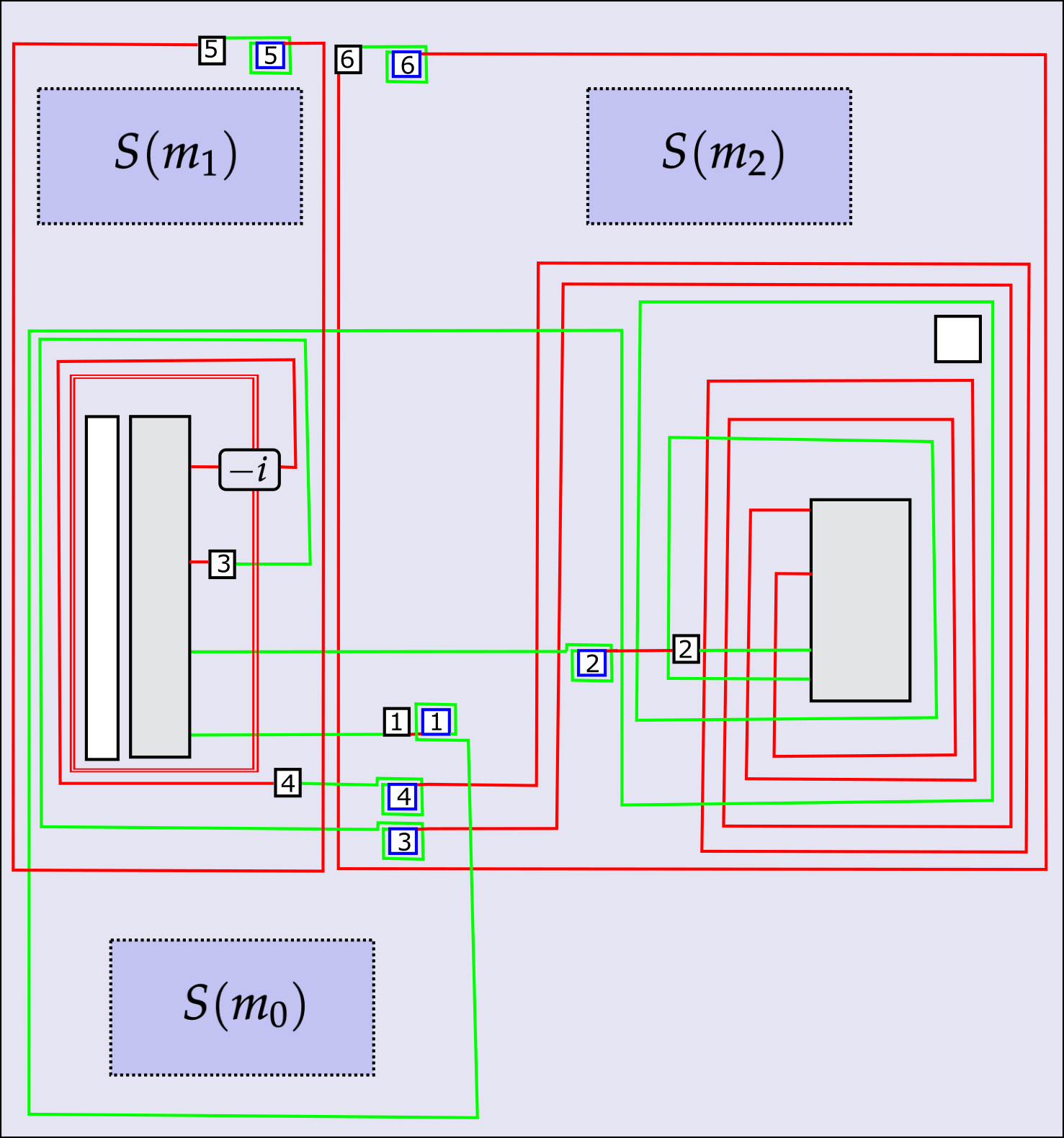}
\end{center}
\caption{The trisection of a general Stein nucleus $N^{m}_i$ acquired via Algorithm \ref{alg:Lef_to_Tri}.}\label{fig:trisection_of_Nmi_general_final}
\end{figure}

\begin{figure}
\begin{center}
\includegraphics[width=.8\textwidth]{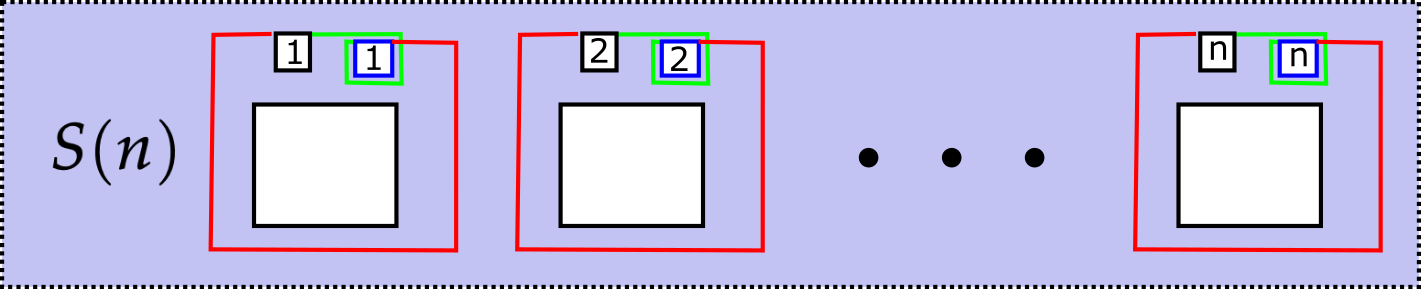}
\end{center}
\caption{The sub-diagram $S(n)$, consisting of $n$ copies of the standard Dehn twist trisection diagram around a curve parallel to a boundary component.} \label{fig:trisection_of_Nmi_component}
\end{figure}

\vspace{8pt}

\subsection{Trisection invariants of Stein nuclei} We can now discuss the trisection invariants of Stein nuclei in the involutory (i.e.~semisimple) case. 

\vspace{3pt}

In this setting, the invariant is computed with uncombed diagrams, and it satisfies a multiplicative property with respect to disjoint (non-intersecting) groups of curves in the trisection diagram. Thus we can derive the following proposition.

\begin{prop} \label{prop:ss_invariant_of_stein_nuclei} Let $\mathcal{H}$ be any involutory Hopf triple and let $N^m_i$ be the Stein nucleus for non-negative integers $m = (m_0,m_1,m_2)$ and an integer $i$. Let
\[
(N,\pi) = (N^0_0,\pi^0_0)
\]
Then the trisection invariant of $(N^m_i,\pi^m_i)$ is given by
\[\tau_{\mathcal{H}}(N^m_i,\pi^m_i) = \tau_{\mathcal{H}}(N,\pi) \cdot \tau_{\mathcal{H}}(\C P^2)^{m_0 + m_1 + m_2}\]\end{prop}

\begin{proof} We examine the trisection in Figure \ref{fig:trisection_of_Nmi_general_final}. Since the Hopf triple under consideration is involutory, we may compute the invariant without any choice of combings.

The regions $S(m_i)$ for $i = 0,1,2$ have tensor diagrams that are identical to the tensor diagram of $\C P^2$. Therefore, they contribute a factor of $\tau_{\mathcal{H}}(\C P^2)^{m_0 + m_1 + m_2}$. The complement of these regions has a tensor diagram identical to that of $N^0_i$. Thus we have
\[\tau_{\mathcal{H}}(N^m_i,\pi^m_i) = \tau_{\mathcal{H}}(N^0_i,\pi^0_i) \cdot \tau_{\mathcal{H}}(\C P^2)^{m_0 + m_1 + m_2}\]
and it suffices to show that $\tau_{\mathcal{H}}(N^0_i,\pi^0_i)$ is independent of $i$. By \cite[Rmk 3.9]{yasui2014partial}, each of the Lefschetz diagrams for $N^0_i$ induce isomorphic Lefschetz fibrations. Thus there is a diffeomorphism
\[
N^0_i \simeq N
\]
intertwining the boundary open books. The result follows. \end{proof}

\begin{remark} Due to Wall's principle, we know that any invariant of 4-manifolds that satisfies a multiplicative connect sum property must vanish on $\C P^2$ and $S^2 \times S^2$, if it is able to detect smooth structures. Proposition \ref{prop:ss_invariant_of_stein_nuclei} thus states that no such trisection invariant can detect the difference between Stein nuclei. Proposition \ref{prop:ss_invariant_of_stein_nuclei} also reflects the more general fact that semisimple TQFTs cannot detect smooth structures (cf. \cite{reutter2022semisimple}). \end{remark}

For a specific example, recall that our invariant yields a conjectural extension of the Kashaev invariants in \cite{kashaev2014asimple} to the case of $4$-manifolds with boundary, equipped with a boundary open book. Kashaev's invariants seem to correspond (up to a multiplicative factor) to the trisection invariant associated to a certain famil of Hopf triples $\mathcal{Z}(p)$ associated to cyclic group algebras (see \cite{cui2019four}) $\mathcal{Z}(p)$ for $p \in \N$.  Kashaev's invariant vanishes on $\C P^2$ in the case of $p = 2$. We thus acquire the following corollary.

\begin{cor} \label{cor:mod2_Kashaev_Stein_nuclei} The trisection invariant of $N^m_i$ with respect to the Hopf triple $\mathcal{Z}_2$ is $0$ if $m_1 + m_2 + m_3 > 0$.
\end{cor}

\begin{remark}[Non-semisimple case] In the general non-semisimple case, the trisection invariant of the Stein nuclei does not necessarily need to vanish if the trisection invariant of $\C P^2$ vanishes.

\vspace{3pt}

The reasoning is as follows. The trisection invariants satisfy a boundary sum property along the marked component of the open book's binding. However, they do \emph{not} satisfy a more general multiplicative property with respect to arbitrary sums or more generally arbitrary unions of disjoint regions of a diagram, as in the involutory case. This is because the rotation number of the trisection combing along any curve contributes non-trivially to the invariant. These rotation numbers can depend strongly on the isotopy class of the singular combing. In particular, the restriction of a combing of the trisection diagrams in Figure \ref{fig:trisection_of_Nmi_general_final} will not necessarily restrict in the regions $S(n)$ to a connect sum of combings of the diagram for $\C P^2$. Thus, the contributions of these regions may differ from the contribution of a $\C P^2$ connect summand.

\vspace{3pt}





\end{remark}

\subsection{Takahashi Manifolds} We conclude this section by discussing the small exotic pairs $P$ and $Q$ of trisections introduced by Takahashi in \cite{takahashi2023exotic} given by the pair of $(4,3;0,4)$-trisections in Figure \ref{fig:td-PQ-cutarc}. We let $\pi_P$ and $\pi_Q$ be the open books induced on $\partial P$ and $\partial Q$ by the trisections in Figure \ref{fig:td-PQ-cutarc}. The pair $P$ and $Q$ are exotic (homeomorphic and not diffeomorphic), and related by a twist along the Akbulut cork \cite{takahashi2023exotic}. Takahashi's examples have several similarities to Stein nuclei. In particular, we have the following result.

\begin{figure}[!htbp]
\centering
\includegraphics[scale=1]{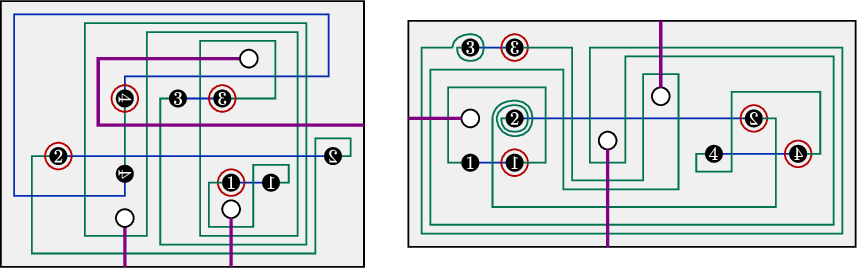}
\caption{Takahashi's manifolds $P$ (left) and $Q$ (right) from \cite{takahashi2023exotic}. The purple arcs in this figure are an auxilliary cut system, and can be ignore for our purposes.}
\label{fig:td-PQ-cutarc}
\end{figure}

\begin{lemma} Let $\mathcal{H}$ be any involutory Hopf triple. Then there are scalars $I_P$ and $I_Q$ with
\[
\tau_{\mathcal{H}}(P,\pi_P) = \tau_{\mathcal{H}}(\C P^2) \cdot I_P \qquad \text{and} \qquad \tau_{\mathcal{H}}(Q,\pi_Q) = \tau_{\mathcal{H}}(\C P^2) \cdot I_Q\qquad 
\]
Thus $\tau_{\mathcal{H}}(P,\pi_P)$ and $\tau_{\mathcal{H}}(Q,\pi_Q)$ vanish if $\tau_{\mathcal{H}}(\C P^2)$ vanishes. \end{lemma}

\begin{proof} As in the proof of Proposition \ref{prop:ss_invariant_of_stein_nuclei}, it suffices to find a sub-diagram of the trisections for $P$ that is equivalent to the standard $\C P^2$ diagram and disjoint from the rest of the trisection for $P$. In the diagram of $P$ in Figure \ref{fig:td-PQ-cutarc} contains a sub-diagram $D_P$ (consisting of the red, blue and green curves near the $1$-handle labelled $1$ in the diagram for $P$) that is equivalent to the diagram for $\C P^2$. 

Similarly, $Q$ has a sub-diagram $D_Q$ consisting of the blue and green curves on the $1$-handle labelled $3$, and the red curve near that $1$-handle. This sub-diagram intersects the blue curve over the $1$-handle labelled $2$. However, $D_Q$ can be disjoined from this green curve by three two point moves. These moves cross some boundary components of the trisection surface, but this is not important for the argument (as in Proposition \ref{prop:ss_invariant_of_stein_nuclei}). \end{proof}

In particular, this implies that the mod 2 Kashaev-type invariant vanishes for these spaces.

\begin{cor} \label{cor:mod2_Kashaev_takahashi}The trisection invariant of $P$ and $Q$ with respect to the Hopf triple $\mathcal{Z}_2$ is $0$.
\end{cor}

\bibliographystyle{plain}
\bibliography{trisection_inv_bib}

\end{document}